\documentclass[11pt]{article}
\usepackage{amssymb,amsmath,amsthm}
\usepackage{upgreek}
\usepackage{yfonts}
\usepackage{mathrsfs}
\usepackage{stmaryrd}
\newtheorem{remark}{Remark}
\newtheorem{definition}{Definition}

\newtheorem{proposition}{Proposition}
\newtheorem{lemma}{Lemma}
\newtheorem{corollary}{Corollary}
\newtheorem{theorem}{Theorem}

\newcommand{\PGL}{(\text{PGL})_\eps}

\def\bd{\partial}

\def\rest{\hskip 1pt{\hbox to 10.8pt{\hfill
\vrule height 7pt width 0.4pt depth 0pt\hbox{\vrule height 0.4pt width 7.6pt depth 0pt}\hfill}}}

\newcommand{\eps}{\varepsilon}
\newcommand{\E}{{E}_\varepsilon}

\def\beq{\begin{equation}}
\def\eeq{\end{equation}}

\def\hzero{(\text{H}_0)}

\def\B{{\mathcal B}}

\def\R{{\mathbb R}}
\def\N{{\mathbb N}}
\def\S{{\mathbb S}}

\def \Le{\mathcal L^\eps}
\def \L{\mathcal L} 

\def \rL{{\rm L}}
\def \tL{{\scriptscriptstyle \rm L}}
\def \IL{{{\rm I}_{{\scriptscriptstyle \rm L}}}}
\def \ILL{{{\rm I}_{\scriptscriptstyle \rm 2L}}}
\def \ILLL{{{\rm I}_{\scriptscriptstyle \rm 3L}}}
\def \ILLLL {{{\rm I}_{\scriptscriptstyle \rm 4L}}}
\def \ILLLLL {{{\rm I}_{\scriptscriptstyle \rm 5L}}}
\def \ItdL {{{\rm I}_{\scriptscriptstyle \frac32 \rm L}}}
\def \IctL {{{\rm I}_{\scriptscriptstyle \frac53 \rm L}}}
\def \IkzL {{\rm I}_{\scriptscriptstyle \rm \upkappa_0 \rL}}
\def \IdkzL {{\rm I}_{\scriptscriptstyle \rm 2\upkappa_0 \rL}}

\def \q {\mathfrak q}

\def \re {r^\eps}
\def \reext  {r^\eps_{\rm ext}}
\def \reint  {r^\eps_{\rm int}}

\def\ud{\updelta}
\def\udl {\updelta_{\rm log}^{^\eps}} 
\def \udll  {\updelta_{\rm loglog}^{^\eps}} 

\def \v  {{v_\eps}}
\def \dminep {d_{\rm min}^{\eps, +}}

\def \dmines {d_{\rm min}^{\eps, -}}
\def \dmine {d_{\rm min}^\eps}
\def \dmineL {d_{\scriptscriptstyle \rm min}^{\scriptscriptstyle \eps, \rL}}
\def \dmin {d_{\rm min}}
\def \dminstar {d_{\rm min}^*}

\def \Iinte{\mathcal J_\eps^{\rm int}}
\def \Iexte{\mathcal J_\eps^{\rm ext}}
\def\vm{\mathfrak{v}_\eps}
\def\vmen{\mathfrak{v}_{\eps_n}}

\def\Wm{\mathfrak{W}_\eps}
\def\Wmn{\mathfrak{W}_{\eps_n}}
\def\Wms{\mathfrak{W}_*}
\def \Ump {\overset {\vee}{\mathfrak u}^+}
\def\Umps {\overset {\rhd} {\mathfrak u}}
\def \Urp {\overset {\vee}{{\rm  U}_r^\lambda} }
\def\Urps {\overset {\rhd} {{\rm  U}_r^\lambda} }
\def \Umpep {\overset {\vee} {\mathfrak u}^+_{\eps, r} }
\def \Umpes {\overset {\vee} {\mathfrak u}^-_{\eps, r}}
\def \Umpeps {\overset {\rhd} {\mathfrak u}_{\eps, r}}
\def \Umpepstt {\overset {\lhd} {\mathfrak u}_{\eps, r^\eps(s)}}

\def \Umpepst  {\overset {\rhd} {\mathfrak u}_{\eps, r^\eps(s)}}
\def \Umpepint {\overset {\vee} {\mathfrak u}^+_{\eps, \reint} }
\def \Umpesint {\overset {\vee} {\mathfrak u}^-_{\eps, \reint}}

\def \Umpepext {\overset {\vee} {\mathfrak u}^+_{\eps, \reext}}

\def \tetakeps {\Theta_{k+\frac12}^\eps}

\def \llogepsL { \log \vert \log \frac{\eps}{\rL} \vert}

\def \dis {{\rm dissip}_\eps }
\def \disL   {{\rm dissip}_\eps^{\scriptscriptstyle \rL} \, }

\def \Ldint{\Lambda_\eps^{\rm int}}
\def \Ldext {\Lambda_\eps^{\rm ext}}
\def \Ldtrans  {\Lambda_\eps^{\rm trans}}

\def \Itrans  {\mathcal I_\eps^{\rm trs}}

\def\S { {\mathfrak S}}
\def\SL{ {\mathfrak S}^\rL}
\def\E{{ \mathfrak q} }

\def \tda{ { \textswab d}_{\, a}}

\def\tU {{\textfrak U}}

\def\tu {{\textswab  u}}

\def\rtcol {\mathcal T^{\eps, -}_{\rm col}}

\def \rtcolp{\mathcal T_{\rm col}^{\eps, +}}
\def \rtcros{\mathcal T_{\rm cros}^\eps}

\def \wp{\mathcal{WP}_\eps^{\tL}}
\def \wpz{\mathcal{WP}_\eps^{\tL_0}}
\def \wpi{\mathrm{WPI}_\eps^{\tL}}
\def \wpo{\mathrm{WPO}_\eps^{\tL}}

\def \WP{\mathcal{W P }_\eps }
\def \WPL{\mathcal{W P }_\eps^{\text{\tiny {\rm L}}} }

\def \WPI{{\rm W PI}_\eps^{\text{\tiny {\rm  L}\,}  }}

\def \WPOL{{\rm W PO}_\eps^{\text{\tiny{\rm L}\,}  }}
\def \Frep{ F_{\rm rep} }
\def \tFrep{ {\tilde F}_{\rm rep} }
\def \Fat{ F_{\rm att} }
\def \tFat{\tilde{ F}_{\rm att} }
\def \mFrep{\mathcal  F_{\rm rep} }
\def \mFat{ \mathcal F_{\rm att}^\eps }

\def \mFrep{\mathcal  F_{\rm rep}^\eps}
\def \mFatn{ \mathcal F_{\rm att}^n }
\def \tmFatn{\tilde{\mathcal  F}_{\rm att}^n }
\def \matfE {\mathfrak E_\eps}
\def \matfEL {\mathfrak E_\eps^\rL}
\DeclareMathAlphabet{\mathpzc}{OT1}{pzc}{m}{it}

\def\bd{\partial}
\def\QED{\hbox{${\vcenter{\vbox{
   \hrule height 0.4pt\hbox{\vrule width 0.4pt height 6pt
   \kern5pt\vrule width 0.4pt}\hrule height 0.4pt}}}$}\vspace{7pt}}

\begin{document}

\author{Fabrice BETHUEL\thanks{Sorbonne Universit\'es, UPMC Univ Paris 06, UMR 7598, Laboratoire Jacques-Louis Lions, F-75005, Paris, France}  \, and  Didier SMETS \thanks{Sorbonne Universit\'es, UPMC Univ Paris 06, UMR 7598, Laboratoire Jacques-Louis Lions, F-75005, Paris, France} 
}
\title{The  motion law of fronts for scalar reaction-diffusion equations  
with multiple wells: 
 the degenerate case}
\date{}
\maketitle

\begin{abstract}
We  derive a precise motion law for   fronts of solutions  to  \emph{scalar} one-dimensional
reaction-diffusion equations with equal depth multiple-wells, in the case the second derivative of the potential 
vanishes at its minimizers. We show that,   \emph{renormalizing time}  in an \emph{algebraic} way,  
the motion of fronts is governed by a simple system of ordinary differential equations of  nearest 
neighbor interaction type. These interactions may be either  attractive or repulsive. Our results are 
not constrained by the possible occurrence of collisions nor splittings. They  present substantial differences 
with the results  obtained in  the case the second derivative does not vanish at the wells,    a case which 
has been extensively studied in the literature,   and where fronts have been showed to move at  exponentially 
small speed, with motion laws which are   \emph{not renormalizable}. 
\end{abstract}

\bigskip
\noindent
\section{Introduction}

This paper is a continuation of our previous works \cite{BOS8, BS} where we analyzed the behavior of 
solutions $v$ of the reaction-diffusion equation of gradient type
\begin{equation*}
\label{glpara}
\frac{\partial v_\eps} {\partial t}-\frac{ \partial^2  v_\eps} {\partial x^2}= 
- \frac{1}{\eps^2}\nabla V(v_\eps),  \leqno{(\text{PGL})_\eps}
\end{equation*}
where $0<\eps<1$ is a small parameter.   In \cite{BS},  we considered  the case
where the potential $V$ is a  smooth map from $\R$ to $\R^k$ with multiple wells
of equal depth whose second derivative vanishes at the wells.   The main result there, stated
in Theorem \ref{mainbs2} here,  provides an upper bound for the speed of fronts.
In the present paper we \emph{restrict ourselves to the scalar case},  $k=1$, 
  and provide a precise motion law for the fronts, showing in particular that
the \emph{upper bound provided in} \cite{BS}\emph{ is sharp}.  We assume
throughout this paper that the potential $V$ is a smooth function from
$\R$ to $\R$  which satisfies the following assumptions: 
$$ 
\inf V=0 \hbox{ and the set of minimizers}  \  \Sigma\equiv \{ y \in \R,
V(y)=0 \}\   {\rm \ is  \ finite, } \leqno{(\text{A}_1)}
$$
with at least two distinct elements, that is
$$ 
\Sigma=\{\upsigma_1, ..., \upsigma_\q\},\ \q\geq 2, \ \upsigma_1 < \cdots < \upsigma_\q.
$$ 
\smallskip
\noindent
${(\text{A}_2)}$  There exists a number $\theta >1$ such that  
for all $i$ in $\{1,\cdots, \q\}$, we have
$$ V(u)=\lambda_i(u-\upsigma_i)^{2 \theta}+ \underset {u \to \upsigma_i} o ((u-\upsigma_i)^{2 \theta}),  
{\rm where }  \ \lambda_i>0.  
$$

\smallskip
\noindent
${(\text{A}_3)}$ There exists  constants  $\upalpha_{\infty}>0$ and  $R_\infty>0$  such that  
$$
u\cdot\nabla V( u )\geq \upalpha_{\infty} \vert u \vert ^2, \ \hbox {if }  \vert u
\vert >R_\infty.
$$
Whereas  assumption $({\rm A}_1)$ expresses the fact that the potential
possesses at least two  minimizers,  also termed wells,  and  $({\rm A}_3)$
describes the behavior at infinity, and is  of a more technical  
nature,  assumption $({\rm A}_2)$,  which is\emph{ central in the present
paper},  describes the local behavior near  the minimizing wells. 
The number $\theta$ is of course related to the order of vanishing of the derivatives near zero. 
Since $\theta>1$, then $V''(\upsigma_i)=0$, and  $({\rm A}_2)$
holds if and only if
\begin{equation*}
\label{johnson}
\frac{d^j}{du^j} V(\upsigma_i)=0 \ \hbox { for } j=1, \cdots, 2\theta-1  \ \ {\rm  and  \ } 
\frac{d^{2\theta}}{du^{2\theta}} V(\upsigma_i)> 0, 
\end{equation*}
with 
\begin{equation*}
\label{powerball}
\lambda_i  = \frac{1}{(2 \theta) !}\frac{d^{2\theta}}{du^{2\theta}} V(\upsigma_i).
\end{equation*}
A typical example of such potentials is given by
$V(u)=(1-u^2)^{2 \theta}= (1-u)^{2\theta}(1+u)^{2\theta} $
which has two minimizers, $+1$ and $-1$,  so that  $\Sigma=\{+1, -1\}$,
minimizers vanishing at order $2\theta$.   In this paper,  the order of
degeneracy is an integer assumed to be the same at all wells: fractional or
site dependent orders may presumably be handled with the same tools, 
however at the cost of more complicated
statements.\\

\medskip  
We  recall that   equation  $(\text{PGL})_\eps$   corresponds to the $L^2$
gradient-flow    of the  energy functional  $\mathcal{E}$ which is defined for a
function $ u:\R \mapsto \R$ by the formula 
\begin{equation*}
\label{glfunctional}
\mathcal{E}_\eps (u)= \int_\R e_\eps(u)= \int_{\R} \frac{ \eps \vert \dot u \vert ^2}{2}+\frac{V(u)}{\eps}.
\end{equation*}
As in  \cite {BOS8,BS}, we consider only \emph{finite energy solutions}. More
precisely, we fix an arbitrary constant  $M_0>0$ and we consider the condition
$$
\mathcal E_\eps(u)\leq M_0<+\infty.  \leqno\text{$(\text{H}_0)$}
$$
Besides the assumptions on the potential, the main assumption is on the
initial data $v_\eps^0(\cdot)=v_\eps(\cdot, 0)$, assumed to satisfy
$(\text{H}_0)$ independently of $\eps.$ 
In particular, in view of the classical energy identity
\begin{equation}
\label{energyidentity}
\mathcal E_\eps (v_\eps(\cdot,T_2))+\eps \int_{T_1}^{T_2} \int_{\R} \left|
{\frac{\bd v_\eps}{\bd t}}
\right|^2(x,t)dx\,dt = \mathcal E_\eps(v_\eps(\cdot,T_1)) \, \quad \forall\, 0\le
T_1\le T_2\, , 
\end{equation}
we have
\begin{equation*}
\label{touttemp}
\mathcal E_\eps \left(v_\eps(\cdot, t)\right )\leq  M_0, \ \forall t \geq 0.
\end{equation*}
This implies in particular that for every  given $t\geq 0$, we have
$V(v_\eps(x, t))\rightarrow 0$ as $\vert x\vert \rightarrow \infty$. It is then
quite straightforward to deduce from assumption $\hzero$, 
${(\text{A}_1)}$, ${(\text{A}_2)}$ as well as the energy identity
\eqref{energyidentity}, that $v_\eps(x, t)\rightarrow \upsigma_{\pm}$ as
$x\rightarrow \pm \infty$ , where $\upsigma_{\pm} \in \Sigma$ do not depend on
$t$. In other words, for any time, our solutions connect to given minimizers of
the potential. 

%%%%%%%%%%%%%%%%%%%%%%%%%%%%%%%%%%%%%%%%%%%%%%%%%%%%%%%%%%%%%%%%%%%%%%%%%%%%%%%%%%%%%%
\subsection{Main results: Fronts and their speed}
%%%%%%%%%%%%%%%%%%%%%%%%%%%%%%%%%%%%%%%%%%%%%%%%%%%%%%%%%%%%%%%%%%%%%%%%%%%%%%%%%%%%%%
The notion of fronts  is central in  the dynamics.
For a map $u:\R\mapsto \R$,  the set 
\begin{equation*}
\label{frontset}
\mathcal D (u)\equiv\{x\in \R,\ \hbox{dist}(u(x),\Sigma) \geq \mu_0\}, 
\end{equation*}
is termed throughout \emph{the front set} of $u$. The constant $\mu_0$ which
appears in its definition is fixed once for all, sufficiently small so that
\begin{equation}\label{eq:souscontrole}
\frac{\lambda_i}{2} (u-\upsigma_i)^{2\theta} \leq V(u) \leq 
\frac{1}{\theta}V'(u)(u-\upsigma_i) \leq 4 V(u) \leq 8\lambda_i
(u-\upsigma_i)^{2\theta},
\end{equation}
for each $i\in \{1,\cdots,\q\}$ and whenever $|u-\upsigma_i|\leq \mu_0.$
The front set corresponds to
the set of points  where $u$ is ``far'' from the minimizers $\upsigma_i$, and hence
where transitions from one minimizer to the other may occur. A straightforward
analysis yields  
     
\begin{lemma}[see e.g. \cite{BOS8}] 
\label{interface} 
Assume  that $u$ verifies $\hzero$. Then there exists $\ell$ points
$x_1,...,x_\ell $ in $\mathcal D(u)$ such that 
\begin{equation*}
\mathcal D(u) \subset  \underset {k=1}{ \overset{ \ell}  \cup}  [x_k- \eps,x_k +\eps],
\end{equation*}
with a bound $\ell \leq \frac{M_0}{\eta_0}$ on the number of points,  $\eta_0$ being some constant depending only on  $V$. 
\end{lemma} 

In view of Lemma \ref{interface}, the measure of the front sets is of order
$\eps$, and corresponds to a small neighborhood of order $\eps$ of the points
$x_i$. Notice that if $(u_\eps)_{\eps>0}$ is a family of  functions satisfying 
$\hzero$ then it is well-known that the family is
locally bounded in $BV(\R,\R)$  and hence locally compact in $L^1(\R,\R)$.
Passing to a subsequence if necessary, we may assert that  
\begin{equation*}
\label{puyol0}
u_\eps \to  u^{\star}  \hbox { in } L^1_{\rm loc}(\R), 
\end{equation*}
where $u^{\star}$  takes values in $\Sigma$ and is a step function. More
precisely  there exist an integer  $\ell  \leq \frac{M_0}{\eta_0}$,
$\ell$ points $a_1 < \cdots < a_{\ell}$ and a function $\hat \imath \: : \:
\{\frac12,\cdots,\frac12+\ell\} \to \{1,\cdots,\q\}$ such that   
\begin{equation*}
\label{puyol}
u^{\star}= \upsigma_{\hat \imath(k+\frac 12)}   \hbox {    on  \  }  (a_k, a_{k+1}), 
\end{equation*}
for $k=0, \cdots, \ell,$ and where we use the convention $a_0:=-\infty$ and
$a_{\ell+1} := +\infty.$ The points $a_k$,  for $k=1\cdots, \ell$,   are
the limits as $\eps$ shrinks  to $0$  of the points $x_i$ provided by Lemma
\ref{interface} (the number and the positions of which are of course $\eps$
dependent), so that the front set  $\mathcal D (u_\eps)$ shrinks as
$\eps$ tends to $0$ to a finite set.  In the sequel, we shall refer to step
functions with values into $\Sigma$ as steep front chains and we will write
$$u^\star=u^\star(\ell,\hat \imath,\{a_k\})$$ to determine them unambiguously.    

\medskip
     
We go back to  equation  $(\text{PGL})_\eps$  and consider a family of functions
$(v_\eps)_{\eps>0}$  defined on $\R\times \R^+$ which are solutions to the
equation $(\text{PGL})_\eps$ and satisfy the energy bound $\hzero$.   
We set
\begin{equation*}
\mathcal D_\eps(t)=\mathcal D(v_\eps(\cdot, t)).
\end{equation*}
The evolution of the  front set  $\mathcal D_\eps (t)$  when $\eps$ tends to $0$
is the main focus of our paper.   The following  result\footnote{which holds
also more generally for systems.} has been proved in \cite{BS}:  
\begin{theorem}[\cite{BS}] 
\label{mainbs2}
There exists
constants $\uprho_0>0$  and $\upalpha_0>0$, depending only on the potential $V$
and on $M_0$ such that if $r\geq \upalpha_0\eps,$ then 
\begin{equation}
\label{theodebase}
\mathcal D_\eps  (t+ \Delta t) \subset \mathcal D_\eps (t) + [-r,r],\qquad \text{
 for every  } 
t\geq 0, 
\end{equation}
provided $0\leq \Delta t  \leq 
\uprho_0r^2\left(\frac{r}{\eps}\right)^{\frac{\theta+1}{\theta-1}}.$
\end{theorem}
As a matter of fact,  it follows from this result that the average
speed of the front set at that length-scale should not  exceed  
\begin{equation}
\label{speedy}
c_{\rm ave} \ \simeq \ \frac {r}{(\Delta t)_{\text{\rm \tiny max}
\!\!\!\!\!\!\!\!\!\!\!\!}} \quad \leq \  \uprho_0^{-1}  r ^{-(\omega+1)} \eps^{\omega}, 
\end{equation} 
where
\begin{equation}
\label{picomega}
\omega= \frac{ \theta+1}{\theta-1}.
\end{equation}
Notice that $1< \omega <+ \infty$ 
and that the upper bound provided  by \eqref{speedy}  decreases with $\theta$,
that is, the more degenerate the minimizers of $V$ are, the higher  the possible
speed allowed by the bound \eqref{speedy}. In contrast,  
the speed is at most exponentially small in the case of  non degenerate
potentials (see e.g. \cite{ carpego}, \cite{BOS8} and the references therein).
One aim of the present paper is to show that the  \emph{upper bound}
provided by the estimate \eqref{speedy} is in fact optimal\footnote{at least in
the scalar case considered here.} and actually to derive a precise motion law for
the fronts. An important fact, on which our results are built, is the following
observation\footnote{which to our knowledge has not been observed before, even
using formal arguments.}:

\begin{center}  
\emph { Equation  $(\text{PGL})_\eps$  is renormalizable}. 
\end{center}
This assertion means that, rescaling time in an appropriate way, the evolution of fronts in 
the asymptotic limit $\eps \to 0$  is governed by an ordinary differential
equation which \emph{does not involve the parameter $\eps$}.   More precisely,
we accelerate time by the  factor $\eps^{-\omega}$ and consider the new time
$s=\eps^{\omega} t$.  
In the accelerated time, we  consider the map
\begin{equation}
\label{mathfrakv}
\vm (x,s)=v_\eps(x,s\eps^{-\omega}), {\rm \ and \ set \ }
\mathfrak  D_\eps (s)=\mathcal D(\vm(\cdot, s)). 
\end{equation}
It follows from Theorem \ref{mainbs2} that for given $\displaystyle{r \geq
\upalpha_0\eps}$,
\begin{equation}
\label{theodebase1}
\mathfrak D_\eps  (s+ \Delta s) \subset \mathfrak D_\eps (s) + [-r,r], \qquad
\text{for every } s\geq 0, 
\end{equation} 
provided that $0\leq \Delta s  \leq  \uprho_0r^{\omega+2}.$

Concerning the initial data, we will assume 
that there exists a steep front chain $v^\star(\ell_0,\hat \imath_0,\{a_k^0\})$ 
%integer $\ell_0$,  $\ell_0$ points \linebreak $a_1^0 <
%\cdots < a_{\ell_0}^0$ in $\R$, and a function $\hat \i\: : \:
%\{\frac12,\cdots,\frac12+\ell_0\} \to \{1,\cdots,\q\}$ such that,   
such that   
$$
\left\{
\begin{aligned}
v_\eps^0&\longrightarrow v^\star(\ell_0,\hat \imath_0,\{a_k^0\})  \ \hbox{ in
} \, L^1_{\rm loc} (\R),\\  
\mathfrak D_\eps (0) &  \longrightarrow \{a_k^0\}_{1\le k \le \ell_0},  
\hbox{  locally in the sense of the Hausdorff distance  }, 
\end{aligned}
\right.
\leqno\text{$(\text{H}_{1})$}
$$
as $\eps \to 0.$
Let us emphasize  that \emph{assumption} $(\text{H}_1)$ \emph{is not
restrictive}, since it follows  assuming only the energy bound  $(\text{H}_0)$
and passing possibly to a subsequence (see above). 
In our first result, we will  impose the additional condition 
\begin{equation*}
\label{transition}
\vert {\hat \imath_0 (k+\frac 12)}-{\hat \imath_0(k-\frac 12)} \vert =1
{\rm \ for \ } 1 \leq k \leq \ell_0.
\leqno\text{$(\text{H}_{  \,  \rm min} )$}
\end{equation*}
This assumption could be rephrased as a ``multiplicity one" condition: it means
that the jumps consist of exactly one transition  between  consecutive minimizers
$\upsigma_i$ and $\upsigma_{i\pm1}$. To each transition point $a_k^0$ we may
assign a sign, denoted by  $\dagger_k\in \{+,-\}$, in the following way: 
\begin{equation*}
\dagger_k=+ \ {\rm if }  \  \upsigma_{\hat \imath_0(k+\frac 12)}=\upsigma_{\hat \imath_0(k-\frac 12)}+1  \ \ {\rm and  \ } 
\dagger_k=- \ {\rm if } \ \upsigma_{\hat \imath_0(k+\frac12)}=\upsigma_{\hat \imath_0(k-\frac1 2)}-1. 
\end{equation*}
We consider next the system of ordinary differential equations
\begin{equation}  
\label{tyrannosaure}\tag{$\mathcal{S}$}
\S_{k} \frac{d}{ds} {a}_{k}=  \frac{\Gamma_k^+}{\big(a_k-a_{k+ 1}\big)^{\omega+1}}
-\frac{\Gamma_k^-}{ \big(a_k-a_{k-1}\big)^{\omega+1}},
\end{equation}
for $1 \leq k \leq \ell_0,$ where $\displaystyle{ \S_{k}}$ stands for the energy of 
the corresponding stationary front, namely 
\begin{equation}
\label{quantenergy}
\S_{k} =
\int_{\upsigma_{\hat \imath_0(k-\frac12)}}^{\upsigma_{\hat \imath_0(k+\frac12)}}\sqrt{2V(u)} du, 
\end{equation}
and where we have set, for $k=1, \cdots, \ell_0$
\begin{equation}
\label{defgamma}
\left\{
\begin{aligned}
 &\Gamma_k^+ =\phantom{-}2^{\omega}\left(\lambda_{\hat \imath_0(k+\frac
12)}\right)^{-\frac{1}{\theta-1}} \mathcal A_{\theta} {\rm \ \ if  \ }  \dagger
_k=- \dagger_{k+1} \\ 
& \Gamma_k^- =-2^{\omega}\left(\lambda_{\hat \imath_0(k+\frac 12)}
\right)^{-\frac{1}{\theta-1}}\mathcal B_{\theta}  {\rm \ \ if  \ }  \dagger
_k=\phantom{-}\dagger_{k+1}. 
\end{aligned}
\right.
\end{equation}
In \eqref{defgamma},  $\lambda_{\hat \imath_0(k+\frac 12)}$ is defined in $({\rm
A}_2)$ and  the constants  $ \mathcal  A_{\theta}>0$ and $ \mathcal
B_{\theta}>0$, depending only  on $\theta$, are  defined  in   \eqref{defab} of
Appendix A, they are related to the unique solutions of the two singular boundary value problems
\begin{equation*}
\left\{\begin{split}
&-\frac{d^2\mathcal U} {dx^2}+ U^{2\theta-1}=0  \qquad\text{on } \ (-1,1), \\ 
&\ \ \ \ \mathcal U(-1)= \pm \infty, \qquad  \mathcal U(1)=+\infty.  
\end{split}
\right.
\end{equation*} 
Note in particular that \eqref{tyrannosaure} is fully determined
by the pair $(\ell_0,\hat \imath_0)$, and we shall therefore sometimes refer to it 
as $\mathcal{S}_{\ell_0,\hat \imath_0}.$ 
%We supplement the system \eqref{tyrannosaure} with the initial time
%condition 
%%
%\begin{equation}
%\label{initialtime}
%a_k(0)=a_k^0. 
%\end{equation}
%
Our first result is
\begin{theorem}
\label{maintheo1}
Assume that  the initial data  $(v_\eps(0))_{0<\eps<1}$   satisfy conditions
$({\rm H}_0)$, $({\rm H}_{1})$,  and $({\rm H}_{  \rm min })$, and let  
$0<S_{\rm max}\leq + \infty$ denote the maximal time of existence for the
system $\mathcal{S}_{\ell_0,\hat \imath_0}$ with initial data
$a_k(0)=a_k^0.$ Then, for $0 < s <S_{\rm max},$   
\begin{equation}
\label{unif}
\vm(s) \longrightarrow v^\star(\ell_0,\hat \imath_0,\{a_k(s)\})
\end{equation}
in $L^\infty_{\rm loc}(\R \setminus  \cup_{k=1}^{\ell_0} \{a_k(s)\}),$ as $\eps \to 0.$  
In particular, 
\begin{equation}
\label{trajectoryset0}
\mathfrak D_\eps (s) 
\longrightarrow \cup_{k=1}^{\ell_0} \{a_k(s)\}
\end{equation}
locally  in the sense of the Hausdorff distance, as $\eps \to 0.$  
\end{theorem}
We consider now the more general situation where $({\rm H}_{\rm min})$ is not
verified, and for $1 \leq k \leq \ell_0$ we denote by
$m_k^0$ the algebraic multiplicity of $a^0_k$, namely we set 
\begin{equation}
\label{multiplicity}
m_k^0=\hat \imath (k+\frac 12)-{\hat \imath(k-\frac 12)}.
\end{equation}
The case $m_k^0=0$ corresponds to {\it ghost fronts}, whereas $|m_k^0| \geq 2$ 
corresponds to {\it multiple fronts}. The total number of fronts that will
eventually emerge from
such initial data is given by 
$$
\ell_1=\sum_{k=1}^{\ell_0} \vert m_k^0 \vert,
$$
and their ordering is obtained by splitting multiple fronts according to 
the order in $\Sigma$. More precisely, we define the function $\hat \imath_1$ by
\begin{equation}
\label{defistar}
\left\{
\begin{aligned}
&\hat \imath_1(\frac 12) = \hat \imath_0(\frac 12),\\
&\hat \imath_1(M_k^0+p+\frac 12)=\hat \imath_0(k+\frac 12)+ p, {\rm \  for  \ } p=0,
\dots, \vert m_k^0 \vert-1  \  {\rm if \ } m_k^0>0  \\  
&\hat \imath_1(M_k^0+p+\frac 12)=\hat \imath_0(k+\frac 12)- p, {\rm \  for  \ } p=0,
\dots, \vert m_k^0 \vert-1   {\rm \ if \ } m_k^0<0,
\end{aligned}
\right.
\end{equation}
where $k=1,\cdots,\ell_0$ and $M_k^0 := \sum_{k=1}^{k-1} |m^0_k|.$ We say that $(\ell_1,\hat\imath_1)$ is
the splitting of $(\ell_0,\hat\imath_0).$

\begin{definition}
A {\it splitting solution} of \eqref{tyrannosaure} with initial data
$(\ell_0,\hat \imath_0,\{a_k^0\})$ on the interval $[0,S)$ is a solution $a\equiv 
(a_1,\cdots,a_{\ell_1})\: : (0,S) \to \R^{\ell_1}$ of $(\mathcal{S}_{\ell_1,\hat \imath_1})$ such
that
$$
\lim_{s\to 0^+}a_k(s) = a_j^0 \qquad \text{for} \qquad k =
M_j^0,\cdots,M_j^0+|m_j^0|-1,  
$$
for any $j=1,\cdots,\ell_0,$ where $(\ell_1,\hat \imath_1)$ is the splitting of $(\ell_0,\hat \imath_0).$
\end{definition}

We are now in position to complete Theorem \ref{maintheo1} by relaxing assumption $({\rm
H}_{\rm min})$.

\begin{theorem}
\label{alignedsolution}
Assume that  the initial data  $(v_\eps^0)_{0<\eps<1}$ satisfy conditions 
$({\rm H}_0)$ and  $({\rm H}_{1})$. Then there exists a subsequence $\eps_n \to
0$, and a splitting solution of $(\mathcal{S})$ with initial data $(\ell_0,\hat \imath_0,\{a_k^0\})$, 
defined on its maximal time of existence $[0,S_{\rm max})$,
and such that for any $0<s<S_{\rm max}$
\begin{equation}
\label{unifbis}
\vmen(s) \longrightarrow v^\star(\ell_1,\hat \imath_1,\{a_k(s)\})
\end{equation}
in $L^\infty_{\rm loc}(\R \setminus  \cup_{k=1}^{\ell_1} \{a_k(s)\}),$ as $n \to
+\infty.$  
In particular, 
\begin{equation}
\label{trajectoryset}
\mathfrak D_{\eps_n} (s) 
\longrightarrow \cup_{j=1}^{\ell_1} \{a_k(s)\}
\end{equation}
locally  in the sense of the Hausdorff distance, as $n \to +\infty.$  
\end{theorem}

\begin{remark}{\rm
Local existence of splitting solutions can be established in different
ways (including in particular using Theorem \ref{alignedsolution} !).
To our knowledge, \emph{ uniqueness is not known,} unless of course if $|m_k^0|
\leq 1$ for all $k.$ }  
\end{remark}
So far, our results are constrained by the maximal time of existence
$S_{\rm max}$ of the differential equation \eqref{tyrannosaure}, which is
related to the occurrence of collisions.
To pursue the analysis past collisions, we first briefly discuss some 
properties of the system of equations
\eqref{tyrannosaure}, we refer to Appendix B for more details.
The system \eqref{tyrannosaure} describes nearest neighbor interactions with an
interaction law of the form $\pm d^{-(\omega+1)}$, $d$ standing for  the
distance between fronts. The sign of the interactions is  crucial,  since the
system may contain both repulsive forces leading to spreading and attractive
forces leading to collisions, yielding  the maximal time of existence $S_{\rm
max}$.  In order to take  signs  into account,  we set  
\begin{equation}
\label{thalys}
\epsilon_{k+\frac 12}={\rm sign} \, ( \Gamma_{k+\frac 12})=- \dagger_k \dagger_{k+1}, 
{\rm \ for \ } k=0, \cdots, \ell_0-1.
\end{equation}
The case
$\epsilon_{k+\frac 12}=-1$ corresponds to \emph{repulsive forces} between
$a_k$ and $a_{k+1}$, whereas the case  $\epsilon_{k+\frac 12}=+ 1$ corresponds
to \emph{attractive forces} between $a_k$ and $a_{k+1}$, leading to collisions.
As a matter of fact, in this last case $a_{k+1}$ corresponds to the \emph{anti-front}
of $a_k$. 
In order to describe the magnitude of the forces, we introduce the subsets
$J^\pm$ of  $\{1,\cdots,\ell_0\}$ 
defined by $J^\pm=\{ k \in \{1,\cdots, \ell_0-1\}, \hbox{ such that } 
\epsilon_{k+\frac 12}=\mp 1 \}$ and the quantities 
\begin{equation}
\label{distancepond}
\left\{
\begin{aligned}
\tda(s) &=\inf\{ \vert a_k(s)-a_{k+1}(s)\vert , \ {\rm for } \  k \in{1, \cdots, \ell_0-1}\}  \\
\tda^{\pm}(s)&=\inf\{ \vert a_k(s)-a_{k+1}(s)\vert , \ {\rm for } \  k \in J^\pm\}
\end{aligned}
\right.
\end{equation} 
\begin{proposition} 
\label{getrude}
There are positive constants $\mathcal S_1$,  $\mathcal S_2$, $\mathcal S_3$
and $\mathcal S_4$  depending only on the coefficients of the equation
\eqref{tyrannosaure}, such that for any  time  
$s \in [0, S_{\rm max})$ we have
%\label{gedinox}
%
\begin{equation}
\label{gedoche}
\left\{
\begin{aligned}
&\textswab d_{a}^+(s)  \geq
 \left( \mathcal S_1 s+ \mathcal S_2  \textswab d_{a}^+(0)^{\omega+2} \right)^{\frac{1}{\omega+2}}, \\
&\textswab d_{a}^-(s)  \leq  \left(\mathcal S_3\textswab d_{a}^-(0)^{\omega+2}-\mathcal 
S_4 t \right)^{\frac{1}{\omega+2}}.
\end{aligned}
\right.
\end{equation}
If for every $k=1,\cdots,\ell_0$ we have $\epsilon_{k+\frac 12}=-1$, then $S_{\rm max}=+\infty$. 
Otherwise, we have the estimate
\begin{equation}
\label{gedelope}
S_{\rm max}  \leq  \frac{\mathcal S_3}{\mathcal S_4} \left( \textswab d_{a}^-(0)\right)^{\omega+2} \equiv \mathcal K_0\left( \textswab d_{a}^-(0)\right)^{\omega+2}.
\end{equation}
\end{proposition}
This result shows that the maximal time of existence for solutions to
\eqref{tyrannosaure} is related to the value of  $\textswab d_{a}^-(0)$, the
minimal distance between fronts and anti-fronts at time $0$. By the semi-group property, 
the same can be said about $\textswab d_{a}^-(s)$, namely 
$$
S_{\rm max}-s  \lessapprox  \textswab d_{a}^-(s)^{\omega+2}.
$$
On the other hand,  in view of $(\mathcal S)$, $\textswab d_{a}^-(s)$ provides an upper bound for the speeds $\dot{a}_k(s)$ 
in case of collision, namely
$$
\vert \frac{d}{ds}a_k(s) \vert \lessapprox   \textswab d_{a}^-(s)^{-(\omega+1)}.
$$
It follows that 
$$
\int_0^{S_{\rm max}} \vert \frac{d}{ds}a_k(s)\vert \, ds  \lessapprox \int_0^{S_{\rm max}} (S_{\rm max}-s)^{-\frac{\omega+1}{\omega+2}} \, ds < +\infty
$$
and therefore that the trajectories are absolutely continuous up to the collision time. 
Also, since
$\textswab d_{a}^+$ remains bounded from below by a positive constant,
each front can only enter in collision with its anti-front (but there could
be multiple copies of both). From a heuristic  point of view,  it
is therefore rather simple to extend  solutions past the collision time: it
suffices to remove the colliding pairs from the collection of points, so that
the total number of points has been decreased by an even number. More precisely,
we have

\begin{corollary}\label{alarrivee}
Let $\ell_1,\hat \imath_1$, $a\equiv(a_1,\cdots,a_{\ell_1})$ and $S_{\rm max}$ be as in Theorem \ref{alignedsolution}. 
Then, there exist
$\ell_2 \in \N$ such that $\ell_1 -\ell_2 \in 2\N_*$, and there exist $\ell_2$
points $b_1<\cdots<b_{\ell_2}$ such that for all $k=1,\cdots,\ell_1$
$$
\lim_{s\to S_{\rm max}^-} a_k(s) = b_{j(k)} \qquad\text{for some } j(k) \in
\{1,\cdots,\ell_2\}.
$$  
Moreover, if we set $\hat \imath_2(\frac 12)= \hat \imath_1(\frac 12)$ and
$$
\hat \imath_2(q+\frac{1}{2}) = \hat \imath_1(k(q)+\frac 12) \quad \text{where} \quad
 k(q) = max\{k\in \{1,\cdots,\ell_1\} \text{ s.t. } j(k)=q\},
$$ 
for $q=1,\cdots,\ell_2,$ then 
$$
\hat \imath_2(q+\frac 12) - \hat \imath_2(q-\frac 12) \in \{+1,-1,0\}
$$
for all $q=1,\cdots,\ell_2.$
\end{corollary}

We stress than Corollary \ref{alarrivee} is obtained from Theorem \ref{alignedsolution}
using only properties of the system of ODE's \eqref{tyrannosaure}, in particular
Proposition \ref{getrude}.

We are now in position to state our last result, namely
 
\begin{theorem}
\label{colissimo} 
Under the assumptions of Theorem \ref{alignedsolution}, we have  as $n \to +\infty,$ 
\begin{equation}
\label{etcestreparti}
\vmen(S_{\rm max}) \longrightarrow v^\star(\ell_2,\hat \imath_2,\{b_k\})\quad \text{ in } 
\quad L^\infty_{\rm loc}(\R \setminus  \cup_{k=1}^{\ell_2} \{b_k\}),
\end{equation}
where $\ell_2$, $\hat \imath_2$ and $b_1<\cdots<b_{\ell_2}$ are given by Corollary \ref{alarrivee}. 
In particular the sequence $(\vmen(S_{\rm max}))_{n\in \N}$, considered as initial data, satisfies the 
assumptions $({\rm H}_0)$ and $({\rm H}_1)$
with $\ell_0:=\ell_2$ and $\{a_k^0\}:=\{b_k^0\}.$
\end{theorem}

We may therefore apply Theorem \ref{alignedsolution} to the sequence of initial data  $(\vmen(S_{\rm max}))_{n\in \N}$, 
and therefore, using the semi-group property of \eqref{glpara}, extend the analysis past $S_{\rm max}.$ Notice that
since the multiplicities given by $\hat \imath_2$ are either equal to $\pm 1$ or $0,$ no further subsequences are needed 
to pass through the collision times. Finally, since the total number of fronts is decreased at least by $2$ at each 
collision times, the latter are finitely many.    

\medskip
\noindent
{\bf  Some comments on the  results.}
Motion  of fronts for one-dimensional  \emph{scalar}  reaction-diffusion
equations has already  a quite long history. 
 Most of the efforts have been devoted until recently  to the case where the
potential possesses\emph{ only two wells} with non vanishing second derivative:
such potentials are often referred to as \emph{Allen-Cahn potentials}.  Under
suitable preparedness  assumptions on the initial datum,   the  precise motion
law for the fronts has been derived  by Carr and Pego in their seminal work
\cite{carpego} (see also Fusco and Hale \cite{fuschale}).  They showed that the
front points are moved, up to the first  collision time,  according to a first
order  differential equation of nearest neighbor interaction type, with
interactions  terms  proportional to  $ \exp
(-\eps^{-1}(a^\eps_{k+1}(t)-a_k^\eps(t)))$. These results present substantial 
differences with   the results in the present paper, in particular we wish to
emphasize the following points:   
\begin{itemize}
\item only attractive   forces leading eventually to the annihilation  of
\emph{fronts} with  \emph{anti-fronts} forces are present. 
\item the  equation  is \emph{not} renormalizable. Indeed, the various forces $
\exp (- \eps^{-1}(a_{k+1}^\eps (t)-a_k^\eps (t)))$ for different values of $k$ may be of very
different orders of magnitude, and hence not commensurable.
\end{itemize}
Besides this, the essence of their method  is quite   different: it  relies on a
careful study of the linearized problem around the stationary front, in
particular from the spectral point of view.  This type of approach is also
sometimes termed the \emph{geometric  approach} (see e.g. \cite{chengene}).   
    At least two other methods have been applied successfully on the Allen-Cahn
equation. Firstly, the method of subsolutions and supersolutions turns out to
be extremely powerful and allowed to handle larger classes of initial data and
also to extend the analysis past collisions:  this is for instance achieved  by
Chen  in   \cite{chengene}.  Another direction is given by the global energy
approach initiated  by Bronsard and Kohn \cite{BrKo}.  We refer to \cite{BOS8}
for a more  references on these methods.  
    
Several ideas and concepts presented here  are  influenced  by  our earlier
work on the motion of vortices in the two-dimensional parabolic Ginzburg-Landau
equation \cite{BOS1,BOS2}.  As a matter of fact, this equation  yields   another
remarkable  example of \emph{ renormalizable slow motion}, as  proved by Lin or
Jerrard and Soner (\cite {Lin, JeSo}). Our interest  in the questions studied in
this paper was certainly driven by the possibility of finding an analogous
situation in one space dimension.  
   
This paper belongs  to series of papers we  have written on the slow motion
phenomenon for reaction-diffusion equation   of gradient type with multiple-wells
(see \cite{BOS8, BS, BS2}). Common to these papers is a general approach based
on the following ingredients:  
\begin{itemize}
\item A  \emph{localized version of the energy identity} (see subsection
\ref{discrepence}). Fronts are  then handled as concentration points of the
energy, so that the evolution of local energies yields also the motion of
fronts. Besides dissipation, this localized energy identity  contains a flux
term,  involving  the \emph{discrepancy} function, which has a simple
interpretation for stationary solutions. Using  test functions which are
\emph{affine near the fronts},  the flux term does not see the core of the
front, only its tail. 
\item Parabolic estimates \emph{away} from the fronts. 
\item Handling the time derivative as a \emph{perturbation } of the
one-dimensional elliptic equations, allowing hence elementary tools as
Gronwall's identities.  
\end{itemize}
Parallel to this paper, we are also\emph{ revisiting the  scalar non-degenerate
case} in \cite{BS2},  considering in particular the case were there are more
than two wells, leading as mentioned to repulsive forces which are not present
in the Allen-Cahn case. Several   tools are  shared by the two papers, for
instance we rely  on   related   definitions
%\footnote{ 
%{ Notice however that property WP3 takes different forms, due to the differences
%in the decay of stationary fronts. Notice also that a localization  induced by
%a length $\rL$ is required in the present paper: such a localization is not
%becessary in \cite{BS2} }  
%}  
and properties of regularized fronts, and the properties of the ordinary
differential equations are quite similar.  From a technical point of view
differences appear at the level  of the magnitudes of energies as well as of the
parameter $\updelta$ involved in the definition of regular fronts, and more
crucially on the nature of the parabolic estimates off  the front sets.  Whereas
in \cite{BS2} we rely  essentially on \emph{linear} estimates, in the degenerate
case considered here our  estimates are  \emph{truly non-linear}, obtained
mainly through  an extensive use of the comparison principle.  

\smallskip
Finally,  it is presumably worthwhile to mention that the situation in higher
dimension is very different: the dynamics is  dominated by \emph{mean-curvature}
effects. The phenomena  considered in the present paper are  therefore  of
lower order, and do not appear in the limiting equations. 
   
\medskip
Among  the problems left open in our paper, we would like to emphasize again
the question of \emph{uniqueness} of splitting solutions for
\eqref{tyrannosaure}, as well as the possibility to interpret our
convergence results in terms of \emph{Gamma-limit} involving  a renormalized
energy (see e.g \cite{sanser} for related results on the Ginzburg-Landau
equation).   

%%%%%%%%%%%%%%%%%%%%%%%%%%%%%%%%%%%%%%%%%%%%%%%%%%%%%%%%%%%%%%%%%%%%%%%%%%%%%%%  
\subsection{Regularized fronts}   
%%%%%%%%%%%%%%%%%%%%%%%%%%%%%%%%%%%%%%%%%%%%%%%%%%%%%%%%%%%%%%%%%%%%%%%%%%%%%
The notion of regularized fronts   
is   central  in our   description of  the dynamics of  equation $({\rm PGL})_\eps$. 
It is aimed to describe in a quantitative way chains of stationary solutions
which are well-separated  and suitably glued together. It also allows to pass
from \emph{front sets} to   \emph{front points}, a notion which is more accurate
and requires therefore improved estimates.  
Recall first  that for $i \in \{1, \cdots, \q-1\}$, there exist a unique (up to
translations) solution $\zeta_i ^+$ to the  stationary  equation with  $\eps=1$,

\begin{equation}
\label{stat}
-v_{xx}+  V'(v)=0    \ { \rm on \ }  \R,  
\end{equation}
with, as conditions at infinity,  $v( -\infty)=\upsigma_i $ and $v(
+\infty)=\upsigma_{i+1}$.  Set,   for $i \in \{1, \cdots, \q-1\}$,  $\zeta_i
^-(\cdot)\equiv \zeta_{i}(-\cdot)$, so that $\zeta_i^{-}$ is the unique (up to
translations) solution to \eqref{stat} such that  $v( +\infty)=\upsigma_i $ and
$v(-\infty)=\upsigma_{i+1}$.  A remarkable yet elementary fact, related to the
scalar nature of the equation, is that there are no other non trivial finite
energy solutions to equation \eqref{stat} than the solutions $\zeta_i^\pm$  and
their translates: 
in particular there are no solutions connecting minimizers which are not nearest
neighbors.  For $i=1, \cdots, \q-1$, we fix a point $z_i$ in the
interval $(\upsigma_i,
\upsigma_{i+1})$ where the potential $V$  restricted to $[\upsigma_i,
\upsigma_{i+1}]$ achieves its maximum and we set $\displaystyle{\mathcal Z =
\{z_1, \cdots, z_{\q-1}\}}$. Again, since we consider only the one-dimensional
scalar case, any solution $\zeta_i$ takes once and only once the value $z_i$. 
 
\medskip 
 
We next  describe a  \emph{local} notion of well-preparedness\footnote{By local,
we mean with respect to the interval $[-\rL, \rL]$. In contrast the related
notion  introduced in \cite{BS2} is global on the whole of $\R$}. For an
arbitrary $r>0$, we denote by ${\rm I}_r$ the interval $[-r,r].$  
%In the entire 
%paper, we will also always implicitely assume that the maps $u: \R \to \R$
%which we consider satisfy the energy bound $\hzero$.

\begin{definition}
\label{def}
Let $\rL >0$ and $\updelta>0.$ We say that a map $u$ verifying $\hzero$ 
satisfies the preparedness
assumption $\mathcal{W P}_{\eps}^{\tL} (\updelta)$ if the following  two
conditions are fulfilled:
\begin{itemize}
\item $\left(\wpi(\ud)\right)$ \qquad
We have 
\begin{equation}
\label{confedere}
\mathcal D(u) \cap \ILL \subset \IL
\end{equation}
and there exists a collection of points $\{a_k \}_{k\in J }$ in $\IL$, 
with $J=\{1, \cdots, \ell\}$, such that
\begin{equation}
\label{confondu}
\mathcal D(u) \cap \ILL \subset 
\underset{k \in J} \cup I_k,\ {\rm \ where \ } I_k= [a_k-\updelta ,a_k+\updelta
].
\end{equation}
For   $ k\in J$, there exist a number $i(k) \in \{1, \cdots, \q-1\}$  such that 
$u(a_k)=z_{i(k)}$
and  a symbol $\dagger_k\in \{+,-\}$ such that
\begin{equation}
\label{bugs}
\left \Vert  u(\cdot)- \zeta_{i(k)}^{\dagger_k}\left(\frac{\cdot-a_k}{\eps}\right)\right \Vert_{C_\eps^1 (I_k)}
\leq  \exp\left(-\frac{\ud}{ \eps}\right), 
\end{equation}
where $\Vert u \Vert_{ C_\eps^1(I_k)}=\Vert u \Vert_{L^\infty(I_k)}+\eps \Vert u'\Vert_{L^\infty(I_k)}$.
\item $\left(\wpo(\ud)\right)$ \qquad Set 
$\Omega_{\rm L}=\ILL\setminus {\underset {k=1}{\overset {\ell}\cup}} I_k. $
We have the energy estimate
\begin{equation}
\label{bunny}
\int_{\Omega_{\rm L} } e_\eps\left(u (x)   \right)dx \leq {\rm C}_{\rm w} M_0 \left(
\frac{\eps}{\ud}\right)^{\omega}.
\end{equation}
\end{itemize}
\end{definition}
In the above definition ${\rm C}_{\rm w}>0$ denotes a constant, 
whose exact value is fixed once for all 
by Proposition \ref{vaderetro} below, and which depends only on
$V$. Condition $\wpi(\ud)$ corresponds
to an \emph{inner matching} of the map with
stationary fronts, it is only
really meaningful if $\ud >> \eps$. In the sequel we always assume that 
\begin{equation}\label{eq:alpha1grand}
\frac{\rL}{2}\geq \ud \geq \alpha_1 \eps,
\end{equation}
where $\alpha_1$ is larger than the $\alpha_0$ of Theorem \ref{mainbs2} and also
sufficiently large so that if $\wpi(\ud)$ holds then the points
$a_k$ and the indices $i(k)$ and $\dagger_k$ are \emph{uniquely} and therefore 
\emph{unambigously} determined and the intervals $I_k$ are disjoints. In particular,
the quantity $\dmineL(s)$, defined by
$$
\dmineL(s) := \min\left\{ a_{k+1}^\eps(s) -a_k^\eps(s),\quad k=1,\cdots, \ell(s)-1\right\} 
$$
if $\ell(s)\geq 2$, and $\dmineL(s)=2\rL$ otherwise, satisfies $\dmineL(s) \geq 2\delta.$
Condition  
$\wpo(\ud)$ is in some weak sense an
\emph{outer matching}: it is crucial for some of our energy estimates and its form
is motivated by  energy decay estimates for stationary solutions. Note that condition
$\wpi(\ud)$ makes sense on its own, whereas condition 
$\wpo(\ud)$ only makes sense if condition $\wpi(\ud)$ is fulfilled. Note also that
the larger $\ud$ is, the stronger condition $\wpi(\ud)$ is. The same is not
obviously true for condition $\wpo(\ud)$, since the set of integration $\Omega_{\rm
L}$ increases with $\ud$. As a matter of fact, the constant ${\rm C}_{\rm w}$ 
in \eqref{bunny} is
chosen sufficiently big\footnote{In view of $\wpi(\ud)$, how big it needs to be is
indeed related to energy decay estimates for the fronts $\zeta_i$.} so that
$\wpo(\ud)$ also becomes stronger when $\ud$ is larger. 
We next specify Definition \ref{def} for the  maps $x \mapsto \vm(x, s).$ 
\begin{definition}
\label{def2}  For $s\geq 0,$ we say that the assumption $\wp(\updelta,
s)$ (resp.
$\wpi(\updelta, s)$)  holds if the map $x \mapsto \vm(x, s)$ satisfies 	
$\wp(\updelta)$ (resp.
$\wpi (\updelta)$). 
\end{definition}
When assumption $\wpi (\updelta, s)$ holds, then all symbols will be indexed
according to $s$.  In particular, we write\footnote{In principle and at this stage, all those symbols depend also upon $\eps$. Since eventually $\ell$ and $J$ will be $\eps$-independent, at least for $\eps$ sufficiently small, we do not explicitly index them with $\eps$.}    
$\ell(s)=\ell$, $J(s)=J$, and $a_k^\eps (s)=a_k.$ 
The  points $a_k^\eps (s)$ for  $k \in J(s)$, are now termed the \emph{front points}.
Whereas in \cite{BS2} we are able, due to parabolic regularization, to
establish under suitable conditions that $\wp(\updelta, s)$ is
fulfilled for length of  the same order as the minimal distance between the
front points, this \emph{is not the  case}  in the present situation.  More precisely,
two orders of magnitude for $\ud$ will be considered, namely 
\begin{equation}
\label{magnidelta}
\udl = \frac{1}{\uprho_{\rm w}} \eps \left|\log \left(4M_0^2\frac{\eps}{\rL}\right)\right| 
\qquad \text{and}\qquad
\udll =  \frac{\omega}{\uprho_{\rm w}} \eps \log \left(\frac{1}{\uprho_{\rm w}} \left|\log \left(4M_0^2\frac{\eps}{\rL}\right)\right|\right).
\end{equation}
In \eqref{magnidelta}, the constant $\uprho_{\rm w}$ (given by Lemma \ref{prop:wpiok} below) depends only on $V.$ The main property for our purposes is that $\udll/\eps$ and $\udl/\udll$ both tend to $+\infty$ as $\eps/\rL$ tends to $0$. 
\medskip

In many places, it is useful to rely on a slightly stronger version of the confinement condition \eqref{confedere},
which we assume to hold on some interval of time.  More precisely, for positive $\rL,S$  we consider the condition 
\begin{equation*}
%\left\{
%\begin{array}{ll}
% \displaystyle 
\mathfrak{D}_\eps(s) \cap \ILLLL \subset \IL, \qquad\forall \ 0\leq s \leq S.\\
% \end{array}\right.
% \qquad\text{and}\qquad S\leq \rho_1 \rL^{\omega+2}.
\leqno{(\mathcal{C}_{\rL,S})} 
\end{equation*}
where the constant ${\rm C}_{e}$ is defined in Proposition \ref{estimpar} here below.  
%The constant $\rho_1>0$ which appears in the definition of $(\mathcal{C}_{\rL,S})$ is fixed once for all by Lemma \ref{avecdec} below\footnote{Compare with the constant $\uprho_0$ in Theorem \ref{mainbs2}}, and depends only on $V.$  
For given $\rL_0>0$ and $S>0$, it follows easily from assumption $({\rm H}_1)$ and Theorem \ref{mainbs2} that there exists 
$\rL\geq \rL_0$ for which the first condition in $(\mathcal{C}_{\rL,S})$ is satisfied.
Under condition $(\mathcal{C}_{\rL,S})$, the estimate
\begin{equation}\label{cirque}
\mathcal{E}_\eps(\vm(s),\ILLL\setminus \ItdL ) \leq {\rm C}_{e} \left(\frac{\eps}{\rL}\right)^\omega, \qquad\forall
 \ s\in [\eps^\omega \rL^2,S], 
\end{equation}
follows from the following regularizing effect, which was obtained in \cite{BS}:
\begin{proposition}[\cite{BS}]
\label{estimpar}
Let $v_\eps$ be a solution to $\PGL$, let $x_0\in \R,$ $r>0 $ and $0\leq s_0<S$ be such that 
\begin{equation}
\label{offside}
\vm(y,s) \in B(\upsigma_i,\upmu_0) \quad \text{for all }\ (y,s) \in
[x_0-r,x_0+r]\times[s_0,S],
\end{equation}
for some $i \in \{1,\cdots,\q\}.$  Then we have for  $s_0<s\leq S$ 
\begin{equation}
\label{crottin}
\eps^{-\omega}	\int_{x_0-3r/4}^{x_0+3r/4} e_\eps(\vm(x,s))\, dx \leq
  \frac{1}{10} {\rm C}_e\left( 1 + \left(\frac{\eps^\omega r^2}{s-s_0}
\right)^{\frac{\theta}{\theta-1} }   \right)
\left(\frac{1}{r}\right)^{\omega} 
\end{equation}
as well as 
\begin{equation}
\label{crottin2}
\vert \vm (y, s)-\upsigma_i \vert \leq \frac{1}{10} {\rm C}_e \eps^{\frac{1}{\theta-1}}\left(\left (\frac{1}{r}\right)^{\frac{1}{ \theta-1}}+\left(\frac{\eps^{\omega} }{s-s_0}\right)^{\frac{1}{2(\theta-1)}} \right), 
\end{equation}
for $y \in [x_0-3r/4, x_0+3r/4],$ where the constant $C>0$  depends only on $V$.
% as well as the upper bounds, \begin{equation}
% \label{crottin2}
% \vert \vm (x, s)-\upsigma_i \vert \leq {\rm
%C}_{_V}\eps^{\frac{1}{\theta-1}}\left(\left (\frac{1}{r}\right)^{\frac{1}{
%\theta-1}}+\left(\frac{\eps^{\omega} r^2}{(s-s_0)}\right)^{\frac{1}{\theta-1}}
%\right), {\rm \ for \ } x \in [x_0-r/2, x_0+r/2]. 
% \end{equation}
\end{proposition}
\noindent
%As a matter of fact we will use assumption $(\mathcal{C}_{\rL,S})$ for different values of $\rL$ and $S,$ and also in Section \ref{clearstream} for rescaled versions of $\mathfrak{v}_\eps.$

\medskip
Our first ingredient is
\begin{proposition}
\label{parareglo}  There exists $\alpha_1>0,$ depending only on $M_0$ and $V,$ such that if  $L\geq \alpha_1\eps$ and if $(\mathcal{C}_{\rL,S})$ holds, then each subinterval of $[0,S]$ of length $\eps^{\omega+2}\big(\rL/\eps\big)$ contains at least one time $s$ for which $\wp(\udl,s)$ holds.
\end{proposition}
\noindent
The idea behind Proposition \ref{parareglo} is that, $\PGL$ being a gradient flow, on a sufficiently large interval of time one may find some time where the dissipation of energy is small. Using elliptic tools, and viewing the time derivative as a forcing term, one may then establish property $\wp(\udl,s)$ (see Section \ref{cavenecadas} and Section \ref{regupara}). 

\medskip
The next result expresses the fact that the equation preserves to some extent the
well-preparedness assumption.

\begin{proposition}
\label{parareglo2} 
Assume that $(\mathcal{C}_{\rL,S})$ holds, that $\eps^\omega \rL^2\leq s_0\leq S$ is such that 
$\wp(\udl,s_0)$ holds, and assume moreover that  
\begin{equation}
\label{greece0}
\dmineL(s_0) \geq 16 \Big(\frac{\rL}{\uprho_0\eps}\Big)^\frac{1}{\omega+2} \eps.
\end{equation}
Then  $\wp(\udll,s)$ holds for all times 
$s_0 + \eps^{2+\omega}  \leq  s \leq \mathcal T_0^\eps(s_0)$, where
$$
\mathcal T_0^\eps(s_0) = \max \left\{ s\in [s_0+\eps^{2+\omega},S] \quad \text{s.t.}\quad \dmineL(s')\geq 8 \Big(\frac{\rL}{\uprho_0\eps}\Big)^\frac{1}{\omega+2} \eps \quad \forall s'\in [s_0+\eps^{\omega+2}, s]\right\}.
$$ 
For such $s$ we have  $J(s)=J(s_0)$ and  for any $k \in J(s_0)$ we have
$ \upsigma_{i(k\pm \frac 12)}(s)=\upsigma_{i(k\pm \frac 12)} (s_0)$ and $\dagger_k(s)=\dagger_k(s_0).$
\end{proposition}

%$\leq  s' \leq \gT_0^\eps(s)$, where
%\begin{equation}
%\label{nonobstant}
%\gT_0^\eps (s)= \max\left\{ s\leq s'\leq S, \   \dmineL(s'') \geq \frac12  
%\eps^{\frac12}  \text{ for all } s \leq s'' \leq s'\right\}.
%\end{equation}
%Moreover $J(s')=J(s)$ and  for any $k \in J(s)$ we have
%$ \upsigma_{i(k\pm \frac 12)}(s')=\upsigma_{i(k\pm \frac 12)} (s).$

Given a family of solution $(v_\eps)_{0<\eps<1}$, we introduce   the additional condition
\begin{equation}
\label{dminstar}
d_{\rm min}^{*}(s_0)\equiv   \underset {\eps \to 0}  \liminf \,  d_{\rm min}^{\eps, \rL} (s_0)  >0, 
\end{equation}
which makes sense if $\WPL(\upalpha_1 \eps, s_0)$ holds and expresses the fact that the fronts stay uniformly well-separated. 
The first  step in our proofs, which is stated in Proposition \ref{firststep1}
below,   is  to establish the conclusion of Theorem \ref{maintheo1} under this
stronger assumptions on the initial datum.  
%We introduce  a second stopping
%time, namely 
%\begin{equation}
%\label{stopla}
%\gT_1^\eps (s)= \inf\{ s \leq s'\leq s+\DSL, \ \,   \vert
%a_k^\eps(s')-a_k^\eps(s)¬¨‚Ä†\vert  \geq  \frac{1}{16} d_{\rm min}^\eps (s), \ k
%\in J  {\rm \ or \ } s'=s+\DSL \} 
%\end{equation}
%Notice that, if condition \eqref{dminstar} is met, then $\gT_1^\eps <
%\gT_0^\eps$, provided $\eps>0$ is sufficiently large.   
From the inclusion  \eqref{theodebase1} and Proposition \ref{parareglo2} we will obtain:

\begin{corollary}  
\label{jamboncru}
Assume also that $\mathcal{C}_{\rL,S}$ holds,  let $s_0\in [0,S]$ and assume  that  $\mathcal{W P}_\eps^\rL (\upalpha_1 \eps, s_0 )$ holds for all $\eps$ sufficiently small and  
that \eqref{dminstar} is satisfied. Then, for $\eps$ sufficiently small,  
%$$ 
%s+\DSL \geq  \gT_1^*(s)\equiv    \underset {\eps \to 0}  \liminf \  \gT_1^\eps
%(s)  \geq s+ \uprho_0 \left( \frac { \dminstar }{32}\right)^{\omega+2}. 
%$$
%Moreover
\begin{equation}\label{eq:levain1}
\mathcal{W P}_\eps^\rL(\udll, s ) \qquad \text{ and } \qquad d_{\rm min}^{\eps,\rL}(s) \geq \frac12  d_{\rm min}^*(s_0)
\end{equation}
are satisfied for any
$$
s\in I^\eps(s_0)\equiv\Big[s_0+2\rL^{2}\eps^\omega, s_0+ \uprho_0 \left( \frac{ d_{\rm min}^*(s_0)}{8}\right)^{\omega+2}\Big]\cap [0,S],
$$  
as well as the identities $J(s)=J(s_0)$, $\upsigma_{i(k\pm \frac 12)}(s)=\upsigma_{i(k\pm \frac 12)} (s_0)$ and $\dagger_k(s)=\dagger_k(s_0),$ for any $k \in J(s_0).$
\end{corollary}
Hence,  the collection of  front points $\{a_k^\eps(s)\}_{k \in J} $ is well-defined, and the approximating regularized fronts $\zeta_{i(k)}^{\dagger_k}$ do not depend on $s$ (otherwise than through their position), on the  full time interval  $I^\eps(s_0).$

%%%%%%%%%%%%%%%%%%%%%%%%%%%%%%%%%%%%%%%%%%%%%%%%%%%%%%%%%%%%%%%%%%%%%%%%%%%%%%%%%%%%%%%%%
\subsection{Paving the way to the motion law}
\label{discrepence}
%%%%%%%%%%%%%%%%%%%%%%%%%%%%%%%%%%%%%%%%%%%%%%%%%%%%%%%%%%%%%%%%%%%%%%%%%%%%%%%%%%%%%%%%%%

As   in \cite{BOS8}, we use extensively    the
localized version of \eqref{energyidentity}, a tool which turns out to be
perfectly adapted to track the evolution of fronts. Let $\chi$ be an arbitrary 
smooth test function with compact support. 
Set, for $s\geq 0,$
\begin{equation}
\label{testf}
\mathcal I_\eps(s, \chi)=\int_\R e_\eps\left(\vm(x, s)\right)\chi(x)dx.
\end{equation}
In integrated form the localized version of the energy identity  writes
\begin{equation}
\label{localizedenergy0}
\mathcal I_\eps (s_2, \chi)-\mathcal I_\eps (s_1, \chi)+\int_{s_1}^{s_2}\int_\R
\eps^{1+\omega}\chi (x) |\partial_s \vm (x, s)|^2\,  dx ds=  
\eps^{-\omega} \int_{s_1}^{s_2}\mathcal{F}_S(s,\chi,v_\eps)  ds, 
\end{equation}
where the term $\mathcal{F}_S$ is given by
\begin{equation}
\label{flux}
\mathcal{F}_S(s,\chi, \vm)= \int_{\R \times\{s\}} \left(    
\left[ \eps \frac{\dot \vm^2}{2}-\frac{V(\vm)}{\eps}\right ]\ddot \chi (x)
\right)\,dx\equiv  \int_{\R \times\{s\}} \xi_\eps (\vm(\cdot, s)) \ddot \chi dx.
\end{equation}
The last integral on the left hand side of identity \eqref{localizedenergy0}
stands for local dissipation, whereas the right hand side second  is   a flux.
The quantity $\xi_\eps$ is defined for a scalar function $u$ by  
\begin{equation}
\label{discrepance}
\xi_\eps (u) \equiv \eps\frac{\dot u^2}{2}-\frac{V(u)}{\eps}, 
\end{equation}
and is referred to as {\bf the discrepancy term}. It  is constant for
solutions to the stationary equation 
$ -{u}_{xx}+\eps^{-2} V'(u)=0  $
on some given interval $I$
and  vanishes  for finite energy solutions on $I=\R$. Notice  that 
$\displaystyle{ \vert \xi_\eps(u)\vert  \leq e_\eps(u).}$
We set  for two given  times $s_2\geq  s_1\geq 0$ and $\rL \geq 0$  
\begin{equation}
\label{leglaude}
\disL [s_1, s_2]=\eps \int_{\IctL \times [s_1\eps^{-\omega}, s_2\eps^{-\omega}]} 
\vert \frac{\partial v_\eps}{\partial t}\vert^{^2} dxdt
= \eps^{1+\omega} \int_{\IctL \times [s_1, s_2]} 
\vert \frac{\partial \vm}{\partial s}\vert^{^2} dxds.
\end{equation}
Identity \eqref{localizedenergy0}  then yields the  estimate, if we assume that ${\rm supp} \chi \subset \IctL,$
\begin{equation}
\label{louloulepou}
\left \vert \mathcal I_\eps (s_2, \chi)-\mathcal I_\eps (s_1,
\chi)-\eps^{-\omega} \int_{s_1}^{s_2}\mathcal{F}_S(s,\chi,v_\eps)  ds \right
\vert \leq  \disL[s_1, s_2] \Vert \chi \Vert_{L^\infty (\R)}.  
\end{equation}
We will show that under  suitable  assumptions, that   the right hand side of
\eqref{louloulepou} is small (see Step 3 in the proof of Proposition \ref{firststep1}), so that the term
$\displaystyle {\eps^{-\omega} \int_{s_1}^{s_2}\mathcal{F}_S(s,\chi,v_\eps)
ds}$ provides a good approximation of  $\displaystyle {  \mathcal I_\eps (s_2,
\chi)-\mathcal I_\eps (s_1, \chi)}$.  On the other hand, it follows from the
properties of regularized maps proved in Section \ref{patdef} (see Proposition
\ref{globalmoquette} there)  that if $\WPL(\udll, s)$ holds then 
\begin{equation}
\label{mireille}
 \left \vert \mathcal I_\eps (s, \chi)-   \underset {k \in J} \sum \chi (a_k^\eps(s)) \S_{i(k)}  \right\vert  \leq  CM_0\left( \big(\frac{ \eps}{\udll}\big)^{\omega}
 \Vert \chi\Vert_{\infty} + \eps \Vert \chi'\Vert_{\infty}\right),
\end{equation}
where $\displaystyle{ \S_{i(k)}}$ stands for the energy of the corresponding stationary front. 
Set
$$\mathfrak F_\eps(s_1, s_2, \chi)\equiv \eps^{-\omega} \int_{s_1}^{s_2}\mathcal{F}_S(s,\chi,\vm)  ds
\equiv \int_{s_1}^{s_2}\eps^{-\omega}\xi_\eps (\vm(\cdot, s)) \ddot \chi(\cdot)\,  ds.$$
Combining \eqref{louloulepou} and \eqref{mireille}  shows that,  
if $\WPL(\udll, s)$ holds for any $s \in (s_1, s_2)$, then we have
\begin{equation}
\begin{aligned}
\label{gratindo}
& \quad \vert \underset {k \in J} \sum\left [ \chi (a_k^\eps(s_2))-\chi (a_{k}^\eps (s_1))\right]  \S_{i(k)} -\mathfrak F_\eps(s_1, s_2, \chi) \vert \\
\leq 
& \quad CM_0 \left( \big( \log \vert\!\log \frac{\eps}{\rL} \vert \big)^{-\omega} 
  \Vert \chi\Vert_\infty 
 + \eps \Vert \chi'\Vert_\infty \right)  + \disL[s_1, s_2] \Vert \chi\Vert_\infty .
\end{aligned}
\end{equation}

If the test function $\chi$ is choosen to be affine near a given front point
$a_{k_0}$  and zero near the other  front points in the collection, then the
first term on the left hand side yields a   measure of the motion of   $a_{k_0}$
between  times $s_1$ and $s_2$, whereas the second, namely $ \mathfrak
F_\eps(s_1, s_2, \chi)$, 
is hence a good approximation of  the measure of this motion, \emph{provided we
are able to estimate the dissipation $\disL(s_1, s_2)$}.  Our previous discussion
suggests that  
\begin{equation*}
\label{suggestion}
a_{k_0}^\eps(s_2) - a_{k_0}^\eps (s_1) \simeq \frac{1}{\chi'(a_{k_0}^\eps)\S_{i(k_0)}}\mathfrak F_\eps(s_1, s_2, \chi).  
\end{equation*}
It turns out that the computation of  $ \mathfrak F_\eps(s_1, s_2, \chi)$ can be
performed with  satisfactory accuracy if the  
test function $\chi$  is affine (and hence as  vanishing second derivatives) close to the front set, this is the object of the next subsections.

%%%%%%%%%%%%%%%%%%%%%%%%%%%%%%%%%%%%%%%%%%%%%%%%%%%%%%%%%%%%%%%%%%%
\subsection{A first compactness result}
%%%%%%%%%%%%%%%%%%%%%%%%%%%%%%%%%%%%%%%%%%%%%%%%%%%%%%%%%%%%%%%%%%%

%We have seen  so far that under suitable  assumptions, fronts points are
%well-defined and \eqref{gratindo} gives an approximation for their motion, provided 
%we are able to bound the dissipation $\disL [s_1,s_2]$ and to compute the value of 
%$\mathfrak F_\eps(s_1, s_2, \chi).$ 

A first step in deriving the motion law for the fronts is to obtain rough bounds from above for both $\disL [s_1,s_2]$ and $\mathfrak F_\eps(s_1, s_2, \chi).$  To obtain these, and under the assumptions of Corollary \ref{jamboncru}, notice that if $\ddot \chi$ vanishes on the set $\{a_k^\eps(s_0)\}_{k \in J} + [-d^*_{\rm min}(s_0)/4,d^*_{\rm min}(s_0)/4]$, then from the inequality 
$\vert \xi_\eps(u)\vert  \leq e_\eps(u)$, from Corollary \ref{jamboncru} and from \eqref{crottin} of Proposition \ref{estimpar}, we derive that for $s_1\leq s_2$ in $I^\eps(s_0)$,   
\begin{equation}
\label{spain}
\left   \vert \mathfrak F_\eps (s_1, s_2, \chi) \right \vert 
 \leq C {d_{\rm min}^*(s_0)}^{-\omega}  \Vert \ddot \chi \Vert_{L^\infty(\R)}
(s_2-s_1). 
\end{equation}
Going back to \eqref{localizedenergy0}, and choosing the test function $\chi$ so that $\chi \equiv 1$ on $\IctL$ with compact support on $\ILL$, estimate \eqref{spain} combined with \eqref{mireille} yields in turn a first rough upper bound on the dissipation $\disL [s_1,s_2]$.  Combining these estimates we will obtain  
 
\begin{proposition}
\label{conspascites}
Under the assumptions of Corollary \ref{jamboncru}, for $s_1\leq s_2 \in I^\eps(s_0)$ we have 
\begin{equation}
\label{uniformita}
\vert a_k^\eps(s_1)-a_k^\eps(s_2) \vert \leq  C\left( {d_{\rm min}^*(s_0)}^{-(\omega+1)} 
\ (s_2-s_1)   + M_0\big( (\llogepsL)^{-\omega} d_{\rm min}^*(s_0) + \eps\big)\right).
\end{equation}
\end{proposition}
As an easy consequence, we deduce the following compactness property, setting
$$
I^*(s_0)=\Big( s_0, s_0+  \uprho_0 \big( \frac{ d_{\rm min}^*(s_0)}{8}\big)^{\omega+2}\Big)\cap (0,S).
$$ 
\begin{corollary}
\label{lartichaut} Under the assumptions of Corollary \ref{jamboncru},  
there exist a subsequence $(\eps_n)_{n \in \N}$ converging to $0$ such that for any $k \in J$ the function
$a_k^{\eps_n} (\cdot)$ converges uniformly on any compact interval of 
$I^*(s_0)$ to a Lipschitz continuous function $a_k(\cdot)$. 
\end{corollary}

%%%%%%%%%%%%%%%%%%%%%%%%%%%%%%%%%%%%%%%%%%%%%%%%%%%%%%%%%%%%%%%%%%%%%%
\subsection{Refined estimates off the front set and  the motion law }
\label{raffarinage}
%%%%%%%%%%%%%%%%%%%%%%%%%%%%%%%%%%%%%%%%%%%%%%%%%%%%%%%%%%%%%%%%%%%%%%
In order to derive the  \emph{precise motion law}, we have to provide an
accurate  asymptotic value for the discrepancy term \emph{off the front set}. 
In other words, for  a given index $k\in J$ we need to provide a  uniform limit of
the function  
$\eps^{-\omega} \xi_\eps$ near the points 
$$
a_{k+\frac{1}{2}}^\eps(s) \equiv \frac{a_k^\eps(s)+ a_{k+1}^\eps (s)}{2} \  \
{\rm and \  } a_{k-\frac{1}{2}}^\eps(s) \equiv \frac{a_{k-1}^\eps(s)+ a_{k}^\eps
(s)}{2}. 
$$  
We notice first that $\vm$ takes values close to $ \upsigma_{i(k+\frac 12 )}$ near 
$a_{k+\frac 1 2}^\eps(s)$. In view of estimate \eqref{crottin},  we introduce the functions
\begin{equation}
\label{Wm}
\mathfrak w_\eps(\cdot, s) = \mathfrak w_\eps^k (\cdot, s)= \vm -\upsigma_{i(k+\frac 12)} \   {\rm and  \  } 
\mathfrak W_\eps=\mathfrak W_\eps^k\equiv \eps^{-\frac{1}{\theta-1}}  \mathfrak w_\eps^k= \eps^{-\frac{1}{\theta-1}} \left(\vm -\upsigma_{i(k+\frac12)} \right). 
\end{equation}
As a consequence of inequality \eqref{crottin2} and Corollary \ref{jamboncru} we have the uniform bound:

\begin{lemma} 
\label{reset}
Under the assumptions of Corollary \ref{jamboncru}, we have 
\begin{equation}
\label{loindubord}
\vert  \mathfrak W_\eps(x, s) \vert  \leq C 
\big( d(x,s) \big)^{-\frac{1}{ \theta-1}}
\end{equation}
for any 
$x \in (a_k(s)+\udll, a_{k+1}(s)-\udll)$ and any $s\in I^\eps(s_0),$ where we have set $d(x,s):={\rm dist}(x,\{a_k^\eps(s),a_{k+1}^\eps(s)\})$ and
where $C>0$ depends only on $V$ and $M_0.$ 
Moreover, we also have
\begin{equation}
\label{presdubord}
\left\{
\begin{aligned}
-{\rm sign}(  \dagger_{k} )  \Wm \left(a_k^\eps(s)+\udll \right) &\geq \frac{1}{C}  \left( \udll\right)^{-\frac{1}{\theta-1}} \\
{\rm sign}(  \dagger_{k+1} )  \Wm \left(a_{k+1}^\eps(s)-\udll \right)&\geq \frac{1}{C} \left( \udll\right)^{-\frac{1}{\theta-1}}. 
\end{aligned}
\right.
\end{equation}
\end{lemma}
We describe next on a formal level how to obtain the desired asymptotics for
$\eps^{-\omega} \xi_\eps$, as $\eps\to 0$, near the point  
 $a_{k+\frac 12 } (s)$.  
Going back to the limiting points $\{a_k(s)\}_{k \in J} $ defined in Proposition
\ref{conspascites}, we  consider  the subset of $\R \times \R^+$   
\begin{equation}
\label{dopo}
\mathcal V_k(s_0)= \bigcup_{ s \in I^*(s_0) } \left(a_k(s), a_{k+1} (s)\right)\times \{s\}.
\end{equation}
It follows from the uniform bounds established in Lemma \ref{reset}, that,
passing possibly to a further subsequence, we may assume that 
$$
\Wmn  \rightharpoonup \Wms  {\rm \  in  \  } L_{\rm loc}^p (\mathcal
V_k(s_0)), {\rm \ for \   any \ } 1\leq p <\infty. 
$$ 
On the other hand, thanks to  estimate \eqref{loindubord}, for a given  point $(x, s) \in \mathcal V_k(s_0)$ we expand
${\rm (PGL)}_\eps$ near $(x, s)$ as 
\begin{equation}
\label{formaldehyde}
\eps^{\omega}\frac{\partial \Wm}{\partial s}-\frac{\partial^2 \Wm}{\partial
x^2}+2 \theta \lambda_{i(k+\frac12)} \Wm^{2 \theta-1}= O(\eps^{\frac{1}{\theta-1}}).
\end{equation}
Passing to the limit $\eps_n \to 0$, we expect that  for every $s\in I^*(s_0)$, $\Wms$ solves 
\begin{equation}
\label{formaleq}
\left\{
\begin{aligned}
-&\frac{\partial^2 \Wms}{\partial x^2}(s, \cdot)+2 \theta \lambda_{i(k+\frac12)}
\Wms^{2 \theta-1}(s, \cdot)=0  {\rm  \  on } \ (a_k(s), a_{k+1}(s)), \\ 
&\Wms(a_k(s))=-{\rm sign}(  \dagger_k)\infty  \ {\rm and} \ \   \
\Wms(a_{k+1}(s))={\rm sign}(  \dagger_k)\infty,  
\end{aligned}
\right. 
\end{equation} 
the boundary conditions being a consequence of the asymptotics \eqref{presdubord}. 
It turns  out, in view of Lemma \ref{veryordinary} of the Appendix, that the
\emph{boundary value problem}  \eqref{formaleq} has a \emph{unique}
solution. 
By scaling, and setting $r_k(s)=\frac 12 (a_{k+1}(s)-a_k(s))$ %and $d_k(s)=2 r_k(s)= a_{k+1}(s)-a_k(s)$
, we obtain
\begin{equation*}
\left\{
\begin{aligned}
& \Wms (x, s)=\pm r_k(s)^{-\frac{1}{ \theta-1}} \left (\lambda_{i(k+\frac12)} \right)^{-\frac{1}{2 (\theta-1)}}
\Ump\left( \frac{x-a_{k+\frac 12}}{r_k(s)} \right), \ &&{\rm if } \ \dagger_k= -\dagger_{k+1}, \\
&\Wms (x, s)=\pm r_k(s)^{-\frac{1}{ \theta-1}} \left (\lambda_{i(k+\frac12)}  \right)^{-\frac{1}{2 (\theta-1)}}
\Umps\left( \frac{x-a_{k+\frac 12}}{r_k(s)} \right), \ &&{\rm if } \ \dagger_k= \dagger_{k+1}, 
\end{aligned} 
\right.
\end{equation*}
where  $\Ump$ (resp. $\Umps$) are the unique solutions to  the problems
\begin{equation}
\label{umps0}
\left\{
\begin{aligned}
-&\mathcal U_{xx}+2 \theta\, \mathcal U^{2 \theta-1}=0 \qquad {\rm  \  on } \ (-1,+1), \\
&\mathcal U(-1)= +\infty  \  ({\rm resp.} \   \mathcal U(-1)= -\infty)  \  \,  {\rm and} \ \   \ \mathcal U(+1)=+\infty.
\end{aligned}
\right. 
\end{equation} 
Still on a formal level, we deduce therefore the corresponding values of the discrepancy
\begin{equation}
\label{formalite}
\left\{
\begin{aligned}
&\eps^{-\omega}\xi_\eps(\vm)\simeq  \xi(\Wms)=-
\lambda_{i(k+\frac12)}^{-\frac{1}{\theta-1}} r_k(s)^{-(\omega+1)}\mathcal A_\theta
\ &&{\rm if } \ \dagger_k= -\dagger_{k+1}, \\
&\eps^{-\omega}\xi_\eps(\vm)\simeq  \xi(\Wms)=
\lambda_{i(k+\frac12)}^{-\frac{1}{\theta-1}}
 r_k(s)^{-(\omega+1)}\mathcal B_\theta
 \ &&{\rm if } \ \dagger_k= \dagger_{k+1},
\end{aligned}
\right.
 \end{equation}
where the numbers $\mathcal A_\theta$ and $\mathcal B_\theta$ are positive, 
depend only on $\theta$, and correspond to the absolute value of the
discrepancy of $\Ump$ and $\Umps$ respectively.  
Notice  that the signs in \eqref{formalite} are different, the first case yields attractive forces
whereas the second yields repulsive ones. Inserting this relation in
\eqref{gratindo} and arguing as for \eqref{uniformita}, we will derive the motion
law. 
 
The previous formal discussion can be put on a sound mathematical ground, relying on 
comparison principles and the construction of appropriate upper and lower solutions (see Section \ref{raffinage}). 
This leads to the central result of this paper:   
  
\begin{proposition}
\label{firststep1}
Assume that conditions $({\rm H}_0)$ and $({\rm H}_1)$ are fulfilled. Let $0<S<S_{\rm max}$ be given and set 
$$
\rL_0:= 3 \max\left\{ |a_k^0|,\, 1\leq k \leq \ell_0\, ,\ (\frac{S}{\uprho_0})^\frac{1}{\omega+2}\right\}.
$$
%be sufficiently large\footnote{As mentioned previously the existence of such an $\rL$ is a direct consequence of Theorem \ref{mainbs2} and assumption $({\rm H}_1)$. If $S_{\rm max}$ is finite then we can actually take $S=S_{\rm max}$, but in general a
%finite $\rL$ requires a finite $S.$} so that all the points $a_k^0$ ($1\leq k \leq \ell_0$) are 
%contained in $\IL$ and so that 
%$$
%\mathfrak{D}_\eps(s) \cap \ILLLL \subset \IL \qquad \forall 0 \leq s \leq S.
%$$
Assume that ${\rm W PI}_\eps^{\rL_0} (\upalpha_1\eps, 0 )$ holds as well as \eqref{dminstar} at time $s=0$.  
Then $J(s) = \{1,\cdots,\ell_0\}$ and the functions $a_k^{\eps} (\cdot)$ are  well defined and
converge uniformly on any compact interval of   
   $(0, S)$   to  the solution $a_k(\cdot)$ of
\eqref{tyrannosaure} supplemented with the initial condition
$a_k(0)=a_k^0$.   
\end{proposition}

Notice that  the combination of assumptions $\wpi(\upalpha_1 \eps,0)$,
$({\rm H}_1)$ and  \eqref{dminstar} at $s=0$ implies  the multiplicity one
condition $({\rm H}_{\rm min})$.   Whereas the conclusion of Proposition
\ref{firststep1} is similar to the one of Theorem \ref{maintheo1},  
the assumptions of Proposition \ref{firststep1}  are
more restrictive. Indeed, on one hand we assume the well-preparedness
condition  $\wpi$, and on the other hand we  impose \eqref{dminstar} which is far
more constraining than $({\rm H}_{\rm min})$:  it excludes in particular  the
possibility of having  small pairs of fronts and anti-fronts.  Our next  efforts
are hence devoted to handle this type of situation:  
Proposition \ref{firststep1}, through rescaling arguments, will nevertheless be 
the main building block for that task.

\smallskip
 In order to prove Theorem \ref{maintheo1},  \ref{alignedsolution}  and
\ref{colissimo} we need to relax  the assumptions on  the initial data, in
particular  we need to analyze the behavior of  data with small  pairs of
\emph{fronts and anti-fronts}, and show that they are going to annihilate on a
short interval of time. For that purpose we will consider  the following
situation, corresponding to confinement of the front set at initial time.  
Assume that for a collection of points
$\{b^\eps_q\}_{q\in J_0}$ in $\R$ we have 
\begin{equation}
\label{confitdoie}
\mathfrak D_\eps(0) \cap \ILLLLL \subset  \underset{ q\in J_0} \cup  [b^\eps_q- r, b^\eps_q+r] \subset \IkzL 
\qquad \text{and}\qquad b^\eps_p-b^\eps_q \geq 3R \qquad \text{for}  \ p \neq q \in J_0, 
\end{equation}
for some $\upkappa_0\leq \frac12$ and $\upalpha_1\eps \leq r \leq R/2 \leq \rL/4.$ 
It follows from \eqref{theodebase1} that if $0\leq s\leq \uprho_0(R-r)^{\omega+2}$ then 
\begin{equation*}
\mathfrak D_\eps(s) \cap \ILLLL \subset  \underset{ k\in J_0} \cup  (b^\eps_k-R,
b^\eps_k+R) \subset \IdkzL, \ \   {\rm where \ the \ union \ is \ disjoint}. 
\end{equation*}
Consider next   $0\leq s\leq \uprho_0(R-r)^{\omega+2}$ such that
$\wp(\upalpha_1 \eps, s)$ holds, so that the front points $\{a^\eps_k(s)\}_{k
\in J(s)}$ are well-defined.  For $q \in J_0$,  consider $J_q(s)=\{k \in J(s),
a^\eps_k (s) \in (b^\eps_q-R, b^\eps_q+R)\}$, set  $\ell_q=\sharp  J_q $, and   write
$J_q(s)=\{k_{q} , k_{q+1},\cdots,  k_{q +\ell_q-1}\}, $ 
 where $k_1=1$, and  $k_q=\ell_1+\cdots +\ell_{q-1}+1$, for $q \geq 2$. 
Our next result shows that, after a  small time, only the repulsive forces
survive at the  scale given by $r$, provided the different lengths are sufficiently 
distinct.  

\begin{proposition}
\label{nettoyage}
There exists positive constants $\upalpha_*$ and $\uprho_*$, depending only on $V$ and $M_0$, such that if \eqref{confitdoie} holds and 
\begin{equation}\label{poolish}
\upkappa_0^{-1}\geq \upalpha_*, \qquad r \geq \upalpha_* \eps\left(\frac{\rL}{\eps}\right)^{\frac{2}{\omega+2}}, \qquad R \geq \upalpha_* r,
\end{equation}
then at time
$$
s_r= \uprho_* r^{\omega+2} 
$$ 
condition $\WPL(\alpha_1\eps, s_r)$ holds and, for any $q\in J_0$ and any $k, k' \in J_q
(s_r)$ we have $\dagger_k(s_r)=\dagger_k' (s_r)$ or equivalently for any $k \in
J_q(s_r)\setminus  
\{k_{q}(s_r)+\ell_q(s_r)-1\}$, we have
\begin{equation}
\label{clearance}
\epsilon_{k+\frac 12}(s_r)=\dagger_k(s_r) \dagger_{k+1}(s_r)=+1.
\end{equation} 
Moreover, we have  
\begin{equation}
\label{clarence}
\dmineL(s_r) \geq r,  %\ {\rm and} \ \dmines(s_r) \geq\frac 3 2  R. 
\end{equation}
and if $\sharp J_q(s_r) \leq 1$ for every $q\in J_0$, then we actually have $\dmineL(s_r) \geq R.$
\end{proposition}
The proofs of  Theorems \ref{maintheo1},  \ref{alignedsolution} and
\ref{colissimo} are then  deduced from Propositions \ref{firststep1} and
\ref{nettoyage}.

\medskip

The  paper is organized as follows. We describe in Section \ref{cavenecadas} some
properties of stationary fronts, as well as for solutions to  some perturbations
of the stationary equations.  
In Section \ref{regupara} we describe several properties related to the
well-preparedness assumption $\WPL$, in particular the quantization of the
energy, how it relates to dissipation, and its numerous implications for the
dynamics. We provide in particular the proofs to Proposition \ref{parareglo}, Proposition \ref{parareglo2} and Corollary \ref{jamboncru}.  In
Section \ref{compactness}, we prove  the compactness results stated in Proposition \ref{conspascites}
and Corollary \ref{lartichaut}.  Section \ref{raffinage},
provides an expansion of the discrepancy term off the front set, %is  the
%core of our paper\footnote{and also essentially the only place in the paper where PDE
%methods are used.}, at least from a technical point of view: it is also the part
 from a technical point of view it is the place where the analysis differs most from the non-degenerate case.
 Based on this
analysis, we show in Section \ref{motion} how the motion law follows from
prepared data  establishing the proof to Proposition \ref{firststep1}. In
Section \ref{clearstream} we analyze the clearing-out of small pairs of front-antifront and more
generally we present the proof of Proposition \ref{nettoyage}. Finally, in section \ref{extended} we present
the proofs of the main  theorems, namely Theorem \ref{maintheo1}, \ref{alignedsolution} and \ref{colissimo}.  
Several results concerning the first or second order differential equations involved in
the analysis of this paper are given in  separate appendices, in particular the proof of Proposition \ref{getrude}.  

%%%%%%%%%%%%%%%%%%%%%%%%%%%%%%%%%%%%%%%%%%%%%%%%%%%%%%%%%%%%%%%%%%%%%%%%%%%%%%%%%%%%%%%%%%%%%%
\numberwithin{theorem}{section} \numberwithin{lemma}{section}
\numberwithin{proposition}{section} \numberwithin{remark}{section}
\numberwithin{corollary}{section}
\numberwithin{equation}{section}

%%%%%%%%%%%%%%%%%%%%%%%%%%%%%%%%%%%%%%%%%%%%%%%%%%%%%%%%%%%%%%%%%%%%%%%%%%%%%%%%%%%%%%%%%%%%%%%%%%%%
\section{Remarks on stationary solutions}
\label{cavenecadas}

%%%%%%%%%%%%%%%%%%%%%%%%%%%%%%%%%%%%%%%%%%%%%%%%%%%%%%%
\subsection{ Stationary solutions on $\R$ with vanishing discrepancy}
%%%%%%%%%%%%%%%%%%%%%%%%%%%%%%%%%%%%%%%%%%%%%%%%%%%%%%%
Stationary solutions  are described using  the method of separation of variable.
For $u$  solution to \eqref{stat}, we multiply 
\eqref{stat} by $u$ and  verify   that $\xi$  is constant. We restrict
ourselves to solutions with vanishing discrepancy 
\begin{equation}
\label{vanishing}
\xi=  \frac{1}{2}{ \dot u^2}-V(u)=0,
\end{equation}
and solve equation \eqref{vanishing} by  separation of variables. Let
$\gamma_i$  be defined  on  $(\upsigma_i, \upsigma_{i+1})$  by   
 \begin{equation}
 \label{separation}
 \upgamma_i(u)=\int_{z_i}^u \frac{ ds}{\sqrt{2 V(s)}},  {\rm  for \ } u \in  (\upsigma_i, \upsigma_{i+1}),
 \end{equation}
where we recall that $z_i$ is a fixed maximum point of $V$ in the interval
$(\upsigma_i, \upsigma_{i+1})$.  The  map $\upgamma_i$ is one-to-one  from
$(\upsigma_i, \upsigma_{i+1})$ to $\R$, so that we may define its inverse map
$\zeta_i^+: \R \to (\upsigma_i, \upsigma_{i+1}) $ by  
\begin{equation}
\label{inverse}
\zeta_i^+(x)= \upgamma_i^{-1}(x) {\rm \ \,  as \ well \ as \,\ }  \zeta_i^-(x)= \gamma_i^{-1}(-x) \ {\rm \ for \ } x \in \R. 
\end{equation}
In view of the definition \eqref{inverse}, we have
$\zeta_i^\pm(0)=z_i, \ \  {\zeta_i^{+}}'(0)=\sqrt{2V(z_i)}>0, $
whereas a  change of variable shows that  $\zeta_i$  has finite energy  given by the formula \eqref{quantenergy}. 
We   verify that  $\zeta_i^+\left(\frac{\cdot}{\eps}\right)$ and
$\zeta_i^-\left (\frac{\cdot}{\eps}\right)$ solve  
\eqref{vanishing} and hence \eqref{stat}.   The next elementary result then directly
follows from uniqueness in ode's:
  
\begin{lemma}
\label{tertre}
Let $u$ be a solution to \eqref{stat}   such that \eqref{vanishing} holds, and
such that  $u(x_0)\in (\upsigma_i, \upsigma_{i+1})$, for some $x_0\in \R$, and
some $i\in1, \cdots \q-1$. Then,  there exists  $a\in \R$ such that  
$\displaystyle{u(x)= \zeta_i^+\left({x-a} \right)  \hbox{ or } u(x)=
\zeta_i^-\left (x-a\right), \forall x \in \R. }$ 
\end{lemma}
We provide a few  simple  properties of the functions $\zeta_i^\pm$ which enter directly in our arguments. 
We expand  $V$ near $\upsigma_i$ for $u \geq \upsigma_i$ as
\begin{equation*}
\sqrt{V(u)}=\sqrt{\lambda_i}(u-\upsigma_i)^{\theta} (1+ O(u-\upsigma_i)),
\quad\text{as } u \to \sigma_i.
\end{equation*}
Integrating, we are led to the expansion
\begin{equation*}
\begin{aligned}
\upgamma_i(u)= 
-\frac{\theta-1}{\sqrt{2\lambda_i}}(u-\upsigma_i)^{-\theta+1}(1+O(u-\upsigma_i)),
\quad\text{as } u \to \sigma_i,
\end{aligned}
\end{equation*}
and therefore also to the expansions
$$
\zeta_i^\pm (x) = \upsigma_i +
\left(\frac{\sqrt{2\lambda_i}|x|}{\theta-1}\right)^{-\frac{1}{\theta-1}}(1+o(1)),
\quad\text{ as } x \to \mp \infty.
$$
Similarly, 
$$
\zeta_i^\pm (x) = \upsigma_{i+1} -
\left(\frac{\sqrt{2\lambda_{i+1}}|x|}{\theta-1}\right)^{-\frac{1}{\theta-1}}(1+o(1)),
\quad\text{ as } x \to \pm \infty,
$$
and corresponding asymptotics for the derivatives can be derived as
well (e.g. using the fact that the discrepancy is zero).  

For $0<\eps<1$ given, and $i=1, \cdots, q-1$, consider the scaled function
$\displaystyle{\zeta_{i, \eps}^\pm=\zeta_i^\pm\left (\frac{\cdot}{\eps}\right)}$  
which is a solution  to $$ -u_{xx}+ \eps^{-2} V' (u)=0, $$
hence a stationary solution to $({\rm PGL})_\eps$.   
Straightforward computations based on the previous expansions show that
\begin{equation}
\label{leonardo0}
\left \{
\begin{aligned}
 & e_\eps \left(\zeta_{i, \eps}^\pm\right)(x)=
(2\lambda_i)^{-\frac{1}{\theta-1}}(\theta-1)^\frac{2\theta}{\theta-1} \tfrac{1}{\eps}
\left\vert
\tfrac{x}{\eps}\right\vert^{-(\omega+1)}  + 
  {\underset {\tfrac{x}{\eps} \to  \mp\infty } {o}}  \left ( \tfrac{1}{\eps} 
\left|\tfrac{ x}{\eps}\right|^{-(\omega+1)}\right)\\
& e_\eps \left(\zeta_{i, \eps}^\pm\right)(x)=
(2\lambda_{i+1})^{-\frac{1}{\theta-1}}(\theta-1)^\frac{2\theta}{\theta-1}
\tfrac{1}{\eps} \left| \tfrac{x}{\eps}\right|^{-(\omega+1)}  + 
  {\underset {\tfrac{x}{\eps} \to  \pm\infty } {o}}  \left ( \tfrac{1}{\eps}
\left|\tfrac{ x}{\eps}\right|^{-(\omega+1)}\right)
\end{aligned}
\right.
\end{equation}    
with 
$\omega$ defined in  \eqref{picomega}. 
Hence there is some constant $C>0$ independent of $r$ and $\eps$   such that
\begin{equation}
\label{leonardo2}
\S_i\geq \int_{-r}^{r} e_\eps \left(\zeta_{i, \eps}^\pm\right) dx \geq \S_i- C\left(
\frac{ \eps}{r}\right)^{\omega}.
\end{equation}

%%%%%%%%%%%%%%%%%%%%%%%%%%%%%%%%%%%%%%%%%%%%%%%%%%%%%%%  
\subsection{On the energy of  chains  of stationary solutions}
\label{patdef}
%%%%%%%%%%%%%%%%%%%%%%%%%%%%%%%%%%%%%%%%%%%%%%%%%%%%%%%
If $u$ satisfies condition  $\WPI (\updelta)$ and $\hzero$, we set 
\begin{equation}
\label{hereke0}
\matfE^\rL(u)= {\underset {k\in J}\sum}\S_{i(k)} \ \, 
{\rm and \ }  \mathcal E_\eps^\rL(u)=\int_{\ILL} e_\eps(u(x))dx. 
\end{equation}
   
\begin{proposition}
\label{globalmoquette}
We have
\begin{equation}
 \label{hereke1} 
 \left\{
\begin{aligned}
& \mathcal E_\eps^\rL(u)  \geq  \matfE^\rL(u)- {\rm C_f}M_0\left( \frac{
\eps}{\updelta}\right)^{\omega} 
&&{\rm  \ if \ } \wpi(\ud) \ {\rm holds},   \\ 
&  \mathcal E_\eps^\rL(u)  \leq  \matfE^\rL(u) + ({\rm C_w +C_f}) M_0\left( \frac{ 
\eps}{\updelta}\right)^{\omega} &&{\rm \ if  \ \, } \wp(\ud) 
{\rm \ holds.  \ }
\end{aligned}
\right.
\end{equation}
Moreover, for any  smooth function $\chi$ with compact support in $\ILL$ we
have 
\begin{equation}\label{hereke1bis}
 \left \vert \mathcal I_\eps (\chi)-   \underset {k
\in J} \sum \chi (a_k)
\S_{i(k)}  \right\vert  \leq  ({\rm C_w +C_f}) M_0 \left( \left(\frac{
\eps}{\updelta}\right)^{\omega}
  \Vert \chi\Vert_{\infty}  + \eps \Vert \chi'\Vert_{\infty} \right),
 \qquad \text{if } \wp(\ud) \text{ holds,}
\end{equation}
where $\mathcal I_\eps( \chi)= \int_{\ILL} e_\eps (u)\chi(x)dx.$ The constant ${\rm C_f}$
which appears in \eqref{hereke1} and \eqref{hereke1bis} only depends on $V,$ and the constant
${\rm C_w}$ appears in the definition of condition $\wp.$
\end{proposition}
\begin{proof} 
We estimate the integral of  $\vert e_\eps
(u)-e_\eps (\zeta_{i(k)}^{\dagger_k}(\cdot-a_k)) \vert$ on $I_k$ as 
\begin{equation*}
%\begin{aligned}
\frac{\eps}{2}  \int_{I_k} \vert  \dot u^2-  (\dot \zeta_{i(k),
\eps}^{\dagger_k}(\cdot-a_k))^2 \vert dx
%& 
\leq \eps \Vert   \dot u -  \dot \zeta_{i(k),
\eps}^{\dagger_k}(\cdot-a_k)\Vert_{L^\infty (I_k)} 
 \left  [\mathcal E_\eps (u)^\frac12 + \mathcal E_\eps ( \zeta_{i(k),
\eps}^{\dagger_k})^\frac12 \right] \sqrt{\frac {\ud}{\eps}} 
%\end{aligned}
\end{equation*}
and likewise we obtain
\begin{equation*}
\eps^{-1} \int_{I_k} \vert   V(u)- V(\zeta_{i(k), \eps}^{\dagger_k}(\cdot-a_k))\vert
\leq 
 C \frac {\ud}{\eps}\Vert   u-  \zeta_{i(k),
\eps}^{\dagger_k}(\cdot-a_k)\Vert_{L^\infty (I_k)}. \end{equation*}
It suffices then to invoke $\wpi(\ud)$ and $\wpo(\ud)$ as well as the decay
estimates \eqref{leonardo2} to derive \eqref{hereke1}, using the fact that since
$\ud \geq \alpha_1 \eps,$ negative exponentials are readily controlled by negative
powers. Estimate \eqref{hereke1bis} is derived in a very similar way, the error in
$\eps \Vert \chi'\Vert_{\infty}$ being a consequence of the approximation of $\int
\chi e_\eps(\zeta_{i(k), \eps}^{\dagger_k}(\cdot-a_k))$ by $\chi(a_k)\S_{i(k)}.$
\end{proof}
%
%\begin{remark}
%\label{franzy}
%{\rm 
%The proof above actually shows that, if $\WPI(\updelta)$ holds then for some ${\rm
%C}_{\rm a}>0$
%$$ 
%\left \vert \underset{ k\in J}  \sum \int_{I_k} e_\eps(u) dx   - \matfE^\rL(u)
%\right  \vert 
%\leq
%{\rm C}_{\rm a}M_0\left( \frac{ \eps}{\updelta}\right)^{\omega}.
% $$
% The value of the constant ${\rm C}_{\rm a}$ depending possibly of our choice of
%constant $\uprho_{\rm w}$, but not on ${\rm C}_{\rm w}$, which has not been used
%in the proof above. 
% }
%\end{remark}
%  
This result shows that, if $\updelta$ is sufficiently large, the energy is
close to a set of discrete values, namely the finite sums of $\S_k.$ We will
therefore refer to this property as \emph{ the quantization of the energy}, it will
play an important role later when we will obtain estimates on the dissipation rate
of energy.

%%%%%%%%%%%%%%%%%%%%%%%%%%%%%%%%%%%%%%%%%%%%%%%%%%
%%%%%%%%%%%%%%%%%%%%%%%%%%%%%%%%%%%%%%%%%%%%%%%%%%
\subsection{Study of the perturbed stationary equation}
\label{perturbed}
%%%%%%%%%%%%%%%%%%%%%%%%%%%%%%%%%%%%%%%%%%%%%%%%%%%%%%%
Consider a  function $u$ defined on $\R$
satisfying the perturbed differential equation
\begin{equation}
\label{odepertu}
u_{xx}=\eps^{-2} V'(u)+f, 
\end{equation}
where  $f \in L^2(\R)$, and the energy bound $\hzero$.
We already know, thanks to Lemma \ref{tertre} that 
if $f=0$ then $u$ is of
the form  $\zeta_{i, \eps}^\pm(\cdot-a)$. Our results below, summarized here in
loose terms, show that if $f$ is sufficiently small on some sufficiently large
interval, then $u$ is close to a chain of translations of the functions
$\zeta_{i,\eps}^\pm$ suitably glued together on that interval. 

%%%%%%%%%%
%%%%%%%%%%
%%%%%%%%%%
Following the approach of \cite{BOS8}, we first recast equation \eqref{odepertu} 
as a system of two differential 
equations of first order. For that purpose, we set 
  $w=\eps u_x$
   so that  \eqref{odepertu} is equivalent to the system
   \begin{equation*}
%   \label{prems}
%  \begin{split}
   u_x=\frac{1}{\eps}w  \  {\rm and} \ 
   w_x=\frac{1}{\eps} V'(u)+\eps f, 
%      \end{split}
    \end{equation*}
which we may write in  a more  condensed form  as
\begin{equation}
\label{eqvect}
U_x=\frac{1}{\eps} G(U)+\eps F \ {\rm  on } \  \R, 
\end{equation}
where we have set  $U(x)=(u(x), w(x))$ and $F(x)=(0, f(x))$,
and where $G$ denotes the vector field $G(u,w)=(w,
V'(u))$. Notice that the energy bound $\hzero$ and assumption $(A_3)$
together imply a global $L^\infty$ bound on $u$. In turn, this $L^\infty$ bound 
imply a Lipschitz bound, denoted $C_0$, for the
nonlinearity $G(u,w)$.

\begin{lemma}\label{lem:compgron}
Let $u_1$ and $u_2$ satisfy \eqref{odepertu} with forcing terms $f_1$ and $f_2$,
and assume that both satisfy the energy bound $\hzero$. Denote by 
$U_1,U_2,F_1,F_2$ 
the corresponding solutions and forcing terms of \eqref{eqvect}. 
Then, for any $x,x_0$ in some arbitrary interval $I$, 
\begin{equation}
\label{massa}
\vert (U_1-U_2)(x) \vert \leq \left( \vert (U_1-U_2)(x_0)\vert  +
\frac{\eps^\frac32}{\sqrt{2 C_0}} \Vert F_1-F_2\Vert_{L^2(I)}\right)
\exp\left(\frac{ C_0  \vert x-x_0 \vert }{\eps}\right).
\end{equation} 
\end{lemma}
\begin{proof} 
Since $(U_1-U_2)_x= G(U_1)-G(U_2)+\eps (F_1-F_2)$ we obtain the inequality
$$\vert (U_1-U_2)_x\vert \leq \frac{C_0}{\eps} \vert  U_1-U_2\vert +\eps
\vert F_1-F_2 \vert.$$ 
It follows from Gronwall's inequality that   
\begin{equation*}\vert (U_1-U_2)(x)\vert \\
\leq   \exp \left( \tfrac{ C_0 \vert x-x_0 \vert }{\eps}\right)\vert (U_1-U_2)(x_0)\vert + 
 | \int_{x_0}^x \eps \vert (F_1-F_2)(y)\vert  \exp \left( \tfrac{C_0  \vert
y-x_0\vert }{\eps}  \right) dy|.
\end{equation*}
Claim \eqref{massa}  then follows from the Cauchy-Schwarz inequality.
\end{proof}

We will combine the previous lemma with

\begin{lemma}
\label{discrepok}
Let $u$ be a solution of \eqref{odepertu}  satisfying $\hzero$. 
Then
$$
\sup_{x,y \in I } \vert \xi_\eps(u)(x)-\xi_\eps(u)(y)\vert \leq \sqrt{2
M_0} \eps^{\frac{1}{2}}  \Vert f \Vert_{L^2(I)}, 
$$
where $I \subset \R$ is an arbitrary interval.
\end{lemma}
\begin{proof}
This is a direct consequence of the equality
$\displaystyle{ 
\frac{d}{dx} \xi_\eps(u) = \eps f \frac{d}{dx}u,}
$
the Cauchy-Schwarz inequality, and the definition of the energy.  
\end{proof}

\begin{lemma}\label{prop:wpiok} Let $u$ be a solution of \eqref{odepertu} 
satisfying $\hzero$. Let $L>0$ and assume that
$$
\mathcal{D}(u) \cap \ILL \subseteq \IL.
$$ 
There exist a constant $0<\kappa_{\rm w}<1$, depending only on $V$, such that if
\begin{equation}\label{eq:ptgici}
M_0 \frac{\eps}{\rL} + M_0^\frac12 \eps^\frac32 \|f\|_{L^2(\ItdL)} \leq
\kappa_{\rm w}, 
\end{equation}
then the condition ${\rm WPI}_\eps^\rL(\delta)$ holds where
\begin{equation}\label{eq:wpidelta} 
\frac{\delta}{\eps} := - \frac{2}{\uprho_{\rm w}}  \log \big( M_0 \frac{\eps}{\rL} + M_0^\frac12 
\eps^\frac32 \|f\|_{L^2(\ItdL)} \big),
\end{equation}
and where the constant $\uprho_{\rm w}$ depends only on $M_0$ and $V.$ Moreover, $\kappa_{\rm w}$ is sufficiently small so that $2|\!\log \kappa_{\rm w}|/\uprho_{\rm w} 
\geq \alpha_1,$ where $\alpha_1$ was defined in \eqref{eq:alpha1grand}.   
\end{lemma}
\begin{proof}
 If $\mathcal{D}(u)\cap \ILL = \emptyset$ then there is nothing to prove. If
not, we first claim that there exist a point $a_1 \in \IL$ such that $u(a_1)=
z_{i(1)}$ for some $i(1)\in \{1,\cdots,\q-1\}.$ Indeed, if not, and since the
endpoints of $\ILL$ are not in the front set, the function $u$ would have a
critical point with a critical value in the complement of $\cup_j
B(\upsigma_j,\mu_0).$  At that point, the discrepancy would therefore be larger
than $C/\eps$ for some constant $C>0$ depending only of $V$ (through the choice
of $\mu_0$). On the other hand, since $|\xi_\eps| \leq e_\eps$, by averaging
there exist at least one point in $\ItdL$ where the discrepancy of $u$ is smaller
in absolute value than $M_0/(3\rL).$ Combined with the estimate of Lemma
\ref{discrepok} on the oscillation of the discrepancy, we hence derive our first
claim, provided $\kappa_{\rm w}$ in \eqref{eq:ptgici} is chosen sufficiently small.  
Wet set $\dagger_1 = {\rm sign}(u'(a_1))$, $u_1=u$ and $u_2=\zeta^{\dagger_1}_{i(1),\eps}(\cdot-x_1)).$ Since 
$$
V(u_1(a_1))= V(u_2(a_1))=V(z_{i(1)}),$$
and since 
$$
|\xi_\eps(u_1)(a_1) - \xi_\eps(u_2)(a_1)| = | \xi_\eps(u_1)(a_1)|  \leq M_0/(3\rL) + \sqrt{2M_0}\eps^\frac12 \|f\|_{L^2(\ItdL)},
$$
we obtain
$$
\left| \eps (u_1')^2(a_1) - \eps (u_2')^2(a_1) \right| \leq M_0/(\rL) + 2\sqrt{2M_0}\eps^\frac12 \|f\|_{L^2(\ItdL)}.
$$
Since also 
$$
|u_1'(a_1) + u_2'(a_1)| \geq |u_2'(a_1)| =|\sqrt{\frac{2V(z_{i(1)})}{\eps^2}}|\geq C/\eps,
$$
it follows that
$$
\left| \eps (u_1' - u_2')(a_1) \right| \leq C\left(  M_0\frac{\eps}{\rL} + \sqrt{M_0}\eps^\frac32 \|f\|_{L^2(\ItdL)}\right),
$$
for a constant $C>0$ which depends only on $V.$ We may then apply Lemma \ref{lem:compgron} to $u_1$ and $u_2$ with the choice $x_0=a_1,$ and for which we thus have, with the notations of Lemma \ref{lem:compgron}, 
$$
\vert (U_1-U_2)(x_0)\vert \leq C\left(  M_0\frac{\eps}{\rL} + \sqrt{M_0}\eps^\frac32 \|f\|_{L^2(\ItdL)}\right).
$$
Estimate \eqref{massa} then yields \eqref{bugs} on $I_1=[a_1-\delta,a_1+\delta],$ for the choice of $\delta$ given by \eqref{eq:wpidelta} with $\uprho_{\rm w}=4(C_0+1)$, where $C_0$ depends only on $M_0$ and $V$ and was defined above Lemma \ref{lem:compgron}.

If $\mathcal{D}(u)\cap (\ItdL\setminus [a_1-\delta,a_1+\delta]) = \emptyset,$ we are done, and if not we may repeat the previous construction (the boundary points of $[a_1-\delta,a_1+\delta]$ are not part of the front set), until after finitely many steps we cover the whole front set. 
\end{proof}

\smallskip  
We turn to the outer condition\footnote{for which several 
adaptations have to be carried out compared to the non-degenerate case.}
$\WPOL$.

\begin{lemma}\label{lem:dansletrou}
Let $u$ be a solution of \eqref{odepertu}  verifying
$\hzero$, and assume that for some index $i\in \{1,\cdots,\q\}$
$$
u(x) \in B(\upsigma_i,\mu_0) \quad \forall x\in A,
$$ 
where $A$ is some arbitrary bounded interval.
Set $R={\rm length}(A)$, let $0<\rho<R$, and set $B=\{x\in A\ | \ {\rm dist}(x,A^c)>\rho\}.$  Then we have the estimate 
$$
\mathcal{E}_\eps(u,B) \leq {\rm C_o} \left(
\mathcal{E}_\eps(u,A\setminus B)^\frac{1}{\theta}
\big(\tfrac{\eps}{\rho}\big)^{1+\frac{1}{\theta}} +
R^\frac32 M_0^\frac{1}{2\theta}
\big(\tfrac{\eps}{R}\big)^{1+\frac{1}{2\theta}}\|f\|_{L^2(A)}\right),
$$
where the constant ${\rm C_o}$ depends only on $V.$
\end{lemma}
\begin{proof}
Let $0\leq \chi \leq 1$ be a smooth cut-off function with compact support in $A$ and
such
that $\chi \equiv 1$ on $B$ and $|\chi'|
\leq 2/\rho$ on $A.$ We multiply \eqref{odepertu} by $\eps(u-\upsigma_i) 
\chi^2$ and integrate on $A.$ This leads to
$$
\int_A \eps u_x^2 \chi^2 + \frac{1}{\eps} V'(u)(u-\upsigma_i)\chi^2 = \int_{A\setminus B}
2\eps u_x(u-\upsigma_i)\chi \chi' - \int_A \eps f(u-\upsigma_i)\chi^2.
$$
We estimate the first term on the right-hand side above by
\begin{equation*}
\begin{split}
\big|\int_{A\setminus B}
2\eps u_x(u-\upsigma_i)\chi \chi' \big| & \leq \big( \int_A \eps u_x^2\chi^2\big)^\frac12 
\big( \int_{A\setminus B} \eps^\theta (u-\upsigma_i)^{2\theta} \big)^\frac{1}{2\theta} 
\big( \int_{A\setminus B} |2\chi'|^\frac{2\theta}{\theta-1}\big)^\frac{\theta-1}{2\theta}\\
& \leq \frac{1}{2} \int_A \eps u_x^2\chi^2 + \frac12 \eps^{1+\frac{1}{\theta}}\big(\int_{A \setminus B} \frac{2}{\lambda_i}
e_\eps(u)\big)^\frac{1}{\theta} \left(\frac{4}{\rho}\right)^2 (2\rho)^\frac{\theta-1}{\theta}
\\
& \leq \frac{1}{2} \int_A u_x^2\chi^2 +
16 \lambda_i^{-\frac{1}{\theta}} \left(\frac{\eps}{\rho}\right)^{1+\frac{1}{\theta}} \mathcal{E}_\eps(u,A\setminus B)^\frac{1}{\theta},
\end{split}
\end{equation*}
where we have used \eqref{eq:souscontrole} and the fact that ${\rm length}(A\setminus B)=2\rho.$ Similarly we estimate
\begin{equation*}
\begin{split}
\big|\int_A
\eps f (u-\upsigma_i)\chi^2\big| &\leq \eps \|f\|_{L^2(A)} 
\big( \int_A (u-\upsigma_i)^{2\theta} \big)^\frac{1}{2\theta} 
 R^\frac{\theta-1}{2\theta}\\
& \leq  \eps^{1+\frac{1}{2\theta}}  \|f\|_{L^2(A)} 
(\frac{2}{\lambda_i})^{-1}M_0^\frac{1}{2\theta} R^\frac{\theta-1}{2\theta}.
\end{split}
\end{equation*}
Also, by \eqref{eq:souscontrole} we have
$$
\int_A \frac{1}{\eps} V'(u)(u-\upsigma_i)\chi^2 \geq \theta  \int_B
\frac{1}{\eps} V(u). 
$$  
Combining the previous inequalities the conclusion follows.
\end{proof}

Combining Lemma \ref{prop:wpiok} with Lemma \ref{lem:dansletrou}
we obtain 
\begin{proposition}
\label{vaderetro0}
Let $u$ be a solution to \eqref{odepertu} satisfying assumption $\hzero$, and such that
$
\mathcal{D}(u)\cap \ILLL \subset \IL.
$
There exist positive constants\footnote{Recall that ${\rm C}_{\rm w}$ enters in the definition of condition $\wp$. A parameter named ${\rm C_w}$ already appears in the statement of Proposition \ref{globalmoquette} above: We impose that its updated value here is be larger than its original value in Proposition \ref{globalmoquette} (and Proposition \ref{globalmoquette} remains of course true with this updated value!).} ${\rm C}_{\rm w}$ and $\alpha_1$, depending only on $M_0$ and $V$, such that if $\alpha\geq \upalpha_1$ and if
\begin{enumerate}
\item $\displaystyle{ M_0 \frac{\eps}{\rL} \leq \frac12 \exp(-\tfrac{\uprho_{\rm w}}{2}\alpha),}$
\item $\displaystyle{ \|f\|_{L^2(\ILLL)} \leq \frac12 M_0^{-\tfrac12} \eps^{-\tfrac32} \exp(-\tfrac{\uprho_{\rm w}}{2}\alpha),}$
\item $\displaystyle{ \|f\|_{L^2(\ILLL)} \leq\frac{{\rm C_w}}{2{\rm C_o}} M_0^{1-\tfrac{1}{2\theta}} \left( \frac{\eps}{\rL}\right)^{-1-\tfrac{1}{2\theta}}\rL^{-\tfrac32} \alpha^{-\omega}},$
\end{enumerate}
then $\wp(\alpha\eps)$ holds.
\end{proposition}
\begin{proof}
%Many inequalities which we will prove here below hold provided $\eps$ is sufficiently small, depending only
%on $\rL$ and $M_0;$ we will not repeat this condition each time. 
Direct substitution shows that assumptions $\mathit 1.$ and $\mathit 2.$ imply condition \eqref{eq:ptgici}, provided $\upalpha_1$ is choosen sufficiently large, 
and also imply condition $\wpi(\delta)$ for some $\delta\geq \alpha\eps$ given by \eqref{eq:wpidelta}. It remains to consider $\wpo(\alpha\eps).$
We invoke Lemma \ref{lem:dansletrou} on each of the
intervals $A=(a_k+\frac12 \alpha\eps,a_{k+1}-\frac12\alpha\eps)$, taking $B=(a_k+\alpha\eps,a_{k+1}-\alpha\eps).$ In view of $\wpi(\alpha\eps)$ and \eqref{leonardo2}, we obtain 
$$
\mathcal{E}_\eps(u,A\setminus B) \leq C\alpha^{-\omega},
$$
and therefore
$$
\mathcal{E}_\eps(u,A\setminus B)^\frac{1}{\theta} \alpha^{-1-\frac{1}{\theta}} \leq C \alpha^{-\omega},
$$
where $C$ depends only on $V.$ Also, in view of assumption $\mathit 3.$ we have
$$
{\rm C_o} \sum_k R^\frac32 M_0^\frac{1}{2\theta}
\big(\tfrac{\eps}{R}\big)^{1+\frac{1}{2\theta}}\|f\|_{L^2(A)} \leq 
{\rm C_o} L^\frac32 M_0^\frac{1}{2\theta}
\big(\tfrac{\eps}{L}\big)^{1+\frac{1}{2\theta}}\|f\|_{L^2(\ILLL)}
\leq \frac12 {\rm C_w} M_0 \alpha^{-\omega},
$$
provided $\alpha_1$ is sufficiently large (third requirement).
It remains to estimate $e_\eps(u)$ on the intervals $(-2\rL,a_1)$ and $(a_\ell,2\rL).$ 
We first use Lemma \ref{lem:dansletrou} with $A=(-3\rL,-\rL)$ (resp. $A=(\rL,3\rL)$ and
$B=(-\frac52\rL,-\frac32\rL)$ (resp. $B=(\frac32\rL,\frac52\rL)$). This yields, using the trivial bound
$\mathcal{E}_\eps(u,A\setminus B) \leq M_0,$ the estimate
\begin{equation}\label{eq:surlebord0}
\mathcal{E}_\eps(u,{\rm I}_{\frac52\rL} \setminus {\rm I}_{\frac32 \rL}) \leq C \left( M_0^\frac{1}{\theta}  \big(\frac{\eps}{\rL}\big)^{1+\frac{1}{\theta}} + M_0^\frac{1}{2\theta}  \big(\frac{\eps}{\rL}\big)^\frac{1}{2\theta}\right) \leq C \alpha^{-\omega}, 
\end{equation}
in view of ${\mathit 1.}$ and provided $\alpha_1$ is sufficiently large. 
We apply one last time Lemma \ref{lem:dansletrou},  with $A=(-2\rL-\frac12\alpha\eps,a_1-\frac{1}{2}\alpha\eps)$ (resp. $A=(a_\ell+\frac12 \alpha\eps,2\rL+\frac12 \alpha\eps)$) and $B=(-2\rL,a_1-\alpha\eps)$ (resp. $B=(a_\ell+\alpha\eps,2\rL)$). Since $A\setminus B \subset  {\rm I}_{\frac52\rL} \setminus {\rm I}_{\frac32 \rL},$ it follows from \eqref{eq:surlebord0} and Lemma \ref{lem:dansletrou}, combined with our previous estimates, that condition $\wpo(\alpha\eps)$ is satisfied provided we choose ${\rm C}_{\rm w}$ sufficiently large.
\end{proof}      

\begin{remark}
 Notice that condition ${\mathit 1.}$  in  Proposition \ref{vaderetro0} is always satisfied when $\alpha\eps \leq \udl,$
 since $L/\eps \geq 1.$ 
 Also, for $\alpha = \udl/\eps,$ assumption ${\mathit 3.}$ in Proposition \ref{vaderetro0} is weaker than assumption ${\mathit 2.}$ 
 We therefore deduce
\end{remark}

\begin{corollary}
\label{vaderetro}
Let $u$ be a solution to \eqref{odepertu} satisfying assumption $\hzero$, and such that
$\mathcal{D}(u)\cap \ILLL \subset \IL.$
If
\begin{equation} 
\label{bornf}
\eps \Vert f \Vert_{L^2(\ILLL)} \leq \left(\frac{M_0}{\rL}\right)^\frac12,
\end{equation}
then $\wp(\udl)$ holds.    
\end{corollary}

%%%%%%%%%%%%%%%%%%%%%%%%%%%%%%%%%%%%%%%%%%%%%%%%%%%%%%%%%%%%%%%%%%%%%%%%%%%%%%%%%%%%%%%%%
\section{Regularized fronts}
\label{regupara}

In the whole section, we assume that $v_\eps$ is a solution of $\PGL$ which satisfies $\hzero$ and the 
confinement condition $\mathcal{C}_{\rL,S}$. 

%%%%%%%%%%%%%%%%%%%%%%%%%%%%%%%%%%%%%%%%%%%%%%%%%%%%%%%%%%%%%%%%%%%%%%%%%%%%%%%%%%%%%%%%%
\subsection{Finding  regularized  fronts}
%%%%%%%%%%%%%%%%%%%%%%%%%%%%%%%%%%%%%%%%%%%%%%%%%%%%%%%%
%For two  times $0<s_1\leq s_2$,  consider  the localized dissipation  given by \eqref{leglaude}.
%$$\disL (s_1, s_2)=\eps 
%\int_{s_1\eps^{-\omega}}^{s_2\eps^{-\omega}}\left(\int_{-2{\rm L}}^{+2{\rm L}}
%\vert \partial_t \v_\eps (x, \eps^{-\omega } s) \vert^2 dx\right) ds.
%$$

%\begin{proposition}We have

%\end{proposition}
We provide here the proof to Proposition \ref{parareglo}, which is deduced from the following:
\begin{lemma}
 \label{meanvaluearg}
Given any $s_1<s_2$ in $[0,S],$ there exist at least one time $s$ in $[s_1,s_2]$
for which $\vm(\cdot,s)$ solves \eqref{odepertu} with
\begin{equation}
\label{moutarde} 
\Vert f\Vert_{L^2(\ILLL)}^2 \equiv \eps^{\omega-1}\Vert \partial_s \vm(\cdot,s)\Vert_{L^2(\ILLL)}^2 \leq \eps^{\omega-1}\frac{\dis^{3\rL}(s_1,s_2)}{s_2-s_1} \leq  \eps^{\omega-1}\frac{M_0}{s_2-s_1}.
\end{equation}
\end{lemma}
\begin{proof} 
It is a direct mean value argument, taking into account the rescaling of $\PGL$ according to our rescaling of time. 
\end{proof}

\begin{proof}[Proof of Proposition \ref{parareglo}] We invoke Lemma \ref{meanvaluearg}, and from \eqref{moutarde} and the assumption $s_2-s_1= \eps^{\omega+1}\rL$ of Proposition \ref{parareglo}, we derive exactly the assumption \eqref{bornf} in Corollary \ref{vaderetro}, from which 
the conclusion follows.
\end{proof}  
 
Following the same argument, but relying on Lemma \ref{prop:wpiok} and Proposition \ref{vaderetro0} rather than on Corollary \ref{vaderetro}, we readily obtain   
\begin{proposition}
\label{cornichon}  
For $\alpha_1\leq \alpha \leq \udl:$
\begin{enumerate}
\item Each subinterval of $[0,S]$ of size
$\q_0(\alpha)\eps^{\omega+2}$ contains at least one time $s$ at which $\wpi(\alpha \eps,s)$ holds, where
\begin{equation}
\label{defomega}
\q_0 (\alpha)= 4 M_0^2 \exp\left(\uprho_{\rm w}\alpha\right).
\end{equation}
\item
Each subinterval of $[0,S]$ of size
$\q_0(\alpha,\beta)\eps^{\omega+2}$ contains at least one time $s$ at which $\wp(\alpha \eps,s)$ holds, where
\begin{equation}
\label{defqzeroab}
\beta:=\frac{\rL}{\eps}\quad\text{and}\quad \q_0 (\alpha,\beta)= \max\left(\q_0(\alpha), \big(\frac{2 {\rm C_o}}{{\rm C_w}}\big)^2\Big(\frac{\beta}{M_0}\Big)^{1-\frac{1}{\theta}}\alpha^{2\omega}\right).
\end{equation}
\end{enumerate}
\end{proposition}

%%%%%%%%%%%%%%%%%%%%%%%%%%%%%%%%%%%%%%%%%%%%%%%%%%%%%%%%%%%%%%%%%%
  \subsection{Local dissipation}
  \label{localita}
%%%%%%%%%%%%%%%%%%%%%%%%%%%%%%%%%%%%%%%%%%%%%%%%%%%%%%%%%%%%%%%%
For $s \in [0,S]$, set  $\mathcal E_\eps^\rL (s)=\mathcal E_\eps^\rL(\vm(s))$ and, when $\wpi(\upalpha_1\eps,s)$ holds, $\matfEL(s)=\matfEL(\vm(s))$, $\matfEL$ being defined in \eqref{hereke0}. We assume throughout that $s_1 \leq s_2$ are
contained in $[0,S],$ and in some places (in view of  \eqref{cirque} that $s_2 \geq \rL^2\eps^\omega.$  

 \begin{proposition} If $s_2\geq \rL^2\eps^\omega,$ we have 
%there exists  some constant ${\rm C}_e$ independent of $\eps$  and $\rL$, such
%that for all $s\in (S, S_1)$, we have 
\begin{equation} 
 \mathcal E_\eps^\rL (s_2) +\disL (s_1,s_2) \leq  \mathcal E_\eps^\rL(s_1)   + 100 {\rm
C}_e\rL^{-(\omega+2)}(s_2-s_1)  +{\rm C_e}(1+M_0)
\left(\frac{\rL}{\eps}\right)^{-\omega}. 
 \label{timp}
 \end{equation}
\end{proposition}
\begin{proof} Let $0\leq \varphi\leq 1$ be a smooth  function with compact
support in $\ILL$, such that $\varphi (x)=1$ on $\IctL$, $\vert \varphi'' \vert
\leq  100\rL^{-2}$. It follows from the properties of $\varphi$ and \eqref{cirque} that  
 $$
\mathcal I_\eps (s, \varphi) \leq    \mathcal E_\eps^\rL (s) \ {\rm  \ for \ } s\in (s_1, s_2) \  {\rm and \ } \, 
\mathcal I_\eps (s_2, \varphi)  \geq  \mathcal E_\eps^\rL (s_2)-  {\rm C_e} \left(\frac{\rL}{\eps}\right)^{-\omega}, 
$$
which combined with \eqref{localizedenergy0} yields
$$\displaystyle{
\mathcal E_\eps^\rL (s_2) +\disL (s_1,s_2) \leq  \mathcal E_\eps^\rL (s_1) + {\rm C_e}
\left(\frac{\rL}{\eps}\right)^{-\omega} + \eps^{-\omega} \int_{s_1}^{s_2}
\mathcal F_S(s, \varphi, \vm) ds 
}$$
where $\mathcal F_S$ is defined in \eqref{flux}. The estimate \eqref{timp} is then obtained invoking 
 the inequality $\vert \xi_\eps\vert  \leq e_\eps$ to bound the term involving $\mathcal F_S$: combined with \eqref{cirque} for times
 $s \geq \rL^2\eps^\omega$ and with assumption $(H_0)$ for times $s\leq \rL^2\eps^\omega.$ 
\end{proof} 
 
If $\wp(\ud,s_1)$  and $\wpi(\ud', s_2)$ hold, for some $\ud,\ud'\geq \upalpha_1\eps$ and $s_2\geq \rL^2\eps^\omega$, then
combining inequality \eqref{timp} with the first inequality \eqref{hereke1}
applied to  $\vm(s_2)$ as well as the second applied to $\vm(s_1)$  we obtain 
\begin{equation}
\label{goupil}
\begin{aligned}
&\matfEL(s_2)+ \disL (s_1, s_2)\\
&\leq \mathcal E_\eps^\rL (s_2) + {\rm C_f}M_0\left(
\frac{\eps}{\updelta'}\right)^{\omega}  
  + \disL (s_1, s_2)\\
  &\leq 
 \mathcal E_\eps^\rL (s_1)  + 100{\rm C}_e\rL^{-(\omega+2)}(s_2-s_1) + {\rm C_f}M_0\left(
\frac{\eps}{\updelta'}\right)^{\omega} +{\rm C_e} (1+M_0)\left(
\frac{\eps}{\rL}\right)^{\omega} \\ 
 & \leq \matfEL(s_1)+({\rm C_w + C_f})M_0  \left( \frac{\eps}{\updelta}\right)^{\omega} +
 {\rm C_f}M_0\left( \frac{\eps}{\updelta'}\right)^{\omega}
  + 100{\rm C}_e\rL^{-(\omega+2)}(s_2-s_1) + {\rm C_e} (1+M_0)\left(
\frac{\eps}{\rL}\right)^{\omega}.
  \end{aligned}
  \end{equation}
We deduce from this inequality  an estimate for the dissipation between  $s_1$ and $s_2$ and an upper bound
on $\mathcal E_\eps^\rL (s_2)$: 
\begin{corollary}
\label{murino0}
Assume that $\wp(\ud,s_1)$  and $\wpi(\ud', s_2)$ hold, for some $\ud,\ud'\geq \upalpha_1\eps$ and $s_2 \geq \rL^2\eps^\omega$, and that
$\matfEL(s_1)=\matfEL(s_2).$
Then
\begin{equation*}
\label{bezout}
%\begin{aligned}
\disL[s_1, s_2] \leq  
({\rm C_w + C_f})M_0  \left( \frac{\eps}{\updelta}\right)^{\omega} +
 {\rm C_f}M_0\left( \frac{\eps}{\updelta'}\right)^{\omega}
  + 100{\rm C}_e\rL^{-(\omega+2)}(s_2-s_1) + {\rm C_e} (1+M_0)\left(
\frac{\eps}{\rL}\right)^{\omega},
\end{equation*}
\begin{equation*}\label{gouplibis}
\mathcal E_\eps^\rL (s_2) - \matfEL(s_2) \leq  
({\rm C_w + C_f})M_0  \left( \frac{\eps}{\updelta}\right)^{\omega} 
   + 100{\rm C}_e\rL^{-(\omega+2)}(s_2-s_1) + {\rm C_e} (1+M_0)\left(
\frac{\eps}{\rL}\right)^{\omega}.
\end{equation*}
\end{corollary}
%

%%%%%%%%%%%%%%%%%%%%%%%%%%%%%%%%%%%%%%%%%%%%%%%%
\subsection {Quantization of the energy}
%%%%%%%%%%%%%%%%%%%%%%%%%%%%%%%%%%%%%%%%%%%%%%%%
    
Let $s\in [0,S]$ and $\updelta\geq \upalpha_1\eps$, and assume that $\vm$ satisfies  $\wp(\ud, s)$.
The  front energy $\matfEL(s)$, by definition, may only take a finite number of values, and is
hence \emph {quantized}.   We emphasize  that, at this stage,   $\matfEL(s)$ is
only defined  \emph {assuming}  condition  $\wpi(\ud, s)$ holds. However,
the \emph {value} of $\matfEL(s)$ \emph {does not}  depend on $\ud$,  provided
that  $\ud \geq \upalpha_1\eps$, so that it suffices \emph{ultimately}    to
check that  condition  $\wpi(\upalpha_1\eps, s)$ is fulfilled.

Since    $\matfE(s)$ may take only a finite number of values, let $\upmu_1>0$
be the smallest possible difference between two  distinct  such  values.  Let $\rL_0\equiv \rL_0
(s_1, s_2)>0$ be such that  
\begin{equation}
\label{defrL_0}
100{\rm C}_e\rL_0^{-(\omega+2)}(s_2-s_1)=\frac{\upmu_1 }{4}
\end{equation}
and finally choose $\upalpha_1 $ sufficiently large so that 
\begin{equation}\label{eq:enlargealpha1}
\Big( (2{\rm C_f}+{\rm C_w})M_0+{\rm C}_e(1+M_0)\Big)\upalpha_1^{-\omega} \leq \frac{\upmu_1 }{4}.
\end{equation}   
As a direct consequence of \eqref{goupil}, \eqref{defrL_0}, \eqref{eq:enlargealpha1} and the definition
of $\upmu_1$ we obtain
\begin{corollary} 
\label{avecdec}
For $s_1\leq s_2 \in [0,S]$ with $s_2\geq \eps^\omega\rL^2$, assume that  $\wp(\upalpha_1\eps, s_1)$ and
$\wpi(\upalpha_1\eps,s_2)$ hold and that $\rL\geq
\rL_0 (s_1,s_2)$. Then we have  $\displaystyle{\matfEL(s_2) \leq \matfEL(s_1)}.$
Moreover, if $\matfEL(s_2)< \matfEL(s_1)$, then $\matfEL(s_2) + \upmu_1 \leq \matfEL(s_1).$  
\end{corollary}

In the opposite direction we have:

\begin{lemma}
\label{blog}
For $s_1\leq s_2 \in [0,S]$, assume that  $\wpi(\upalpha_1\eps, s_1)$ and
$\wpi(\upalpha_1\eps,s_2)$ hold and that $\rL\geq
\rL_0 (s_1,s_2)$. Assume also that
\begin{equation}
\label{mystere}
s_2-s_1 \leq \uprho_0 \left(\frac18 \dmineL(s_1)\right)^{\omega+2}. 
\end{equation}
Then we have  $\matfEL(s_2) \geq \matfEL(s_1).$ In case of equality, we have
$J(s_1)=J(s_2)$ and 
\begin{equation}
\label{greecej}
 \upsigma_{i(k\pm \frac 12)}(s_1)=\upsigma_{i(k\pm \frac 12)} (s_{1}), \hbox{
for any }  k \in J(s_1)  \  {\rm and  \  }   \dmineL(s_2)\geq  \frac12
\dmineL(s_1).  
\end{equation}
\end{lemma}
\begin{proof}  
It a consequence of the bound \eqref{theodebase1} in Theorem \ref{mainbs2} on the speed of the front set 
combined with assumption \eqref{mystere}. Indeed, this implies that for arbitrary $s \in [s_1,s_2],$ the front set 
a time $s$ is contained in a neighborhood of size $\dmineL(s_1)/8$ of the front set at time $s_1.$ In view of the 
definition of $\dmineL(s_1),$ and of the continuity in time of the solution, this implies that for all 
$k_0 \in J(s_1)$ the set 
$$
\mathcal A_{k_0}=\Big\{ k \in J(s_2) \ {\rm such \ that} \  a_k^\eps(s_{2}) \in  
 \big[a_{k_0}^\eps(s_{1})- \frac14 \dmineL(s_1),a_{k_0}^\eps(s_{1})+ \frac14
\dmineL(s_1)\big] \Big\}
$$ 
is non empty, since it must contain a front connecting $\upsigma_{i(k_0-\frac 12)}(s_{1})$ to
$\upsigma_{i(k_0+\frac 12) }(s_{1})$. 
 In particular, summing over all fronts in $\mathcal A_{k_0}$, we obtain 
\begin{equation*}
\label{somefronts}
\underset { k \in \mathcal A_{k_0}} \sum \SL_{i(k)} \geq \SL_{i(k_0)}, 
\end{equation*}
with equality if and only if $\sharp  \mathcal A_{k_0}=1.$
Summing over all indices $k_0$, we are led to the conclusion. 
\end{proof}

\renewcommand{\theequation}{\arabic{equation}}
\numberwithin{equation}{section}

%%%%%%%%%%%%%%%%%%%%%%%%%%%%%%%%%%%%%%%%%%%%%%%%%%%%%%%%%%%
\subsection{Propagating regularized fronts}
\label{propafront}
%%%%%%%%%%%%%%%%%%%%%%%%%%%%%%%%%%%%%%%%%%%%%%%%%%%%%%%%%%%
We  discuss  in this subsection the case of equality $\matfEL(s_1) =\matfEL(s_2)$.
We assume  throughout  that we are given $\udl\geq \ud>\upalpha_1\eps$ and  two
times $s_1\leq s_2 \in [\eps^\omega\rL^2,S]$  such that  
\begin{equation*}
  \label{hyphyp}
  \left\{
\begin{aligned}
  &\WPL(\ud, s_1)  \quad\text{and} \quad \WPI(\ud, s_2) \quad {\rm hold \ }\\
  & \matfEL(s_1)= \matfEL(s_2), \quad {\rm with \ }   \rL\geq \rL_0(s_1,s_2).
\end{aligned}
\right. \leqno{\mathcal C(\ud, L, s_1, s_2)}
\end{equation*}
Under that assumption, our first result shows that $\vm$ remains
well-prepared on almost the whole time interval $[s_1,s_2],$  with a smaller
$\ud$ though.

\begin{proposition}
\label{pneurose}
There exists $\upalpha_2 \geq \upalpha_1,$ depending only on $V$, $M_0$ and ${\rm C_w}$, with the following property.   
Assume that $\mathcal C(\ud, L, s_1, s_2)$ holds with $\upalpha_2\eps \leq \ud \leq \udl,$ then 
property $\WPL (\Lambda_{\rm log}(\ud),  s)$ holds for any time $s\in
[s_1+\eps^{2+\omega},  s_2]$,   
 where 
 \begin{equation}
 \Lambda_{\rm log} (\ud) =\frac{ \omega}{\uprho_{\rm w}} \eps  \left( \log \frac{ \ud}{\eps}\right).
\end{equation}
\end{proposition}

The proof of Proposition \ref{pneurose}  relies on the following.

\begin{lemma}
\label{oasisgede} Assume that  $\mathcal C(\ud, L, s_1, s_2)$ holds with $ \updelta \geq
\upalpha_1\eps$. We have the estimate, for $s \in  [
s_1+ \eps^{\omega+2}, s_2]$ 
\begin{equation*}
%\vert  \partial_t v_\eps (x, s\eps^{-\omega}) \vert  \leq C M_0^{\frac 12} \eps^{-2} \left( \frac{ \eps}{\updelta}\right)^{\frac{\omega}{2}} 
%{\rm \ \, and \ } 
      \int_{\ItdL} \vert   \partial_t v_\eps (x , s\eps ^{-\omega}) \vert^2 dx \leq  C\eps^{-3}\dis^{ \rL} [s, s-\eps^{\omega+2}].
     \end{equation*}
\end{lemma}
\begin{proof}[Proof of Lemma \ref{oasisgede}]
 Differentiating equation 
  $({\rm PGL}_\eps$)  with respect to time, we are led to 
  \begin{equation*}
  \label{differentiatet}
\vert   \partial_ t (\partial_t v_\eps)-\partial_{xx} (\partial_t v_\eps) \,
\vert \leq \frac{C}{\eps^2} \vert \partial_t v_\eps\vert. 
  \end{equation*}
   It follows from standard parabolic estimates, working for $x \in \ILL$  on  the cylinder $\Lambda_\eps(x)=[x-\eps, x+{\eps}]\times [t-\eps^2, t]$, where $t:=s\eps^{-\omega},$  that for any point $y \in [x-\frac{\eps}{2}, x+ \frac{\eps}{2}]$ we have
   $$\vert  \partial_t v_\eps (y, t) \vert \leq 
   C\eps^{-\frac32} \Vert \partial_t v_\eps
    \Vert_{L^2(\Lambda_\eps(x))}.
    %\leq C\eps^{-2}( \dis^{\frac 32 \rL} [s-\eps^2,s])^{\frac 12}.
         $$
    %  where the last inequality follows from   Corollary \ref{murino0}.  
      Taking the square of the previous inequality,  and integrating over  $[x-\frac{\eps}{2}, x+ \frac{\eps}{2}]$, 
      we are led to
   $$ \int_{x-\frac{\eps}{2}}^{x+\frac{\eps}{2}} \vert   \partial_t v_\eps (y, t) \vert^2 {d} y \leq  
    C\eps^{-2}\int_{[x-2\eps, x+2 \eps] \times [t-\eps^2, t]}  \vert   \partial_t v_\eps (y, t) \vert^2 dy.
    $$
    %  C M_0^{\frac 12} \eps^{-2} \left( \frac{ \eps}{\updelta}\right)^{\frac{\omega}{2}}.$$
   A elementary covering argument then yields 
    \begin{equation*}
    \label{ds4mat}
%    \begin{aligned}
   \int_{\ItdL} \vert   \partial_t v_\eps (y, t) \vert^2 dy\leq       C\eps^{-2} \Vert \partial_t v_\eps 
    \Vert_{L^2(\IctL\times [t-\eps^2, t])}^2\\
    \leq C\eps^{-3}  \dis^{ \rL} [s, s-\eps^{\omega+2}].\\
   %  \end{aligned}
     \end{equation*}
\end{proof}
\begin{proof}[Proof of Proposition \ref{pneurose}]  
In view of Proposition \ref{cornichon}, of Corollary \ref{avecdec}, and of assumption $\mathcal{C}(\ud,\rL,s_1,s_2)$, we may assume, without loss of generality, that 
\begin{equation}\label{eq:pastropdetemps}
 s_2-s_1 \leq 2 \q_0(\ud/\eps,\rL/\eps).
\end{equation}
Let $s\in (s_1+\eps^{\omega+2}, s_2)$, and consider once more   the
map $u=\vm(\cdot, s)$,  so that $u$ is a solution to \eqref{odepertu}, with  source term 
 $f= \partial_t\v(\cdot, s\eps^{-\omega})$.
  It follows from Lemma \ref{oasisgede}, combined with the first  of Corollary \ref{murino0} on the dissipation, that
\begin{equation*}
  \label{romane}
 \Vert f \Vert_{L^2(\ItdL)}^2 
   \leq  C \eps^{-3}\left[ 
   ({\rm C_w + 2C_f})M_0  \left( \frac{\eps}{\updelta}\right)^{\omega} 
   + 100{\rm C}_e\rL^{-(\omega+2)}(s_2-s_1) + {\rm C_e} (1+M_0)\left(
\frac{\eps}{\rL}\right)^{\omega}\right].
\end{equation*}
Notice that \eqref{eq:pastropdetemps} combined with the assumption $\ud \leq \udl$ yields
$$
 100{\rm C}_e\rL^{-(\omega+2)}(s_2-s_1)\leq C \left(\frac{\eps}{\delta}\right)^\omega.
$$
We deduce  from Lemma \ref{prop:wpiok}, imposing on
$\upalpha_2$  the additional condition 
$\displaystyle{\frac{ \omega}{\uprho_{\rm w}}   ( \log  \upalpha_2) \geq \upalpha_1,}$ that 
$\wpi((\Lambda_{\rm log}(\ud),  s)$ holds. It remains to show that $\wpo(
\Lambda_{\rm log}(\ud),  s)$ holds likewise.  To that aim, we invoke 
\eqref{gouplibis} which we use with the choice $s_1=s_1$ and $s_2=s.$ This yields, taking 
once more \eqref{eq:pastropdetemps} into account,
$$ 
\matfEL(s)-\mathcal E_\eps^\rL (s)  
\leq (C+{\rm C_w}) \left(
\frac{\eps}{\ud}\right)^{\omega}.
$$
Combining this relation with \eqref{leonardo2} and the first inequality of \eqref{hereke1},
we deduce that
\begin{equation}
\int_\Omega e_\eps (\vm(s)) ds \leq 
(C+{\rm C_w}) \left(
\frac{\eps}{\ud}\right)^{\omega} + C \left(\frac{\eps}{\Lambda_{\rm log}(\ud)}\right)^{\omega} \leq {\rm C_w}M_0 
\left(\frac{\eps}{\Lambda_{\rm log}(\ud)}\right)^{\omega}, 
\end{equation}
provided $\upalpha_2$ is chosen sufficiently large.
\end{proof}

In view of \eqref{mystere} and \eqref{theodebase1}, we introduce the function
$$
\q_1(\alpha) := \left(\frac{\q_0(\alpha)}{\uprho_0}\right)^{\frac{1}{\omega+2}},
$$
which represents therefore the maximum displacement of the front set in the 
interval of time needed (at most) to find two consecutive times at which $\wpi(\alpha\eps)$ holds. 

From Proposition \ref{pneurose} and Lemma \ref{blog} we deduce

\begin{corollary}\label{parareglo3}
Let $s\in [\eps^\omega\rL^2,S]$ and $\upalpha_2 \leq \alpha \leq \udl$, and assume that $\wp(\alpha \eps, s )$ holds as well
as  $\dmineL(s)\geq 16 \q_1(\upalpha) \eps.$ Then
$\wp(\Lambda_{\rm log}(\alpha\eps), s' )$ holds for any 
  $s + \eps^{2+\omega}  \leq  s' \leq \mathcal T_0^\eps(\alpha, s)$, where
$$
\mathcal T_0^\eps(\alpha, s) = \max \left\{ s+\eps^{2+\omega} \leq s'\leq S \quad \text{s.t.}\quad \dmineL(s'')\geq 8\q_1(\alpha)\eps \quad \forall s'' \in [s+\eps^{\omega+2},s']\right\}.
$$ 
\end{corollary}

\medskip
We complete this section presenting the

\begin{proof}[Proof of Proposition \ref{parareglo2}]  This follows directly from
Corollary \ref{parareglo3} with the choice $\alpha=\udl,$ noticing that $\Lambda_{\rm log}(\udl)=\udll.$
\end{proof}

\begin{proof}[Proof of Corollary \ref{jamboncru}]
If we assume moreover that $s_0 \geq \eps^\omega\rL^2$ and that 
$\wp(\udl,s_0)$ holds, then it is a direct consequence of  the inclusion  \eqref{theodebase1} and Proposition \ref{parareglo2}, taking into
account the assumption \eqref{dminstar}. If we assume only that $s_0\geq 0$ and that $\wp(\alpha_1\eps,s_0)$ holds, then it suffices to consider the first 
time $s_0'\geq s_0+ \eps^\omega\rL^2$ at which $\wp(\udl,s_0')$ holds and to rely on Proposition \ref{parareglo2} likewise. Indeed, since $s_0'-s_0 \leq 
\eps^\omega\rL^2+ \eps^{\omega+1}\rL$ by Proposition \ref{parareglo}, we may apply Corollary \ref{avecdec} and Lemma \ref{blog} for $s_1=s_0$ and $s_2=s_0'$, which yields $\matfEL(s_0)=\matfEL(s_0')$ and therefore also the same asymptotics for $\dmineL$ at times $s_0$ and $s_0'$. 
\end{proof}

%%%%%%%%%%%%%%%%%%%%%%%%%%%%%%%%%%%%%%%%%%%%%%%%%%%%%%%%%%%%%%%%%%%%%%%%%%%%%%%%%%%%%
\section{A first compactness results for the front points}
\label{compactness}
%%%%%%%%%%%%%%%%%%%%%%%%%%%%%%%%%%%%%%%%%%%%%%%%%%%%%%%%%%%%%%%%%%%%%%%%%%%%%%%%%%%%%
The purpose of this section is to provide the proofs of Proposition \ref{conspascites} and
Corollary \ref{lartichaut}.

\begin{proof}[Proof of Proposition \ref{conspascites}]
As mentioned, we choose the test functions (independently of time) so that they are affine near the
front points for any $s \in I^\eps(s_0)$. More precisely, for a given $k_0\in
J$  we impose the following conditions  on  the  test functions  $\chi\equiv \chi_{k_0}$
in \eqref{gratindo}:  
\begin{equation}
\label{specificitechi}
\left \{
\begin{aligned}
& \chi {\rm \ has \   compact  \ support \  in } \
[a_{k}^\eps(s_0)-\frac {1}{3} \dminstar(s_0)  , a_{k}^\eps(s_0)+\frac {1}{3} \dminstar(s_0)],\\  
&  \chi{\rm \  is \ affine \ on \ the   \ interval } \  
[a_{k}^\eps(s_0)-\frac {1}{4} \dminstar(s_0)  , a_{k}^\eps(s_0)+\frac {1}{4} \dminstar(s_0)],
{\rm \ with \ } \chi'=1 \ {\rm there} \\ 
&  \Vert \chi \Vert_{L^\infty (\R)} \leq C\dminstar(s_0),   \Vert \chi' \Vert_{L^\infty
(\R)} \leq  C  {\rm \ and \ } \Vert \chi''\Vert_{L^\infty (\R)} \leq
C{\dminstar(s_0)}^{-1}. 
\end{aligned}
\right.
\end{equation}
It follows from Corollary \ref{jamboncru} that, for $\eps$ sufficiently small,  we are in position to 
claim \eqref{gratindo} and \eqref{spain} for arbitrary $s_1$ and $s_2$ in the full interval $I^*(s_0).$ Combined with the first estimate  of Corollary \ref{murino0},
with $\updelta=\updelta'=\udll,$ this yields the conclusion \eqref{uniformita}.
\end{proof}

\begin{proof}[Proof of Corollary \ref{lartichaut}]
The family of functions $(\vm)_{0<\eps<1}$ is equi-continuous on every compact
subset of the interval $I^*(s)$, so that the conclusion follows from the 
Arzela-Ascoli theorem.  
\end{proof}

%%%%%%%%%%%%%%%%%%%%%%%%%%%%%%%%%%%%%%%%%%%%%%%%%%%%%%%%%%%%%%%%%%%%%%%%%%%%%%%%
\section{Refined asymptotics off the front set}
\label{raffinage}
%%%%%%%%%%%%%%%%%%%%%%%%%%%%%%%%%%%%%%%%%%%%%%%%%%%%%%%%%%%%%%%%%%%%%%%%%%%%%%%%
\subsection{Relaxations towards stationary solutions}
%%%%%%%%%%%%%%%%%%%%%%%%%%%%%%%%%%%%%%%%%%%%%%%%%%%%%
Throughout this section,  we assume that we are in the situation described by Corollary \ref{jamboncru},
in particular $\rL$ is fixed and $\eps$ will tend to zero.
Our main  purpose is then to provide  rigorous mathematical
statements and proofs  concerning the properties of the function $\Wmn=\Wmn^k$
defined in \eqref{Wm}, for given $k \in J$, which have been presented,
most of them in a formal way,  in Subsection \ref{raffarinage}. 
We notice first   that  we may expand $V'$ near $\upsigma\equiv
\upsigma_{i(k+\frac 12)}$  as 
\begin{equation}
\label{expansion}
 V' (\upsigma   + u)= 2 \theta \lambda  u^{ 2 \theta -1}\left( 1+ug(u)\right), 
\end{equation}
where $g$ is a some smooth function on $\R$  and where we have set for the sake of simplicity  $\lambda=\lambda_{i(k+\frac 12)}$. 
 We work on   the sets $\mathcal V_k(s_0)$ defined in \eqref{dopo} and on their
analogs at the $\eps$ level   
 \begin{equation}
 \label{vkeps}
\mathcal V_k^\eps(s_0)=\underset { s \in I^\eps(s_0) }\cup \mathcal J_\eps (s) \times \{s\}\equiv 
\underset { s \in I^\eps(s_0) }\cup \left(a_k^\eps(s)+\udll, a_{k+1}^\eps (s)-\udll \right)\times \{s\}.
\end{equation}
We will therefore work only with \emph{arbitrary} small values of $u$. Let
$u_0>0$ be sufficiently small so that $\vert ug(u) \vert \leq 1/4$ on $(-u_0, u_0)$
and $V'(\upsigma+u)$ is strictly increasing on $(-u_0,u_0),$ convex on $(0,u_0)$ and concave on $(-u_0,0).$    
For small values of $\eps$, the value of $u$ in \eqref{expansion},  in view of
\eqref{loindubord} in Lemma \ref{reset}, will not exceed $u_0$, and we may therefore 
assume for the considerations in this section that $ug(u)=u_0 g(u_0)$, if $u \geq u_0$ 
and  $-ug(u)=u_0 g(u_0)$, if $u
\leq -u_0.$  Equation $({\rm PGL})_\eps$ translates into the following equation for $ \Wm$
\begin{equation}
\label{aldehyde}
L_\eps( \Wm)\equiv \eps^{\omega}\frac{\partial \Wm}{\partial s}-\frac{\partial^2
\Wm}{\partial x^2}+ \lambda f_\eps( \Wm)=0,  
\end{equation}
where we have set
\begin{equation}
\label{feps}
\displaystyle{f_\eps(w)=2 \theta   w^{2 \theta-1}\left(1+\eps^{\frac{1}{\theta-1}}w g(\eps^{\frac{1}{\theta-1}}w)\right)}. 
\end{equation}
Notice that our assumption yield in particular
\begin{equation}
\label{labourage}
 \vert f_\eps(w) \vert  \geq \frac{3}{2} \theta \vert w \vert^{2 \theta-1}. 
\end{equation}
 The analysis of the parabolic equation \eqref{aldehyde} is the core of this
section.  As mentioned, our results express convergence to stationary solutions.
We first provide a few properties concerning  these stationary solutions:  the
first lemma     describes  stationary solutions involved in the attractive case,
whereas the second lemma is used in the repulsive case.   

\begin{lemma}
\label{specialstationary}
 Let $r>0$ and $0<\eps<1$. There exist unique solutions $\overset {\vee}
{\mathfrak u}_{\eps, r}^+ $ (resp. $\Umpes$) to 
\begin{equation*}
\left\{
\begin{aligned}
-&\frac{d\mathcal U} {dx^2}+ \lambda f_\eps (U)=0  {\rm  \  on } \ (-r,r), \\
&\mathcal U(-r)= +\infty  \  ({\rm resp.} \   \mathcal U (-r)= -\infty)  \  \,  {\rm and}  \ \   \ \mathcal U(r)=+\infty  \,  ({\rm resp.} \   \mathcal U (r)= -\infty).
\end{aligned}
\right. 
\end{equation*} 
Moreover we have, 
\begin{equation}
\label{majorana}
C^{-1} r^{-\frac{1}{\theta-1}} \leq  \ \Umpep \leq C \left(r- \vert x \vert \right)^{-\frac{1}{\theta-1}} \ {\rm and} \ 
C^{-1} r^{-\frac{1}{\theta-1}} \leq  -\Umpes \leq C \left( r- \vert x \vert\right)^{-\frac{1}{\theta-1}}, 
\end{equation}
for some constant $C>0$ depending only on $V.$
\end{lemma}

\begin{lemma}
Let $r>0$ and $0<\eps<1$ be given. There exists a unique solution $\overset {\rhd} {\mathfrak u}_{\eps, r} $ to
\begin{equation*}
-\frac{d\mathcal U} {dx^2}+\lambda f_\eps (U)=0  {\rm  \  on } \ (-r,r), \ \, 
\mathcal U(-r)= -\infty  \   {\rm and} \ \   \ \mathcal U(r)=+\infty.  
\end{equation*} 
\end{lemma}
These and related results are standard and have been considered since the works of Keller \cite{Keller} and Osserman \cite{Osserman} in the fifties, at least regarding existence. The convexity/concavity assumptions are sufficient for uniqueness. We refer to Lemma \ref{veryordinary} in the Appendix for a short discussion of the case a a pure power nonlinearity. 
%The solutions $\Ump$ (resp $\Umps$) of the introduction correspond to the case $r=1$, $\lambda=1$ and $\eps=1$.\\

%Going back to estimate \eqref{uniformita} which bounds the speed of motion of
%the front points, we notice that, if $s_1$ and $s_2$ are in $I^\eps(s_0)$, then
%we have 
%\begin{equation}
%\label{intext}
% \vert a^\eps_k(s_1)- a_k^\eps(s_2) \vert \leq  h_\eps(s_0) \  {\rm and } \   \vert a_k^\eps(s_1)- a_k^\eps(s_2)\vert \leq h_\eps %(s_0)
%\end{equation}
% provided  $\vert s_1-s_2\vert \leq  \uptau_\eps(s_0) $, where we have set 
%\begin{equation}
%\label{deftau}
%   h_\eps (s_0)=   {\rm c_3}   \llogepsw \dmine(s_0)  \ \ {\rm and \ } \          \uptau_\eps (s_0)=\llogepsw (\dmine(s_0))^{\omega}, 
%\end{equation}
%where ${\rm c}_3>0$ is some  suitable constant. 
We set
$\displaystyle{  r^\eps(s)=r^\eps_{ k+\frac 12} (s)= \frac{1}{2} (a_{k+1}^\eps (s)-a_k^\eps(s))}$.
Our aim is to provide sufficiently accurate expansions of $\Wm$ and the renormalized
discrepancy $\eps^{-\omega} \xi_\eps$  on neighborhoods of the points
$a_{k+\frac 12 }^\eps (s)$, for instance the intervals 
\begin{equation}
\label{jike}
\tetakeps(s)=a_{k+\frac 12}^\eps (s)+ [ -\frac{7}{8}r^\eps(s),  \frac{7}{8}r^\eps(s)]=
[ a_k^\eps(s)+ \frac{1}{8}r^\eps(s), a_{k+1}^\eps(s)-\frac{1}{8}\re(s)].
\end{equation}
We first turn to the  the attractive case   $\dagger_k=-\dagger_{k+1}$. We may assume additionally that 
\begin{equation}
\label{attractplus0}
k \in \{1, \cdots, \ell-1\} {\rm \ and \ } \dagger_k=-\dagger_{k+1}=1, 
\end{equation}
 the  case $\dagger_k=-\dagger_{k+1}=-1$  being handled similarly. 
 %In order to
%state our relaxation property, we introduce for  $d>0$ the number $ \epsilon_0
%(d)$ defined as the largest number $\eps>0$ such that 
%$$20 \,  \eps^{\frac{3\omega}{4} }\llogeps^\omega \geq   d^{\omega-2}.$$
%We will also assume that
%\begin{equation}
%\label{euhrk}
%r^\eps_{ k+\frac 12}(s) \leq \eps^{-\frac{w}{8}}{\dmine (s)}.
%\end{equation} 
%This assumption is always satisfied, in view of \eqref{nicosi} provided $\eps $ is sufficiently small. 

\begin{proposition} 
\label{aplus}
If \eqref{attractplus0} hold and $\eps$ is sufficiently small, then for any $s\in I^\eps(s_0)$ and every 
$x \in \tetakeps(s)$ we have the estimate 
\begin{equation}
\label{gloria}
\vert \Wm (x,s)-\lambda^{-\frac{1}{2 (\theta-1)} }
\overset {\vee} {\mathfrak u}_{ r^\eps(s) }^+(x)\vert  \leq  C \eps^{\min(\frac{1}{\omega+2}, \frac{\omega-1}{2(\theta-1)})}.
\end{equation}
\end{proposition}

The repulsive case corresponds to $\dagger_k=\dagger_{k+1}$ and  we may assume as above that
\begin{equation}
\label{attractmoins2}
k \in \{1, \cdots, \ell-1\} {\rm \ and \ }  \dagger_k=\dagger_{k+1}=1. 
\end{equation}
\begin{proposition}
\label{amoins}
If \eqref{attractmoins2} hold and $\eps$ is sufficiently small, then for any $s\in I^\eps(s_0)$ and every 
$x \in \tetakeps(s)$ we have the estimate 
\begin{equation}
\label{gloria2}
%\begin{aligned}
\vert \Wm (x,s)- \lambda^{-\frac{1}{2 (\theta-1)} }\overset {\rhd} {\mathfrak u}_{ r^\eps(s) }(x)\vert  \leq 
%&C
%\llogepsw (\dmine(s_0))^{-\frac{1}{\theta-1}}  + \\
% &C \eps^{\frac{1}{\theta-1}} \vert \llogeps \vert^{\frac{\omega}{\theta-1}} (\dmine(s_0))^{-\frac{2}{\theta-1}}\\
%  \leq 
 C \eps^{\min(\frac{1}{\omega+2}, \frac{\omega-1}{2(\theta-1)})}.
%\end{aligned}
\end{equation}
\end{proposition}

Combining these results with parabolic estimates, we obtain  estimates for the discrepancy.

\begin{proposition}
\label{vraiestime}
If $\eps$ is sufficiently small, then for any $s\in I^\eps(s_0)$ and every 
$x \in \tetakeps(s)$ we have the estimate 
\begin{equation}
\label{vraiedevraie}
%\begin{aligned}
\vert \eps^{-\omega}\xi_\eps(\vm)-\lambda_{i(k+\frac 12)}^{-\frac{1}{2( \theta-1)}}r^\eps (s)^{-(\omega+1)}\gamma_{k+\frac 12}\vert 
  \leq C\, \eps^\frac{1}{\theta^2},%, \dminstar(s_0))+   \\
% &\eps^{\frac{1}{2(\theta-1)}} (\dminstar(s_0))^{-2-\frac{2}{(\theta-1)}}    \Upsilon^{1/2} (\eps, \dminstar(s_0)),
%\end{aligned}
\end{equation}
where 
\begin{equation}
 \label{choix}
 \left\{
 \begin{aligned}\gamma_{k+\frac 12}&= A_\theta \ {\rm if \ } \  \dagger_k=-\dagger_{k+1} \\
 \gamma_{k+\frac 12}&= B_\theta \ {\rm if \ } \   \dagger_k=\dagger_{k+1}. 
 \end{aligned}
 \right.
 \end{equation}
\end{proposition}

For the outer regions, corresponding to $k=0$ and $k=\ell$  estimates for  the
discrepancy are directly deduced from  the crude estimates provided by
Proposition \ref{estimpar}. 
Proposition \ref{vraiestime} provides a rigorous  ground to the formal
computation \eqref{formalite} of the introduction, and hence allows to derive
the precise motion law. 
The proofs of Proposition \ref{aplus} and Proposition \ref{amoins} however  are the central part of this section. Note that
by no mean the estimates provided in Propositions \ref{aplus}, \ref{amoins} and \ref{vraiestime} are optimal, our goal was 
only to obtain convergence estimates, valid for all $\eps$ sufficiently small, uniformly on $\cup_{s \in I^\eps(s_0)}\tetakeps(s)\times \{s\}.$

%%%%%%%%%%%%%%%%%%%%%%%%%%%%%%%%%%%%%%%%%%%%%%
\subsection{Preliminary results}
%\label{prems}
%%%%%%%%%%%%%%%%%%%%%%%%%%%%%%%%%%%%%%%%%%%%%%
We first  turn to the proof of Lemma \ref{reset}, which provides first properties of $\Wm.$

\begin{proof}[ Proof of Lemma \ref{reset}] Let $x \in (a_k(s)+\udll, a_{k+1}(s)-\udll)$ and any $s\in I^\eps(s_0),$ and
recall that $d(x,s):={\rm dist}(x,\{a_k^\eps(s),a_{k+1}^\eps(s)\}).$ In view of Proposition \ref{parareglo}, 
and in particular of estimate \eqref{crottin2}, it suffices to show that 
$$\vm(y,s) \in B(\upsigma_i,\upmu_0) \quad \text{for all }\ (y,s) \in [x-\frac{d(x,s)}{2},x+\frac{d(x,s)}{2}] \times [s-\eps^\omega d(x,s)^2,s].
$$
By Theorem \ref{mainbs2}, on such a time scale the front set moves at most by a distance 
$$
d := \left(\frac{\eps^\omega d(x,s)^2}{\uprho_0 }\right)^{\frac{1}{\omega+2}} \leq \uprho_0^{-\frac{1}{\omega+2}} (\frac{\eps}{\udll})^\frac{\omega}{\omega+2} d(x,s)\leq \frac{d(x,s)}{4},
$$
provided $\eps/\rL$ is sufficiently small. More precisely, Theorem \ref{mainbs2} only provides one inclusion, forward in time, but its combination with Corrolary \ref{jamboncru} provides both forward and backward inclusions (for times in the interval $I^\eps(s_0)$), from which the conclusion then follows. 
\end{proof}

\medskip
For the analysis of the scalar parabolic equation \eqref{aldehyde},  we will
extensively use the fact that the map $f_\eps$ is \emph{non-decreasing} on $\R$,
allowing comparison principles. The desired  estimates  for $\Wm$ will be
obtained  using   appropriate choices of  sub- and super-solutions.  The
construction of these functions involve a number of elementary solutions. First,
we use  the functions $\mathcal W_\eps^{\pm}$, independent  of the space
variable $x$  and  solving  the ordinary differential equation 
\begin{equation}
\label{aldeode}
\left\{
\begin{aligned}
&\eps^{\omega}\frac{\partial \mathcal W_\eps^\pm }{\partial s}=-\lambda f_\eps( \mathcal W_\eps^\pm) \\
&\mathcal  W_\eps (0)= \pm \infty. 
\end{aligned}
\right.
\end{equation}
Using separation of variables, we may  construct such a  solution which verifies  the bounds
\begin{equation}
\label{rerelaxlax}
0<\mathcal W_\eps ^+(s)\leq  C \eps ^{\frac{\omega}{2(\theta-1) } }[\lambda
s]^{-\frac{1} {2(\theta-1)} } \  {\rm and    }  \  
0\geq \mathcal W_\eps ^-(s)\geq - C \eps ^{\frac{\omega}{2(\theta-1) } }[\lambda s ]^{-\frac{1} {2(\theta-1)}},
\end{equation}
so that it  relaxes  quickly to zero. We will  also use  solutions of the
standard heat equation and rely in several places on the next remark:

\begin{lemma}
\label{heatflow} 
Let $\Phi$ be a non negative solution to the heat equation
$\eps^{\omega} \partial_s \Phi-\Phi_{xx}=0, $
and $U$ be such that $L_\eps(U)=0$. Then $L_\eps(U+\Phi) \geq 0$,  and $L_\eps( U-\Phi)\leq 0$. 
\end{lemma}

\begin{proof}
Notice that $L_\eps(U\pm\Phi)=\lambda ( f_\eps (U\pm\Phi)-f_\eps(U))$, so
that the conclusion  follows  from the fact that $f_\eps$ is non-decreasing.  
\end{proof}

Next, let $s$ be given $I^\eps(s_0)$. By translation invariance, we may assume
without loss of generality that 
\begin{equation}
\label{milieu}
 a_{k +\frac{1}{2} }^\eps(s)=0.
 \end{equation} 
We set  $h_\eps=(\eps/2\uprho_0)^{\frac{1}{\omega+2}},$ and consider the cylinders 
\begin{equation}
\label{estanguet}
\Ldext(s)= \Iexte (s)\times [s-\eps, s]   \ {\rm and}    \ \ 
\Ldint (s)= \Iinte (s) \times [s-\eps, s], 
\end{equation}
where  $\Iinte(s)= [-\reint(s), \reint(s)] $, $\Iexte(s)= [-r^\eps_{\rm ext} (s), r^\eps_{\rm ext } (s)]$  with 
$$ 
\reext (s)=\re(s)+ 2h_\eps  {\rm \ and \ } 
\reint(s)= \re(s)-2 h_\eps.
$$
If $\eps$ is sufficiently small, in view of \eqref{theodebase1} we have the inclusions, with  $\mathcal V_k^\eps(s_0)$ defined in \eqref{vkeps} ,
$$
\Ldint (s) \subset \Pi_\eps(s)\equiv \mathcal V_k^\eps(s_0)\cap \left( [s-\eps, s] \times \R \right)\subset  \Ldext(s).
$$
As a matter of fact, still for $\eps$ sufficiently small, we have for any $\tau \in [s-\eps, s]$,
\begin{equation}
 \label{lointext}
 \left\{
 \begin{aligned}
-\reext (s)+h_\eps  &\leq  a_k^\eps(\tau)+\udll \leq -\reint(s) - h_\eps,  \\
\reint + h_\eps  &\leq  a_{k+1}^\eps(\tau)-\udll \leq \reext(s) -  h_\eps(s_0).
\end{aligned}
\right.
\end{equation}
We also   consider the parabolic boundary of  $\Ldext(s)$
\begin{equation*}
\begin{aligned}
\partial_p \Ldext(s)&= [-r^\eps_{\rm ext} (s),r^\eps_{\rm ext} (s)]\times \{s-\eps\} \cup \{-\reext \} \times [s-\eps, s]\cup  \{\reext \} \times [s-\eps, s] \\
&=\partial \Ldext (s) \setminus [-\reext(s), \reext(s)]\times \{s\}, 
\end{aligned}
\end{equation*}
and define $\partial_p \Ldint(s)$ accordingly.  Finally, we set
$$
\partial_p \Pi_\eps (s)=\partial ( \Pi_\eps (s)) \setminus [a_k^\eps(s)+\udll, a_{k+1}^\eps(s)-\udll] \times \{s\}.
$$
 
A first application of the comparison principle leads to the following bounds:

\begin{proposition}
\label{tonio} 
For $x \in \Iinte(s)$
\begin{equation}
\label{tonio1}
 \left\{
 \begin{aligned}
 \Wm (x, s) &\leq \Umpepint (x) + C\eps ^{\frac{\omega-1}{2(\theta-1) } } \\
 \Wm (x, s) &\geq 
 \Umpesint(x)
 -C\eps^{\frac{\omega-1}{2(\theta-1) } }.
\end{aligned}
\right.
\end{equation}
\end{proposition}
\begin{proof}  We work on  the cylinder $\Ldint (s)$ and consider there the  comparison map 
$$ 
W^{\rm sup}_\eps (y, \tau)= \Umpepint (y) + \mathcal W_\eps(\tau-(s-\eps)) \ {\rm for \ } (y, \tau) \in \Ldint(s).   
$$
Since the two functions on the r.h.s of  the definition of $ W_\eps^{\rm sup}$ are
positive  solutions to \eqref{aldehyde} and  since $f_\eps$ is super-additive on
$\R^+$, that is, since 
\begin{equation}
\label{superadditif}
f_\eps(a +b) \geq f_\eps(a)+ f_\eps(b) \ \ {\rm   provided }  \  a\geq 0, b\geq 0,  
\end{equation}
we deduce that 
$$ 
L_\eps \left (W^{\rm sup}_\eps (y, \tau)\right) \geq 0 \   {\rm  on  \  }  \Ldint (s) \ {\rm  \ with \ } 
W^{\rm sup}_\eps (y, \tau)=+\infty \ {\rm for \  } (y, \tau) \in \partial_p \Ldint, 
$$
so that  $\displaystyle{W^{\rm sup_\eps} (x, s) \geq \Wm} $ on $\displaystyle{\partial_p \Ldint}$. It follows that  
$ W^{\rm sup}_\eps (y, \tau) \geq  \Wm  \  {\rm on } \  \  \Ldint, $ 
which, combined with \eqref{rerelaxlax} immediately leads to
the first inequality. The second is derived similarly.  
\end{proof}
At this stage, the constructions are some  somewhat different in the case of
attractive and repulsive forces, so that we need to  distinguish the two cases.

%%%%%%%%%%%%%%%%%%%%%%%%%%%%%%%%%%%%%%%%%%%%%%%%%%%%%%%%%%%%%%%%%%%%%%%%%%%%%%%%%%%%%%%%%
\subsection{The attractive case}
\label{attractiveworld}
%%%%%%%%%%%%%%%%%%%%%%%%%%%%%%%%%%%%%%%%%%%%%%%%%%%%%%%%%%%%%%%%%%%%%%%%%%%%
We assume  here  that $\dagger_k=-\dagger_{k+1}$. Without loss of generality, we may assume that 
\begin{equation}
\label{attractplus}
\dagger_k=-\dagger_{k+1}=1, 
\end{equation}
the  case $\dagger_k=-\dagger_k=-1$ being   handled similarly. The purpose of  this subsection 
is to provide \emph{the proof to  Proposition \ref{aplus}}. 
We   split  the proof into  separate lemmas, the main efforts  being   devoted to the construction of   \emph{subsolutions}.
We start with the following  lower bound: 
\begin{lemma}
  \label{positif}
  Assume that \eqref{attractplus} holds. Then, for  $ x \in \mathcal J_\eps(s-\frac{\eps}{2}),$ we have the lower bound 
$$ 
\Wm (x, s-\frac{\eps}{2})  \geq - C\eps ^{\frac{\omega-1}{2(\theta-1)}}.
$$
\end{lemma}
\begin{proof} 
In view of \eqref{presdubord}, we notice that 
$$
\Wm(y, \tau)\geq 0 \ \, {\rm on } \   \partial_p \Pi_\eps(s)
\setminus [a_k(s-\eps)+\udll,a_{k+1}(s-\eps)-\udll ]\times\{s-\eps\}.
$$
We consider next  the function $W_\eps$ defined for $\tau \geq s-\eps$ by 
$\displaystyle W_\eps(y, \tau)= \mathcal W_\eps^- (\tau-(s-\eps))$. Since
$W_\eps<0$, and since $W_\eps(s-\eps)=-\infty$, we obtain 
$\displaystyle{W_\eps\leq  \Wm   \ {\rm on} \  \partial_p \Pi_\eps (s), }$ 
 so that, by the comparison principle we are led to $W_\eps\leq  \Wm$ on
$\Pi_\eps(s)$ leading to the conclusion.  
\end{proof}

\begin{proposition}
\label{capoue} 
 Assume that \eqref{attractplus} holds. We have the lower bound for $x \in \mathcal J_\eps(s)$
\begin{equation}
\label{lw}
 \Wm (x, s)  \geq \Umpepext (x) -
 C \eps^{-\frac{1}{3\theta-1}}  
\exp \left ( -\pi^2\frac{\eps^{-\omega+1}}{32(\re(s))^2} \right). 
\end{equation}
\end{proposition}
\begin{proof} On $\mathcal J_\eps (s-\frac{\eps}{2})$ we consider the map $\varphi_\eps$ defined by 
\begin{equation}
\label{gilbert}
\varphi_\eps (x)= \inf \{ \Wm (x, s-\frac{\eps}{2})-\Umpepext (x), 0\} \leq 0.
\end{equation}
Invoking \eqref{lointext} and estimates \eqref{majorana} for $\Umpepext$, we obtain, for $x \in \mathcal J_\eps(s-\frac{\eps}{2})$
\begin{equation}
\label{dingue}
0\leq  \Umpepext (x) \leq Ch_\eps^{-\frac{1}{\theta-1}},
%\leq  C \left(  \llogeps\right) ^{\frac{\omega}{\theta-1}} \dmine(s_0)^{\frac{1}{\theta-1}},  
\end{equation}
which combined with Lemma \ref{positif} yields
\begin{equation}
 \label{harlemd}
 \vert \varphi_\eps (x) \vert \leq Ch_\eps^{-\frac{1}{\theta-1}}
 %C \left(  \llogeps\right) ^{\frac{\omega}{\theta-1}} \left(\dmine(s_0)\right)^{-\frac{1}{\theta-1}}, 
 {\rm \ for \  }  x \in \mathcal J_\eps (s-\frac{\eps}{2}).
\end{equation}
Combining \eqref{dingue}, estimate \eqref{majorana} of Lemma \ref{specialstationary} 
and estimate \eqref{presdubord}  of Lemma \ref{reset}, we deduce that, if $\eps$ is sufficiently small then
\begin{equation}
\label{redite}
 \varphi_\eps (a_k^\eps(s-\frac{\eps}{2})+ \udll )=\varphi_\eps (a_{k+1}^\eps(s-\frac{\eps}{2})-  \udll)=0.  
\end{equation}
We extend $\varphi_\eps$ by $0$ outside the set  $\mathcal J_\eps (s-\frac{\eps}{2})$, and  consider  the solution $\Phi_\eps$  to
\begin{equation}
\label{hideux}
\left\{
\begin{aligned}
&\eps^{\omega}  \frac{\partial \Phi_\eps}{\partial \tau}- \frac{\partial
\Phi_\eps} {\partial x^2}=0  {\rm \  on \ } \Ldext (s) \cap \{ \tau \geq s-\frac{\eps}{2}\}  \\ 
& \Phi_\eps (x, s-\frac{\eps}{2})=\varphi_\eps (x) \quad   {\rm  for  \ }  x \in \Iexte (s-\frac{\eps}{2})\\
& \Phi_\eps (\pm \reext(s), \tau)=0  \quad {\rm for }  \   \tau \in
(s-\frac{\eps}{2}, s). 
\end{aligned}
\right. 
\end{equation}
Notice that $\Phi_\eps\leq 0$.  We consider next on $\Ldext(s) \cap \{\tau \geq s-\frac{\eps}{2}\}$ the function $W^{\rm inf}_\eps$ defined by
$$ 
W^{\rm inf}_\eps(y, \tau)= \Umpepext (y) + \Phi_\eps(y, \tau). 
$$
It follows from Lemma \ref{heatflow} that $L_\eps(W_\eps^{\rm inf} ) \leq 0$, so
that $W_\eps^{\inf}$ is a  \emph{subsolution}.  Since $W_\eps^{\rm
inf} \leq \Wm$ on $\partial_p \left( \Pi_\eps(s) \cap \{\tau \geq s-\frac{\eps}{2}\}\right)$ it follows in particular that  
\begin{equation}
\label{sursoir}
W_\eps^{\rm inf} \leq \Wm  \   {\rm on} \  \mathcal J_\eps(s).
\end{equation}
To complete the proof, we rely on the next linear  estimates for $\Phi_\eps$.   
\begin{lemma}
\label{linearpara}
We have the bound, for $y \in \Iexte $ and $\tau \in (s-\frac{\eps}{2}, s)$
$$\vert \Phi_\eps (y, \tau) \vert \leq C
\exp \left ( -\pi^2\eps^{-\omega}\frac{ (\tau-(s-\frac{\eps}{2})}{16(\re (s))^2} \right)
 \Vert \varphi_\eps  \Vert_{L^\infty(\mathcal J_\eps (s-\frac{\eps}{2}))}.$$
\end{lemma}
We postpone the proof of Lemma  \ref{linearpara} and complete  the proof of  Proposition \ref{capoue}.\\

\noindent
{\it Proof of Proposition \ref{capoue} completed}. Combining  Lemma 
\ref{linearpara} with  \eqref{harlemd}, we are led, for $x \in \mathcal J_\eps(s)$, to 
\begin{equation}
\label{reverend}
\vert \Phi_\eps(x, s)\vert \leq  C h_\eps^{-\frac{1}{\theta-1}}  
\exp \left ( -\pi^2\frac{\eps^{-\omega+1}}{32(\re(s))^2} \right). 
\end{equation}
The conclusion then follows, invoking \eqref{sursoir}. 
\end{proof}

\begin{proof}[Proof of Lemma \ref{linearpara}]  Consider on the interval $[- 2\re(s),  2\re(s)]$   the function 
$\psi(x)$ defined by
$\displaystyle{\psi(x)=\cos(\frac{\pi }{4\re(s)}\,  x ),}$ 
 so that   $\displaystyle{ - \ddot\psi=\frac{\pi^2 } {16}(\re(s))^{-2} \psi}$,
$\psi \geq 0$, $\psi(-2\re(s))=\psi(2\re(s))=0$ and $\psi(x)\geq 1/2$ for $x \in [-\reext(s),\reext(s)]$. 
Hence, we obtain 
$$
\eps^{\omega}   \Psi_\tau -\Psi_{xx} =0 {\rm  \ on \ }\Ldext(s)\cap \{\tau \geq s-\frac{\eps}{2}\}, 
{\rm  \ where \ }  
\Psi(x, \tau)=\exp\left(-\pi^2\eps^{-\omega} \frac{\tau-(s-\frac{\eps}{2})}{16 \re(s)^2} \right)\psi(x). 
$$
On the other hand,  for  $(y, \tau) \in \partial_p \left(\Ldext (s)\cap \{\tau \geq s-\frac{\eps}{2}\}\right)$ we have
\begin{equation*}
\label{duplicata}
\vert \Phi_\eps (y, \tau)   \vert \leq 
\Vert \varphi_\eps \Vert_{L^\infty(\mathcal J_\eps (s-\frac{\eps}{2}))} 2 \Psi (y, \tau)  
\end{equation*}
and the conclusion follows therefore from the comparison principle for the heat equation. 
\end{proof}

\begin{proof} [Proof of Proposition \ref{aplus} completed] 
Combining the upper bound \eqref{tonio1} of  Proposition  \ref{tonio} with the
lower bound \eqref{lw} of Proposition  \ref{capoue}, we are led, for $\eps$ sufficiently small, to 
\begin{equation}
\label{nomdelarose}
            \Umpepext (x) -   A_\eps 
\leq  \Wm (x, s) \leq \Umpepint (x) + A_\eps, 
\end{equation}
 where we have set 
 \begin{equation}
 \label{defrestA}
 \begin{aligned}
 A_\eps= C \eps^\frac{\omega-1}{2(\theta-1)}.
\end{aligned}
\end{equation}
The conclusion \eqref{gloria} then follows from Proposition \ref{catamaran} of the Appendix combined with the definition of $h_\eps$
and \eqref{deriveur}.
\end{proof}

%%%%%%%%%%%%%%%%%%%%%%%%%%%%%%%%%%%%%%%%%%%%%%
\subsection{The repulsive case}
\label{repulsiveworld}
%%%%%%%%%%%%%%%%%%%%%%%%%%%%%%%%%%%%%%%%%%%%%%
In this subsection, we assume throughout that $\dagger_k=\dagger_{k+1}$ and  may
assume moreover that  
\begin{equation}
\label{attractmoins}
\dagger_k=\dagger_{k+1}=1,
\end{equation}
the  case $\dagger_k=\dagger_k=-1$  is handled similarly. The main
purpose of  this subsection is to provide \emph{the proof of  Proposition
\ref{amoins}},  the central part being  the construction of accurate
\emph{supersolutions},  
subsolutions being provided by the same construction.
We assume as before that \eqref{milieu} holds, and    use as comparison map
$\mathfrak U_\eps$ defined  on $ \Itrans (s)\equiv (-\reext (s),\reint(s))$ by  
$$\mathfrak U_\eps(\cdot) \equiv \Umpepst \left(\cdot+ 2h_\eps\right),$$
 so  that $\mathfrak U_\eps(x)\to +\infty$ as $x \to \reint(s)$, $\mathfrak
U_\eps(x)\to -\infty$ as $x \to -\reext(s)$ and 
  $\vert  \mathfrak U_\eps(-r^\eps(s))\vert \leq  C h_\eps^{-\frac{1}{\theta-1}}$. 
\begin{proposition}
\label{superette} For $x \in
(a_k(s)+\udll , \reint(s))$ we have the inequality, where $C>0$ denotes some constant
\begin{equation}
\label{superette0}
 \Wm(x, s) \leq \mathfrak U_\eps(x)+ 
 C \eps^{-\frac{1}{3\theta-1}} 
\exp \left ( -\pi^2\frac{\eps^{-\omega+1}}{16(\re(s))^2} \right). 
 \end{equation}
\end{proposition}
\begin{proof}
As for \eqref{gilbert},  write  for  $x \in \Itrans (s)\cap \mathcal J_\eps(s-\eps)$
$$
\psi_\eps(x)= \sup\{\Wm (x, s-\eps)-\mathfrak U_\eps, 0\} \geq 0. 
$$
We notice that
$$ 
\psi_\eps (a_k(s-\eps)+\udll)=\psi_\eps (\reint(s))=0.
$$
Indeed,  for the first relation, we argue as in \eqref{redite} whereas for the second,  we have $\mathfrak  U_\eps (\reint(s))=\Umpepst (r^\eps(s))=+\infty.$  We extend
$\psi_\eps$ by $0$ outside the interval $\Itrans(s)\cap \mathcal J_\eps(s-\eps)$ and derive, arguing as for \eqref{harlemd},
\begin{equation}
 \label{harlemd2}
 \vert \psi_\eps (x) \vert \leq C h_\eps^{-\frac{1}{\theta-1}}\leq C \eps^{-\frac{1}{3\theta-1}}
 {\rm  \ for \ } x \in \R.  
\end{equation}
We introduce the cylinder $\Ldtrans (s)\equiv (-\reext(s), \reint(s)) \times(s-\eps, s)$ and the solution $\Psi_\eps$ to 
\begin{equation}
\label{hideux2}
\left\{
\begin{aligned}
&\eps^{\omega}  \frac{\partial \Psi_\eps}{\partial \tau}- \frac{\partial \Psi_\eps} {\partial x^2}=0  {\rm \  on \ } \Ldtrans (s)  \\
& \Phi_\eps (x, s-\eps)=\psi_\eps (x) \   {\rm  for  \ }  x \in   (-\reext(s), \reint(s)) \ {\rm and}  \\
& \Psi_\eps (- \reext(s), \tau)=\Psi_\eps (\reint(s), \tau)=0  
 \ {\rm for }  \   \tau \in (s-\eps, s),
\end{aligned}
\right. 
\end{equation}
so that $\Psi_\eps\geq 0$. Arguing as for \eqref{reverend}, we obtain for $\tau \in (s-\eps, s)$
\begin{equation}
\label{reverend2}
\vert \Psi_\eps(y, \tau)\vert \leq  C \eps^{-\frac{1}{3\theta-1}} 
\exp \left ( -\pi^2\eps^{-\omega}\frac{ (\tau-(s-\eps))}{16(\re(s))^2} \right). 
\end{equation}
We consider on $\Ldtrans (s)$ the function $W^{\rm trans}_\eps$ defined by 
$$ 
W^{\rm trans}_\eps(y, \tau)= \mathfrak U_\eps (y) + \Psi_\eps(y, \tau). 
$$ 
It follows from Lemma \ref{heatflow} that $L_\eps(W^{\rm trans} )_\eps \geq 0$, that
is $W_\eps^{\rm trans}$ is a \emph{supersolution}   for $L_\eps$ on $\Ldtrans
(s)$.  Consider next the subset $\Pi^{\rm trans}_\eps (s)$ of $\Ldtrans$ defined
by 
$$
\Pi^{\rm trans}_\eps (s)\equiv  \underset{ \tau \in (s-\eps, s)} \cup (a_k(\tau)+\udll , \reint(s))\times\{\tau\}.
$$
We claim that 
\begin{equation}
\label{chamoine}
W_\eps^{\rm trans } \geq \Wm  \  {\rm  on } \  \partial_p \, \Pi^{\rm trans}_\eps (s).  
\end{equation}
Indeed, by construction, we have  $W_\eps^{\rm trans }=+\infty$ on $\reint(s) \times (s-\eps, s)$ and 
$W_\eps^{\rm trans } (x, s-\eps) \geq \Wm(x, s-\eps)$  for $x \in (a_k(s-\eps)+\udll, \reint (s)).$
Finally  on $ \cup_{ \tau \in (s-\eps, s)}
\{a_k(\tau)+\udll \} \times \{\tau\}$, the conclusion \eqref{chamoine} follows
from estimate \eqref{presdubord} of  Lemma \ref{reset}.  Combining inequality
\eqref{chamoine} with the comparison principle, we are led to  
\begin{equation}
\label{sursoire}
W_\eps^{\rm trans} \geq \Wm  \   {\rm on} \  \Pi_\eps^{\rm trans} (s).
\end{equation}
Combining \eqref{sursoire}  with \eqref{reverend2} we are led to  \eqref{superette0}.
\end{proof}
Our next task is to construct a \emph{subsolution}. To that aim, we   rely on
the symmetries of the equation, in particular the invariance $x\to -x$ and the
\emph{almost oddness} of the nonlinearity. To be more specific, we introduce the
operator  
\begin{equation*}
\label{alcanoide}
\tilde L_\eps( u)\equiv \eps^{\omega}\frac{\partial  u}{\partial \tau}-\frac{\partial^2 u}{\partial x^2}+ \lambda  \tilde f_\eps( u)=0, 
{\rm \  with \ }
\tilde f_\eps(u)=2 \theta   u^{2 \theta-1}\left(1-\eps^{\frac{1}{\theta-1}}u g(-\eps^{\frac{1}{\theta-1}}u)\right),
\end{equation*}
which has the same properties as $L_\eps$ and consider the  stationary solution $\Umpepstt$ for  $L_\eps$ defined on $(-\re(s), \re(s))$  by
\begin{equation*}
-\frac{\partial^2  \Umpepstt}{\partial x^2}+ \lambda   f_\eps( \Umpepstt)=0,   \  \ 
\Umpepstt (- \re(s))= + \infty {\rm \ and \ }
\Umpepstt ( \re(s))=- \infty,  
\end{equation*}
so that $-\Umpepstt $ is a stationary solution to $\tilde L_\eps$. 
Consider  the function $\tilde \Wm$ defined by  
\begin{equation}
\label{croissant}
\tilde \Wm(x, \tau)=-\Wm(-x, \tau)
\end{equation}
and observe that
$\displaystyle{\tilde L_\eps (\tilde \Wm)=0.}$
Finally, we define the interval $(-\reint(s), \reext(s))$ the function
$$
\mathfrak  V_\eps (x) \equiv\Umpepstt \left(2h_\eps-x\right), 
$$
so  that $\mathfrak V_\eps(x)\to -\infty$ as $x \to -\reint(s)$ and $\mathfrak
V_\eps(x)\to +\infty$ as $x \to \reext(s).$ 

\begin{proposition}
\label{taskforce} For $x \in (-\reint(s) , a_{k+1}(s)-\udll)$
we have the inequality, 
\begin{equation}
\label{taskforce1}
 \Wm(x, s) \geq \mathfrak V_\eps (x) - 
 C \eps^{-\frac{1}{3\theta-1}} 
\exp \left ( -\pi^2\frac{\eps^{-\omega+1}}{16(\re(s))^2} \right).
 \end{equation}
\end{proposition}

\begin{proof}
We argue as in the proof of Proposition \ref{superette}, replacing  $L_\eps$  by $\tilde \eps$, $\Wm$ by 
$\tilde \Wm$, and $\mathfrak U_\eps$ by 
$\displaystyle{\tilde {\mathfrak U}_\eps=-\Umpepstt\left (\cdot-2h_\eps (s_0)\right)}$. Inequality \eqref{taskforce1} for 
$\Wm$ is then obtained inverting relation \eqref{croissant} and from the corresponding estimate on $\tilde \Wm.$ 
\end{proof}

\begin{proof}[Proof of Proposition \ref{amoins} completed]  Combining \eqref{superette0} with \eqref{taskforce1} we are led to
\begin{equation}
\label{perrault}
\mathfrak U_\eps (x) -\tilde A_\eps \leq \Wm(x, s) \leq \mathfrak V_\eps (x) + \tilde A_\eps, 
\end{equation}
 where we have set
 $\displaystyle{ \tilde A_\eps=  C \eps^{-\frac{1}{3\theta-1}} 
\exp( -\pi^2\frac{\eps^{-\omega+1}}{16(\re(s))^2}). 
}$  The proof is then completed with the same arguments as in the proof of Proposition \ref{aplus}
\end{proof}

%%%%%%%%%%%%%%%%%%%%%%%%%%%%%%%%%%%%%%%%%%%%%%%%%%%%%%%%%%%%%
\subsection{Estimating the discrepancy}
\label{vabenecosi}
%%%%%%%%%%%%%%%%%%%%%%%%%%%%%%%%%%%%%%%%%%%%%%%%%%%%%%%%%%%%%
\subsubsection{Linear estimates}

The purpose of this section  is to provide the proof of Proposition
\ref{vraiestime}. So far Proposition \ref{aplus} and Proposition \ref{amoins}
provide  a good approximation of $\Wm$ on the level of the \emph{uniform norm}.
However, the discrepancy involves also a first order derivative, for which we
rely on the regularization property  of the linear heat equation. To that aim,
set  
\begin{equation*}
  \left\{
  \begin{aligned}
  \Lambda & \equiv (-1, 1) \times [0,1],\, \Lambda^{1/2}  \equiv (-\frac 12,
\frac 12) \times [\frac 34 ,1],  {\rm  \ and  \  more \ generally \ for \
\varrho>0 }  \\  
  \Lambda_\varrho  &\equiv (-\varrho, \varrho) \times [0,\varrho^2],\,
\Lambda^{1/2}_{_\varrho}  \equiv (-\frac 12 \varrho, \frac 12 \varrho) \times
[\frac 34\varrho^2 ,\varrho^2]. 
  \end{aligned} 
  \right.
  \end{equation*}
 The following standard result  (see e.g. \cite{BOS1} Lemma A. 7 for a proof) is useful  in our context.  
 
\begin{lemma}
\label{lemmeBOS}
Let $u$ be a smooth real-valued function on $\Lambda$. There exists a constant $C>0$ such that
$$
  \Vert   u_{x} \Vert _{ L^\infty (\Lambda^{1/2})}\leq
    C(\Vert   u_t-\ u_{xx} \Vert_{L^\infty(\Lambda)}+\Vert u \Vert_{L^\infty(\Lambda)}  ).$$
  \end{lemma}
We deduce from this result the following scaled  version.

\begin{lemma}
\label{farbreton0}
 Let $\varrho>0$ and  let $u$ be defined on $ \Lambda_\varrho$. Then we have for some constant  $C>0$ independent of $\varrho$
\begin{equation}
\label{farbreton}
  \Vert   u_{x} \Vert _{ L^\infty (\Lambda_\varrho^{1/2})}^2\leq
  C\left[  \Vert   u_t-u_{xx} \Vert _{ L^\infty(\Lambda_\varrho)}\Vert u \Vert _{ L^\infty(\Lambda_\varrho)}  +\varrho^{-2}
  \Vert u \Vert _{ L^\infty(\Lambda_\varrho)}^2\right]. 
 \end{equation}
\end{lemma}
\begin{proof} 
The argument is parallel to the proof of Lemma A.1 in \cite {BBH}, which
corresponds to its elliptic version.  Set $h=u_t-u_{xx}$,   let $(x_0, t_0)$ be
given in $\Lambda_\varrho^{1/2}$, and let $0<\mu \leq \frac{\varrho}{2}$ be a
constant to be determined in the course of the proof. We consider the function 
$$ 
v(y, \tau)=u\left (2\mu y+x_0, \, 4\mu^2 (\tau-1)+t_0)\right), 
$$
so that $v$ is defined on $\Lambda$ and satisfies there
$$ 
v_t-v_{yy}=\mu^2h ( \left (2\mu y+x_0, \, 4\mu^2 (\tau-1)+t_0)\right) \ {\rm on \ } \Lambda. 
$$
Applying Lemma \ref{lemmeBOS} to $v$ we are led to
\begin{equation*}
\begin{aligned}
\vert  v_y(0, 1)\vert  &\leq  C\left ( \mu^2 \Vert h\left( 2\mu y+x_0, \, 4\mu^2 (\tau-1)+t_0)\right) \Vert_{L^\infty (\Lambda)}+
  \Vert v \Vert_{L^\infty (\Lambda)}\right) \\
  &\leq   C\left ( \mu^2 \Vert h \Vert_{L^\infty (\Lambda_\varrho)}+\Vert u \Vert_{L^\infty (\Lambda_\varrho)}\right),
 \end{aligned}
 \end{equation*}
so that, going back to $u$, we obtain
\begin{equation}
\label{bbh1}
\mu   \vert  u_x(x_0, t_0)\vert  \leq C\left ( \mu^2 \Vert h \vert_{L^\infty
(\Lambda_\varrho)}+\Vert u\Vert_{L^\infty (\Lambda_\varrho)}\right). 
\end{equation}
We distinguish two cases: \\
\smallskip
\noindent
{\it Case 1:} $\displaystyle {\Vert u \Vert_{L^\infty}\leq \varrho^2  \Vert h \Vert_{L^\infty}}$. In this case we apply \eqref{bbh1} with
$\displaystyle{\mu=\left (\frac{\Vert u \Vert_{L^\infty}}{\Vert h \Vert_{L^\infty}} \right)^{\frac 12}.}$  This yields
$$ \vert  u_y(x_0, t_0)\vert \leq  2C\Vert u \Vert_{L^\infty}^{1/2} \Vert h \Vert_{L^\infty}^{1/2}.  $$
\smallskip
\noindent
{\it Case 2:} $\displaystyle {\Vert u \Vert_{L^\infty}\geq \varrho^2  \Vert h \Vert_{L^\infty}}$. In this case we apply \eqref{bbh1} with 
$\mu=\varrho$.  We obtain
\begin{equation}
\begin{aligned}
 \vert  u_x(x_0, t_0)\vert  \leq & C\left ( \varrho \Vert h \Vert_{L^\infty
(\Lambda_\varrho)}+\varrho^{-1}\Vert u\Vert_{L^\infty
(\Lambda_\varrho)}\right)\\ &\leq C\left ( \Vert h \Vert_{L^\infty
(\Lambda_\varrho)}^{1/2}\Vert u \Vert_{L^\infty (\Lambda_\varrho)}^{1/2} 
  + r^{-1}\Vert u\Vert_{L^\infty (\Lambda_\varrho)}
 \right).
\end{aligned}
\end{equation}
In both cases, we obtain the desired inequality. 
\end{proof}

%%%%%%%%%%%%%%%%%%%%%%%%%%%%%%%%%%%%%%%%%%%%%%%%%%
\subsubsection{Estimating the derivative of $\Wm$}
%%%%%%%%%%%%%%%%%%%%%%%%%%%%%%%%%%%%%%%%%%%%%%%%%%
Consider  the general situation where we are given  two functions $U$ and
$U_\eps$ defined for $(x,t) \in \Lambda_\varrho$ and such that $L_0(U)=0$ and $L_\eps
(U_\eps)=0$, where $s:=\eps^{-\omega}t$, so that, in view of \eqref{feps}, 
$$
\vert \partial_t (U-U_\eps)-\partial_{xx} (U-U_\eps) \vert 
  \leq  C\left[\vert U-U_\eps \vert (\vert U \vert^{2 \theta-2}+ \vert U_\eps
\vert^{2 \theta-2})+ \eps^{\frac{1}{\theta-1}} \vert U_\eps \vert^{2
\theta})\right] {\rm \ on \ } \Lambda_\varrho. 
$$
We deduce from \eqref{farbreton} applied to the difference $U-U_\eps$ that we have 
(we use the notation $\Vert\cdot \Vert= \Vert\cdot \Vert_{L^\infty
(\Lambda_\varrho)}$ for simplicity)
\begin{equation*}
  \Vert \left(  U-U_\eps\right)_{x} \Vert _{ L^\infty (\Lambda_\varrho^{1/2})}^2 \leq C 
 \ \Vert U-U_\eps \Vert ^2 \left( \Vert U\Vert^{2\theta-2}+
  \Vert U_\eps \Vert^{2 \theta-2} 
 +\varrho^{-2} \right) + 
  C \, \eps^{\frac{1}{\theta-1}}\Vert U-U_\eps \Vert \Vert U_\eps \Vert^{2 \theta}.
\end{equation*}
Similarly applying \eqref{farbreton} to $U$ and $U_\eps$  we obtain
\begin{equation*}
%\begin{aligned}
  \Vert \left(  U+U_\eps\right)_{x} \Vert _{ L^\infty (\Lambda_\varrho^{1/2})}^2 \leq C
%& 
  ( \Vert U\Vert^{2\theta}+
  \Vert U_\eps \Vert^{2 \theta}  
 +\varrho^{-2}  \left( \Vert U\Vert^2+ \Vert U_\eps \Vert^2
  \right)  
%\\
%&
+ \eps^{\frac{1}{\theta-1}}( \Vert U_\eps \Vert^{2 \theta+1} +
   \Vert U \Vert \Vert U_\eps \Vert^{2 \theta})), 
%\end{aligned}
\end{equation*}
so that 
\begin{equation}
\label{difference}
 \Vert \left(  U^2-U_\eps^2\right)_{x} \Vert _{ L^\infty (\Lambda_\varrho^{1/2})}^2
 \leq C \left[
 \Vert U-U_\eps \Vert^2 \mathcal R_1^\eps (U, U_\eps)  + 
 \Vert U-U_\eps \Vert \mathcal R_2 ^\eps(U, U_\eps)\right], 
\end{equation}
where we have set   
\begin{equation*}
  \left\{
  \begin{aligned}
  \mathcal R_1^\eps (U, U_\eps)&=( \Vert U\Vert^{2\theta-2}+
  \Vert U_\eps \Vert^{2 \theta-2} 
 +\varrho^{-2}) 
  (
   \Vert U\Vert^{2\theta}+
  \Vert U_\eps \Vert^{2 \theta}  
 +\varrho^{-2} ( \Vert U\Vert ^2
 + \Vert U_\eps \Vert^2) \\
  & \qquad  +\eps^{\frac{1}{\theta-1}} ( \Vert U_\eps \Vert^{2 \theta+1} +\Vert U \Vert \Vert U_\eps \Vert^{2 \theta})),
   \\
    \mathcal R_2^\eps (U, U_\eps)&= \eps^{\frac{1}{\theta-1}}
  \Vert U_\eps \Vert^{2\theta}
     ( \Vert U\Vert^{2\theta}+
 \Vert U_\eps \Vert^{2 \theta}  
 +\varrho^{-2} ( \Vert U\Vert^2 
 + \Vert U_\eps \Vert^2  ) \\
 & \qquad +\eps^{\frac{1}{\theta-1}}(\Vert U_\eps \Vert^{2 \theta+1} +\Vert U \Vert \Vert U_\eps \Vert^{2 \theta})
  ).
  \end{aligned}
  \right.
  \end{equation*}
We specify next the discussion to our original situation. Thanks  to the general
inequality \eqref{difference}, we are in position to establish: 
 
\begin{proposition} \label{prop:bonnetrouge}
If \eqref{attractplus} hold and $\eps$ is sufficiently small, then for any $s\in I^\eps(s_0)$ and every 
$x \in \tetakeps(s)$ we have the estimate 
 \begin{equation*}
  \label{clio4}
%\begin{aligned}
   \vert (\Wm)_x^2 (x)-\lambda^{-\frac{1}{ (\theta-1)} }
(\overset {\vee} {\mathfrak u}_{ r^\eps(s) }^+)_x^2 (x)\vert  \leq  C\,\eps^\frac{1}{\theta^2}.%(\dmine)^{-2-\frac{1}{(\theta-1)}} \Upsilon (\eps, \dmine)+   \\
 %  &C\eps^{\frac{1}{2(\theta-1)}} (\dmine)^{-2-\frac{2}{(\theta-1)}}    \Upsilon^{1/2} (\eps, \dmine).
%\end{aligned}
\end{equation*}
\end{proposition} 
 
\begin{proof} 
We apply  inequality \eqref{difference} on the cylinder $\Lambda_\varrho$ 
with  $\varrho=\frac{1}{16}\dminstar(s_0)$ and to the functions
   $U(y,\tau)=\Wm(y+x,\eps^{\omega}\tau+s)$ and $U_\eps(y,\tau)=\overset {\vee} {\mathcal U}_{ r^\eps(s)}^+(y+x)$. We first
estimate  $\mathcal R_1$ and $\mathcal R_2$. Since we have 
\begin{equation*}
  \label{motopsing}
  \vert  U(y, \tau)  \vert + \vert U_\eps\vert  \leq  C\dminstar(s_0)^{-\frac{1}{\theta-1}},\qquad \text{for } (y, \tau) \in \Lambda_\varrho,  
\end{equation*}
it follows that 
\begin{equation*}
  \mathcal R_1^\eps\, (U,U_\eps) \leq \dminstar(s_0)^{-4-\frac{2}{\theta-1}} \ \ {\rm and} \  \ 
  \mathcal R_2^\eps\, (U,U_\eps) \leq \eps^{\frac{1}{\theta-1}} \dminstar(s_0)^{-4-\frac{4}{\theta-1}}.
  \end{equation*}
Invoking inequality \eqref{farbreton} of Lemma \ref{farbreton0}, and combining
it with \eqref{deriveur} and the conclusion of Proposition  \ref{aplus},  we derive
the conclusion using a crude lower bound for the power of $\eps.$ 
\end{proof}

Similarly we obtain 
\begin{proposition}\label{prop:bonnetbleu}
If \eqref{attractmoins2} hold and $\eps$ is sufficiently small, then for any $s\in I^\eps(s_0)$ and every 
$x \in \tetakeps(s)$ we have the estimate 

\begin{equation}
  \label{clio5}
%\begin{aligned}
   \vert (\Wm)_x^2 (x)-\lambda^{-\frac{1}{ (\theta-1)} }
(\Umps_{ r^\eps(s) })_x^2 (x)\vert  \leq C\,\eps^\frac{1}{\theta^2}.%\, (\dmine)^{-2-\frac{1}{(\theta-1)}} \Upsilon (\eps, \dmine)+   \\
%   & C\eps^{\frac{1}{2(\theta-1)}} (\dmine)^{-2-\frac{2}{(\theta-1)}}    \Upsilon^{1/2} (\eps, \dmine).
%\end{aligned}
\end{equation}
\end{proposition}

\begin{proof}[Proof of Proposition \ref{vraiestime} completed] 
The proof of Proposition \ref{vraiestime} follows combining  Proposition
\ref{prop:bonnetrouge} in the attractive case and Proposition \ref{prop:bonnetbleu} in the repulsive
case with the estimates \eqref{variance2}.  
\end{proof}

%%%%%%%%%%%%%%%%%%%%%%%%%%%%%%%%%%%%%%%%%%%%%%%%%%%%%%%%%%%%%
\section{The motion law for prepared datas}
\label{motion}
%%%%%%%%%%%%%%%%%%%%%%%%%%%%%%%%%%%%%%%%%%%%%%%%%%%%%%%%%%%%%

In this section, we  present the

\begin{proof}[Proof of Proposition \ref{firststep1}] 

{\it Step 1.} 
First, by definition of $\rL_0$, assumption $({\rm H}_1)$ and estimate \eqref{theodebase1}, it follows that for
fixed $\rL\geq \rL_0,$ and for all $\eps$ sufficiently small (depending only on $\rL$),
$$
\mathfrak{D}_\eps(s) \cap \ILLLL \subset \IL \qquad \forall 0 \leq s \leq S,
$$
so that ${(\mathcal{C}_{\rL,S})}$ holds.

%\noindent
{\it Step 2.}
Since the assumptions of Corollary \ref{murino0} are met with the choice $s_0=0$ and $\rL=\rL_0$,
we obtain that for $\eps$ sufficiently small, $\mathcal{W P}_\eps^{\rL_0}(\udll, s )$ holds and $d_{\rm min}^{\eps,\rL}(s) \geq \frac12  d_{\rm min}^*(0)=\frac12 \min\{a_{k+1}^0-a_k^0,\ k=1,\cdots,\ell_0-1\},$ for all 
$s \in I^\eps(0)$, 
as well as the identities $J(s)=J(0)$, $\upsigma_{i(k\pm \frac 12)}(s)=\upsigma_{i(k\pm \frac 12)} (0)$ and $\dagger_k(s)=\dagger_k(0),$ for any $k \in J(0).$

%\noindent
{\it Step 3.}
We claim that for any $s_1\leq s_2 \in I^*(0),$ we have  
\begin{equation}\label{gruss} 
\underset {\eps \to 0}\limsup\,  (\disL (s_1, s_2)) =0.
\end{equation}
Indeed, let $\rL \geq \rL_0$ be arbitrary. We know from Step 1 that ${(\mathcal{C}_{\rL,S})}$ holds provided
$\eps$ is sufficiently small. By Proposition \ref{parareglo}, for $\eps$ sufficiently small there exists 
two times $s_1^\eps$ and $s_2^\eps$ such  that   $0<s_1^\eps\leq s_1\leq s_2\leq s_2^\eps$, $\vert s_i-s_i^\eps
\vert \leq \eps^{\omega+1}\rL$  and $\WP^{\mathcal \rL}(\udl, s_i^\eps)$ holds
for $i=1,2.$ From the second step and assumption $({\rm H}_1)$ we infer that  
$\matfEL (s_1^\eps)=\matfE^{\rL_0}(s_1^\eps)=\matfE^{\rL_0}(s_2^\eps)=\matfEL(s_2^\eps)$. Invoking Corollary \ref{murino0} we are
therefore led to the inequality 
$$
\dis^{\rL_0} (s_1, s_2) \leq \dis^{\rL} (s_1^\eps, s_2^\eps) \leq   C M_0\left(\frac{\eps}{\udl}\right)^\omega + C\rL^{-(\omega+2)}(s_2-s_1+2\eps^{\omega+1}\rL).
$$ 
Since $\rL\geq \rL_0$ was arbitrary the conclusion \eqref{gruss} follows letting first $\eps \to 0$ and then $\rL \to \infty.$

{\it Step 4.} 
In view of Corollary \ref{lartichaut} we may  find a subsequence
$(\eps_n)_{\in \N}$  tending to $0$ such that the functions
$a_k^{\eps_n}(\cdot)_{n \in \N}$ converge uniformly as $n \to 0$ on compact
subsets on $I^\star (0)$.   
Consider the cylinder  
$$
\displaystyle{\mathcal C_{k+\frac 12} ^*\equiv [a_k^0+\frac{1}{4}\dminstar(0), \,  a_{k+1}^0-\frac{1}{4}\dminstar(0)]\times   I^*(0)}. 
$$
It follows from Step 2 and Proposition  \ref{vraiestime} that 
\begin{equation}\label{arlette}
{\eps_n}^{-\omega}\xi_{\eps_n}(\mathfrak v_{\eps_n})  \to    \lambda_{i(k+\frac
12)}^{-\frac{1}{2( \theta-1)}} r_{k+\frac 12}(s)^{-(\omega+1)}\gamma_k  \   {\rm
as \ } \eps_n \to 0, {\rm \ for \ } k =1, \cdots  \ell_0-1  
\end{equation}
uniformly on every compact subset of $\displaystyle{\mathcal C_{k+\frac 12}
^*}$,  where $\gamma_k$ is defined in \eqref{choix} and where $r_{k+\frac12}(s)=a_{k+1}(s)-a_k(s).$ 

{\it Step 5.}  
As in \eqref{specificitechi}, we consider a test function $\chi\equiv \chi_{_k}$ with the following properties
\begin{equation*}
\label{specificitechibis}
\left \{
\begin{aligned}
& \chi {\rm \ has \   compact  \ support \  in } \
[a_{k}^0-\frac {1}{3} \dminstar(0)  , a_{k}^0+\frac {1}{3} \dminstar(0)],\\  
&  \chi{\rm \  is \ affine \ on \ the   \ interval } \  
[a_{k}^0-\frac {1}{4} \dminstar(0)  , a_{k}^0+\frac {1}{4} \dminstar(0)],
{\rm \ with \ } \chi'=1 \ {\rm there} \\ 
&  \Vert \chi \Vert_{L^\infty (\R)} \leq C\dminstar(0),   \Vert \chi' \Vert_{L^\infty
(\R)} \leq  C  {\rm \ and \ } \Vert \chi''\Vert_{L^\infty (\R)} \leq
C{\dminstar(0)}^{-1}. 
\end{aligned}
\right.
\end{equation*}
 It follows from the definition of $\chi_k$ that 
  $\chi_k''=0$ outside $a$, and so is $\eps^{-\omega}\xi_\eps (\vm) \chi''_k.$ 
It follows from \eqref{arlette} that for $s_1\leq s_2 \in I^*(0),$ 
\begin{equation}\label{pinder}
  \begin{aligned}
  \mathfrak  F_{\eps_n} (s_1, s_2, \chi_{_k} )  \to   & \left(\int_{a_{k}^0-\frac{1}{3}\dminstar(0)}^{a_{k}^0-\frac{1}{4}\dminstar(0) } \chi'' (x) dx\right) \left( \int_{s_1}^{s_2} 
   \lambda_{k-\frac 12}^{-\frac{1}{2( \theta-1)}} r_{k-\frac 12}(s)^{-\frac{1}{\theta-1}}\gamma_{k-\frac 12}  
   ds\right)  + \\
   &\left(\int_{a_{k}^0+\frac{1}{4}\dminstar(0)}^{a_{k}^0+\frac{1}{3}\dminstar } \chi'' (x) dx\right)\left( \int_{s_1}^{s_2} 
   \lambda_{k+\frac 12}^{-\frac{1}{2( \theta-1)}} r_{k+\frac 12}(s)^{-\frac{1}{\theta-1}}\gamma_{k+\frac 12}  
   ds\right)
  \end{aligned}
 \end{equation}
 as $\eps_n \to 0.$ 
Since the above two integrals containing $\chi''$ are identically equal to $1$ and $-1$ respectively,   
we finally deduce from \eqref{gratindo} combined  with \eqref{gruss} and \eqref{pinder}, letting $\eps_n$ tend to $0$,  that
for $s_1\leq  s_2 \in I^*(0)$  we have  
$$
[a_k(s_1)-a_k(s_2)] \S_{i(k)}=    \int_{s_1}^{s_2} \left (
   \lambda_{i(k-\frac 12)}^{-\frac{1}{2( \theta-1)}} r_{k-\frac 12}(s)^{-(\omega+1)}\gamma_{k-\frac 12}  -
  \lambda_{i(k+\frac 12)}^{-\frac{1}{2( \theta-1)}} r_{k+\frac 12}(s)^{-(\omega+1)}\gamma_{k+\frac 12}  \right) ds, 
$$ 
which is nothing else than the integral formulation of the system \eqref{tyrannosaure}. 
Since the latter possesses a unique solution, the
limiting points are unique and therefore convergence of the $a_k^\eps$ for $s\in I^*(s)$ holds for the full
family $(\vm)_{\eps >0}.$ 

{\it Step 6.} We use an elementary  continuation method to extend the convergence from $I^*(0)$ to the full
interval $(0,S).$  Indeed, as long as $\dminstar(s)$ remains bounded from below by a strictly positive constant (which holds, by definition of $S_{\rm max}$, as long as $s<S$) we may take $s$ as a new origin of times (Step 2 yields $\wpz(\upalpha_1\eps,s)$) and use Steps 1 to 5 to extend the stated convergence past $s.$  The proof is here completed. 
\end{proof}

%%%%%%%%%%%%%%%%%%%%%%%%%%%%%%%%%%%%%%%%%%%%%%%%%%%%%%%%%%%%%%%%%%%%%%%%%%%%%%%%%%%
\section{Clearing-out }
\label{clearstream}
%%%%%%%%%%%%%%%%%%%%%%%%%%%%%%%%%%%%%%%%%%%%%%%%%%%%%%%%%%%%%%%%%%%%%%%%%%%%%%%%%%%

The purpose of this section is to provide a proof to Proposition
\ref{nettoyage}.    We are led to consider the situation  
where for some length  $\rL \geq 0$ we have
\begin{equation}
\label{degrossir}
%\begin{aligned}
\mathfrak  D_\eps(0) \cap  [-5 \rL, 5\rL ] \subset [-\upkappa_0\rL, \upkappa_0\rL] % &{\rm  \ for \ any \ } \tau \in \mathfrak J_1=(0, S_1 ), \\
 %\mathfrak  D_\eps(\tau) &\cap  [-4 \rL_0, 4\rL_0 ] \subset [-\rL_0, \rL_0 ] &{\rm  \ for \ any \ } \tau \in \mathfrak J_2=(0, S_2 ), \\
%& S_2-S_1 \geq  \mathcal K_1 ( \upkappa_1 \rL_0)^{\omega+2}, 
%\end{aligned}
%\right.
\end{equation}
for some (small) constant $\upkappa_0\leq \frac12.$ It follows from Theorem \ref{maintheo1} that
$$
\mathcal{C}_{\rL,S} \ \text{holds},\qquad\text{where } S = \uprho_0 \left(\frac{\rL}{2}\right)^{\omega+2}, 
$$
and that for $s \in [0,S]$ we have
\begin{equation}
 \label{eq:lafrite}
\mathfrak  D_\eps(s) \cap  [-4 \rL, 4\rL ] \subset [-\upkappa_0(s)\rL, \upkappa_0(s)\rL] 
\end{equation}
where
\begin{equation}\label{eq:buche}
\upkappa_0(s) := \upkappa_0 + \left( \frac{s}{\uprho_0}\right)^\frac{1}{\omega+2}\frac{1}{\rL}.
\end{equation}
For those times $s \in [0,S]$ for which the preparedness
assumption $\wpi(\upalpha_1 \eps,  s)$  holds we set  
\begin{equation*}
\left\{
\begin{aligned}
  &\dminep (s)=\min   \{ \vert a_{k+1}^\eps (s)- a_k^\eps (s) \vert,   \ k\in J^+(s)  \},{\rm  \  \   and} \\
  &\dmines (s)=\min   \{ \vert a_{k+1}^\eps (s)- a_k^\eps (s) \vert,   \ k \in J^-(s)\}, 
\end{aligned}
 \right.
\end{equation*}
with  $J^\pm(s)=\{ k \in \{1,\cdots, \ell(s)-1\}, \hbox{ s. t }    \epsilon_{k+\frac 12}=\mp 1 \}$,
so that $\dmine(s)=\min \{ \dminep(s), \dmines(s)\}$,  with the convention that
the quantities  are equal to $\rL$ in case the defining set is empty.

\medskip
At first, we will focus on the case $J^-(s)\not = \emptyset$. 
The following
result provides an upper bound in terms of $\dmines(s)$  for a
dissipation time for the quantized function $\matfE^{\rL}$. This phenomenon is related
to the cancellation of \emph{a front with its anti-front}, and is the main building block
for the proof of Proposition \ref{nettoyage}.

\begin{proposition}
\label{rtcolplus0}
There exist $\upkappa_1>0,$ $\upalpha_3 >0,$ and $\mathcal K_{\rm col}>0,$  all depending only on $V$ and $M_0$, with the following properties.  If \eqref{degrossir} holds, if $s_0\in (\eps^\omega \rL^2,S)$ is such that $\upkappa_0(s_0)\leq \upkappa_1,$ $\wp(\upalpha_3 \eps, s_0 )$ holds, $J^-(s_0)$ is non empty, and $s_0 + \mathcal K_{\rm col} \dmines(s_0)^{\omega+2}<S,$ then there exists some time  $\rtcolp(s_0)\in (s_0,S)$ such that  $\wpi(\upalpha_3\eps, \rtcolp(s_0))$ holds,
\begin{equation}
  \label{thedissip}
  \matfE^{\rL}(\rtcolp(s_0))\leq \matfE^{\rL}(s_0)- \upmu_1, 
\end{equation}
where $\upmu_1$ is a constant introduced in Lemma \ref{avecdec}, and 
\begin{equation}
 \label{estimertcol0}
  \rtcolp (s_0)- s_0  \leq   \mathcal  K_{\rm col}   \left (  \dmines(s_0)\right)^{\omega+2}.
\end{equation}
\end{proposition}
We postpone the proof of Proposition \ref{rtcolplus0} to after Section \ref{sect:contrad} below, where
we will analyze more into details  the attractive  and repulsive forces at work at the $\eps$ level. We will
then prove Proposition \ref{rtcolplus0} in Section \ref{proofouf}, and finally Proposition \ref{nettoyage} in Section \ref{prooof}.

%\begin{corollary}\label{cor:liquidepaires}
%There exists $\upkappa_*>0$ and $\mathcal K_*>0$, both depending only on $M_0$ and $V$ such that if 
%\eqref{degrossir} holds for some $\upkappa_0\leq \upkappa_*$ then 
%$$
%\wp(\upalpha_1,s) \text{ holds and } \dmineL(s)\geq \upkappa_0\rL \text{ for all } \mathcal K_* (\upkappa_0\rL)^{\omega+2} \leq s %\leq \frac{S}{2}.
%$$
%\end{corollary}

%%%%%%%%%%%%%%%%%%%%%%%%%%%%%%%%%%%%%%%%%%%%%%%%%%%%%%%%%%%%%%%%%%
\subsection{Attractive and repulsive forces at the \mathversion{bold}$\eps$\mathversion{normal} level} \label{sect:contrad}
 In this subsection we consider the general situation  
where $\mathcal{C}_{\rL,S}$ holds, for some length  $\rL\geq 0$ and some $S>0.$ 

\medskip

In order to deal with the attractive and repulsive forces underlying
annihilations or splittings,  we set 
$$\mathcal
F_{k+\frac12}(s)=-\omega^{-1}\mathcal B_{k+\frac 12}
\left(a_{k+1}^\eps(s)-a_k^\eps(s)  \right)^{-\omega}
$$ 
and  consider the
positive functionals 
\begin{equation}
     \label{repattract}
     \mFrep (s)=\underset{k\in J^+(s)}\sum{\mathcal  F}_{k+\frac 12} (s), \  \mFat(s)=- \underset{k\in J^-(s)}\sum{\mathcal  F}_{k+\frac 12} (s), 
     \end{equation}
with the convention that the quantity is equal to $+\infty$ in case the defining set is empty.
For some constants $0<\kappa_2\leq \kappa_3$ depending only on $M_0$, we have
and $V$ 
\begin{equation}
\label{globe}
\left\{
\begin{aligned}
\kappa_2 \, \mFat (s)^{-\frac{1}{\omega}} &\leq    \dmines(s) \leq \kappa_3\mFat (s)^{-\frac{1}{\omega}},  \\
\kappa_2 \, \mFrep (s)^{-\frac{1}{\omega}}& \leq    \dminep(s) \leq \kappa_3\mFrep (s)^{-\frac{1}{\omega}}.
\end{aligned}
\right.
\end{equation}

\medskip
Let $s_0 \in [\eps^\omega \rL^2,S]$ be such that 
\begin{equation}\label{eq:bondebut}
 \wp(\upalpha_2\eps,s_0) \text{ holds} \qquad \text{and} \qquad \dmineL(s_0) \geq 16 \mathfrak q_1(\alpha_2)\eps. 
\end{equation}
We consider as in Corollary \ref{parareglo3} the stopping time
$$
\mathcal T_0^\eps(\upalpha_2,s_0) = \max \left\{ s_0+\eps^{2+\omega} \leq s\leq S \quad \text{s.t.}\quad \dmineL(s')\geq 8\q_1(\upalpha_2)\eps \quad \forall s'\in [s_0+\eps^{\omega+2},s]\right\},
$$ 
and for simplicity we will write 
$\mathcal T_0^\eps(s_0) \equiv \mathcal T_0^\eps(\upalpha_2,s_0).$ 
In view of \eqref{eq:bondebut} and the statement of Corollary \ref{parareglo3}, 
$$
\wp(\upalpha_1 \eps, s) \text{ holds} \qquad \forall\ s \in \mathcal I^\eps_0(s_0)\equiv [s_0+\eps^{2+\omega},\mathcal T_0^\eps(s_0)].
$$
The functionals $\mFat$ and $\mFrep$ are in particular well
defined and continuous on the interval of time $\mathcal I_0^\eps(s_0)$ with $J^+(s) =J^+(s_0)$ and $J^-(s)=J^-(s_0)$ for all $s$ that interval. Note that the attractive forces are dominant when $\dmines(s)\leq \dminep(s)$ and in contrario the repulsive forces
are dominant when $\dminep(s) \leq \dmines(s).$ 

\medskip
We first focus on the attractive case, and for $s\in \mathcal I_0^\eps(s_0),$ we introduce the new stopping times
  $$\mathcal T_1^\eps (s)=\inf \{s\leq s'\leq \mathcal T_0^\eps(s_0), \,  \Fat (s') \geq
  \upupsilon_1^{\omega} \Fat(s)  \ \,   {\rm or \ }  s'=\mathcal T_0^\eps(s_0)\}, 
  $$
  where $\upupsilon_1=10 \kappa_3^2 \kappa_2^{-2},$ so that  $\upupsilon_1>10$ and 
  $\mathcal T_1^\eps(s) \leq \mathcal T_0^\eps(s_0).$
 In view of \eqref{globe}, we have
 %if     $
 \begin{equation}
 \label{majord}
   \frac{1}{10} \left(\frac{\kappa_2}{\kappa_3}\right)^3\,  \dmines(s) \leq
 \dmines(\mathcal T_1^\eps (s)),
 \end{equation}
 and if $\mathcal T_1^\eps(s) < \mathcal T_0^\eps(s_0)$ then 
 \begin{equation}\label{majordbis}
 \dmines(\mathcal T_1^\eps (s)) \leq \frac{1}{10} \frac{\kappa_2}{\kappa_3}\dmines(s)
 \leq \frac{1}{10}\dmines(s).
 \end{equation}

The next result provides an upper bound on $\mathcal T_1^\eps (s)-s.$
Central in our argument is Proposition \ref{firststep1}, which we use  combined with various arguments by contradiction. 
We have 

\begin{proposition}
\label{nihil}
There exists $\upbeta_0>0$, depending only on $V$ and $M_0$,  
with the following properties.   If $J^-(s_0)\neq \emptyset$, $\hat s\in \mathcal I_0^\eps(s_0)$ 
and   
 \begin{equation}
 \label{margeaux}
\upbeta_0\, \eps  \leq  \dmines(\hat s) \leq \dminep(\hat s), 
 \end{equation}
then we have
\begin{equation}
\label{loufdingue}
 \mathcal T_1^\eps (\hat s)-\hat s \leq   \mathcal  K_0 \,\left(  \dmines(\hat s)\right)^{\omega +2}, \ \, 
 \end{equation}
 where $\mathcal K_0$ is defined in \eqref{gedelope}, and moreover if $\mathcal T_1^\eps(\hat s) < S$ then
   \begin{equation}
      \label{louf2}
   \dmines (\mathcal T_1^\eps(\hat s)) \leq   \dminep (\mathcal T_1^\eps(\hat s)).
   \end{equation}
    \end{proposition}
 
 \begin{proof} Up to a translation of times we may first assume that $\hat s=0$, which eases somewhat the notations. We then argue by contraction and assume that the conclusion is false, that is,  there does not exist any such  constant $\upbeta_0$, no matter how large it is chosen,   such that the conclusion holds. Taking $\upbeta_0=n$, this  means that given any  $n \in \N_*$  there exist some $0<\eps_n\leq 1$, a solution $v_n$ to $({\rm PGL})_{\eps_n}$  such that $ \mathcal E_{\eps_n} (v_n) \leq M_0$, such that ${\rm  W P }_{\eps_n}^{\scriptscriptstyle \rL_0} (\upalpha_1 \eps_n, 0 )$ holds,  such that 
 \begin{equation}
 \label{contrad1}
 n\eps_n \leq   \dmin^{\eps_n, -}(0)=\dmin^{\eps_n}(0)\leq \dmin^{\eps_n, +}(0), 
 \end{equation}
 and such that one of the conclusion fails, that is  such that either
   \begin{equation}
   \label{contrad2}
  \mathcal T_1^n\equiv  \mathcal T_1^{\eps_n} (0) >   \mathcal  K_0 \,(\dmines(0))^{\omega +2}, 
 \end{equation}
 or 
 \begin{equation}
    \label{contrad3}
   \dmines (\mathcal T_1^n )>   \dminep (\mathcal T_1^n).
 \end{equation}
Setting $S_0^n=\mathcal  K_0 (\dmin^{\eps_n, -}(0))^{\omega+2}$, relation \eqref{contrad2}  may be rephrased as
\begin{equation}
\label{contrad2prime}
\Fat^n(s) \leq  \     \upupsilon_1^{\omega}\Fat^n(0) \ \  {\rm and }   \  \  
\dmin^{\eps_n}(s) \geq 8\mathfrak q_1(\upalpha_2) \eps_n\ {\rm \  for \  any  \  }
 s \in [0, S_0^n], 
\end{equation}
 where the superscripts $n$ refer to the corresponding functionals computed for the map $v_n$.
  Passing possibly to a subsequence, we may therefore  assume that one  at least of the properties  \eqref{contrad2prime} or \eqref{contrad3} holds for any $n \in \N_*$. 
  Also, passing possibly to a further subsequence, we may assume that the total number of fronts of $v_n(0)$ inside $[-\rL,\rL]$ is constant, equal to a number $\ell$,  denote $a_1^n (s), \cdots,  a_\ell^n (s)$  the corresponding front points,  for $s \in  [0,  \mathcal T_1^n],$ and set $d_n^-(s)=\dmin^{\eps_n, -}(s), 
d_n^+(s)=\dmin^{\eps_n, +}(s),d_n(s)=\dmin^{\eps_n}(s).$

   In order to obtain a contradiction  we shall make use of
 \emph{the scale invariance} of the equation: if $v_\eps$ is a solution to $({\rm PGL})_\eps$ then the map
 $\displaystyle{ \tilde v_{\tilde \eps} (x, t)=v_\eps (rx, r^2 t)}$ 
 is a solution to  $({\rm PGL})_{\tilde \eps}$ with $\tilde \eps= r^{-1}\eps.$
  As scaling factor $r_n$, we choose $r_n= \dmin^{\eps_n, -}(0)\geq n \eps_n$ and set   
\begin{equation}
\label{harry}
 \tilde v_n(x, t)= v_n (r_n x, r_n^2 t), \ \, 
 \tilde{ \mathfrak v}_n(x, \tau)=\tilde v_n(x, \tilde{ \eps}_n^{-\omega} \tau), 
 \end{equation}
 so that $\tilde v_n$ is a solution to $({\rm PGL})_{\tilde \eps_n}$ satisfying $ {\mathcal  WP}_{\tilde\eps_n}^{ L_n} (\upalpha_2 \tilde \eps_n, 0)$
 with $L_n=r_n^{-1} \rL$ and 
 $$\tilde \eps_n=(r_n)^{-1} \eps_n=(\dmin^{\eps_n, -}(0))^{-1} \eps_n\leq \frac{1}{n}, \   {\rm  \ hence  \ we \ have \   } \eps_n \to 0 {\rm \ as  \ } n \to + \infty.  
 $$ 
 The points 
  $\tilde a_1^n (s)=r_n^{-1} a_1^n (r_n^{-(2+\omega)} s), \cdots, \tilde a_\ell^n (s)= r_n^{-1} a_\ell^n (r_n^{-(2+\omega)}s)$  are  the front points of  $\tilde{ \mathfrak v}_n$. Let  $\tilde d_n^-, \ \tilde d_n^+,\  \tilde d_n$  be the quantities corresponding to $\dmines, \dminep, \dmine$ for  $\tilde{ \mathfrak v}_n$, so that
  $$\tilde d_n^-(s)=r_n^{-1}d_n^-(r_n^{-(2+\omega)} s), \ \, 
  \tilde d_n^+(s)=r_n^{-1}d_n^+(r_n^{-(2+\omega)} s), \  {\rm \ and \ } \, 
  \tilde d_n(s)=r_n^{-1}d_n (r_n^{-(2+\omega)} s),
  $$
and notice that $d_n^- (0)=d_n(0)=1$. We next distinguish the following two complementing cases.\\

  \noindent
  {\it Case 1: \eqref{contrad2prime} holds for all $n\in \N_*$.} 
  It follows from  assumption \eqref{contrad2prime} that  ${\rm W P}_{\tilde\eps_n}^{L_n} (\upalpha_1 \tilde \eps_n, \tau)$ holds for every  
 $\tau \in (0, \tilde { S}_1^n)$, where  
 $\tilde {\mathcal S}_1^n=r_n^{-(2+\omega)} S_0^n=\mathcal K_0$.
 Let  $k_0 \in \{1, \cdots, \ell\}$ be  such that
$$ a_{k_0+1}^n (0)-a_{k_0}^n (0)=\dmin^{\eps_n, -}(0). $$
Upon a translation  if necessary, we may  also assume  that 
$\displaystyle {a_{k_0}^n (0)=0}$  so that 
$\displaystyle{ a_{k_0+1}^n (0)=\dmin^{\eps_n, -}(0)}$.
 We denote by $ \tmFatn$ the functional $\mFat$ computed for the front points of $\tilde{ \mathfrak v}_n$,   so that
 $$  \tmFatn( r_n^{-(2+\omega)} s)= r_n^{(2+\omega)} \mFatn(s).$$
  By construction we have
  \begin{equation}
 \label{lequipe} 
 \tilde a_{k_0}^n(0)=0  {\rm \ and \ }  \tilde a_{k_0+1}^n (0)=1=\tilde d_n^-(0).
  \end{equation} 
   Since
 $\displaystyle{\tilde \eps_n \to 0 {\rm \ as  \ } n \to \infty}$,  we may implement part of the already established asymptotic analysis for $({\rm PGL})_\eps$ on the sequence $(\tilde v_n)_{n\in \N}$.
 First, passing possibly to a subsequence, we may assume that for some subset $\tilde J \subset J(0)$ the points  $\{\tilde a_k (0)\}_{k \in \tilde J} $ converge to some  finite   limits $\{\tilde a_k^0\}_{k \in \tilde J}$, whereas the points with indices in $J(0)\setminus \tilde J$ diverge either to $+\infty$ or to $-\infty$.  We choose $\tilde \rL\geq 1$ so that
 \begin{equation}
 \label{confetti}
 \underset {k \in \tilde J}  \cup \{\tilde a_k^0\} \subset [-\frac{\tilde \rL}{2},\frac{ \tilde  \rL}{2}].
 \end{equation}
%We   apply Proposition \ref{parareglo} to assert that there exists some time $s_n \in  (0, 2 M_0 \tilde \eps_n^{\omega+1})$ such that 
% $\mathcal {WP}_{\tilde \eps_n}^{\tilde \rL}  (\updelta^{\tilde \eps_n}_{\rm log} , s_n)$ holds, and such that
 %\begin{equation}
 %\label{lhoste}
 %\vert \tilde a_k^{n} (s_n)-\tilde a_k^{n} (0) \vert \leq \left(\frac{4 M_0}{\uprho_0} \right)^{\frac{\omega +1}{\omega+2}}{\tilde \eps}_n^{\frac{\omega+1}{\omega+2}}+\rc_1\tilde \eps_n \to 0  {\rm \ as  \ } n \to \infty.
 %\end{equation}
 %Passing possibly to a further subsequence, we may assume  that the front points $(\tilde a_{k}^n(0))_{n\in \N}$ converge as $n\to \infty$  in $\bar \R$ to some limits $(\tilde a_k(0))_{k \in J}$.  Notice that\emph{ some points} may escape at infinity: we denote by $J_0$ the set of indices of points those limits is bounded
 % It follows from \eqref{lhoste} that  the front points  $(\tilde a_{k}^n(s_n))_{n\in \N}$ have the same limits as  the points $(\tilde a_{k}^n(0)_{n\in \N}$ . 
 In view of \eqref{lequipe},  we have 
 $\displaystyle{
 \tilde a_{k_0}(0)=0,  \tilde a_{k_0+1}=1  {\rm  \ and  \ } \   \inf\{\vert \tilde a_{k+1}(0)- \tilde a_k(0)\vert, k \in \tilde J\}=1.
 }
 $
 We are hence  in position to apply  the convergence result stated in  Proposition \ref{firststep1} to the sequence $(\tilde{ \mathfrak v}_n(\cdot))_{n \in \N}$. It states that the front points $(\tilde a_k^n(\tau))_{k\in J_0} $ which do not escape at infinity  converge  to the  solution $(\tilde a_k(\cdot))_{k \in \tilde J}$  of the ordinary differential equation \eqref{tyrannosaure} supplemented with the corresponding initial values $(\tilde a_k(0))_{k\in \tilde  J},$ uniformly in time on every compact subset of $(0, \tilde S_{\rm max})$, where $\tilde S_{\rm max}$ denotes the maximal time of existence for the solution.   In particular, we have
 \begin{equation*}
 \left\{
 \begin{aligned}
& \tilde d_n^-(\tau)\to \textswab d_{\tilde a}^- (\tau), \hbox{ uniformly on every compact subset of  }
 (0, \tilde S_{\rm max}),  \\
 & \limsup_{n\to +\infty} \,  {\tFat}^n (\tau ) \geq \Fat (\tilde a(\tau))  {\rm  \ for  every \ } \tau \in    (0, \tilde S_{\rm max}),
 \end{aligned}
 \right.
 \end{equation*}
the presence of the $\limsup$ being related to the  fact that some points might escape at infinity so that the limiting  values of the functionals are possibly smaller. 
We use next   the properties of the differential equation \eqref{tyrannosaure} established in Appendix B.   We first invoke Proposition  \ref{getrude}  which
  asserts that    $\tilde S_{\rm max} \leq \mathcal  K_0$ and that
  $$ \Fat (\tilde a(\tau) )\to +\infty {\rm \ as  \  }  \tau\to \tilde S_{\rm  max}.$$
Hence, there exists some  $\tau_1 \in  (0, \tilde S_{\rm max})\subset (0, \mathcal  K_0)$ such that, if $n$ is sufficiently large, then
  $$ \tmFatn(\tau_1) >  \upupsilon_1^\omega  \tmFatn (0). $$
 Going back to the original time scale, this yields
 $\displaystyle{ \mFatn( r_n^{\omega+2} \tau_1 ) > \upupsilon_1^\omega \mFatn (0).}$
Since $ r_n^{\omega+2} \tau_1 \in  (0, r_n^{2+\omega} \mathcal  K_0)= (0, S_0^n)$ this contradicts \eqref{contrad2prime} and completes the proof in  Case 1. \\

 \noindent
 {\it  Case 2: \eqref{contrad3}  holds for all $n\in \N_*$.}   
 %It follows that the stopping time $\mathcal T_1^n$ satisfies  
 %$$\mathcal T_1^n \leq S_1^n=\mathcal  K_0\left( \dmin^{\eps_n, -}(0)\right)^{\omega+2}, $$
 %and that \eqref{majord},\eqref{majordbis} hold. 
 We consider an arbitrary index  $j \in J^+$. As above, translating the origin, we may assume without loss of generally that $a_j^n (0)=0$.   We also define  the map $\mathfrak v_n$ as in Case 1,  according to the  same scaling  as described  in \eqref{harry},  the only difference  being that the origin has  been  shifted  differently. With similar notations, we have
 $$ \tilde a_{j}^n(0)=0  {\rm \ and \ }  \tilde a_{j+1}^n (0) \geq 1=\tilde d_n^-(0).$$
Passing possibly to a further subsequence, we may  assume  that the front points at time $0$ converge to some limits in $\bar \R$ denoted $\tilde a_k(0)$.  We are hence  again  in position to apply  the convergence result of   Proposition \ref{firststep1},  so that the front points $(\tilde a_k^n(s))_{k\in J_j} $ which do not escape at infinity  converge  to the  solution $(\tilde a_k(\cdot))_{k \in J_j}$  of the ordinary differential equation \eqref{tyrannosaure} supplemented with the corresponding initial values $(\tilde a_k(0))_{k\in J_j},$ uniformly in time on every compact subset of $(0, \tilde S'_{\rm max})$, where $\tilde  S'_{\rm max}$ denotes the  (new) maximal time of existence for the solution.  
It follows from assumption \eqref{demine}, Theorem \ref{mainbs2} and scaling that
%
%Going back to \eqref{majord}-\eqref{majordbis}  we are led to
%\begin{equation}
%\label{cyan}
% \frac{1}{10} \left(\frac{\kappa_2}{\kappa_3}\right)^3\ \leq
%  \dmin^{\eps_n, -}( \tilde {\mathcal T}_1^n) 
% \leq 
% \frac{1}{10} \frac{\kappa_3}{\kappa_2}, 
% \end{equation}
% which implies  in particular that 
 $\displaystyle{ 0<\tilde {\mathcal T_1}\equiv \liminf \tilde {\mathcal T}_1^n \leq \tilde  S'_{\rm max}.}$
 We claim that, for any $\tau \in (0, \tilde {\mathcal T}_1)$, and for sufficiently large $n$,  we have
 \begin{equation}
 \label{ouijeclaim}
     \vert \tilde a_j^n (\tau)- \tilde a_{j+1}^n(\tau) \vert \geq \frac{\kappa_2}{2\kappa_3} .
     \end{equation}
 This is  actually a property of the differential equation \eqref{tyrannosaure}. We have indeed, in view of Proposition \ref{geducci},  
 $\displaystyle{0< \Frep (\tilde a (\tau) ) \leq \Frep (\tilde a (0)), }$
  so that it follows from \eqref{globe} that
  $$  \vert \tilde a_j (\tau)- \tilde a_{j+1}(\tau) \vert  \geq \frac{\kappa_2}{\kappa_3},$$ which yields \eqref{ouijeclaim} taking the convergence into account. Since \eqref{ouijeclaim} holds for any $j$, we deduce that 
  $$
  \dmin^{\eps_n, +}(\mathcal T_1^n) \geq \frac{\kappa_2}{2\kappa_3}\dmin^{\eps_n,-}(0)
  $$ 
  and therefore by \eqref{contrad3} we have
  \begin{equation}\label{eq:bmme}
  \dmin^{\eps_n}(\mathcal T_1^n) = \dmin^{\eps_n, -}(\mathcal T_1^n) \geq \dmin^{\eps_n, +}(\mathcal T_1^n) \geq \frac{\kappa_2}{2\kappa_3}\dmin^{\eps_n,-}(0) \geq \frac{\kappa_2}{2\kappa_3}n \eps.
  \end{equation}
 For $n$ sufficiently large, this implies that $\mathcal T_1^n < \mathcal T_0^n,$ and therefore from \eqref{majordbis} we have 
 $$
 \dmin^{\eps_n,-}(\mathcal T_1^n) \leq \frac{1}{10} \frac{\kappa_2}{\kappa_3}\dmin^{\eps_n,-}(0),
 $$
 which is in contradiction with \eqref{eq:bmme}.
 \end{proof}
  
 We turn now to the case where $\dminep(s)\leq \dmines(s).$ In order to handle the repulsive forces at work, 
 for $s\in \mathcal I_0^\eps(s_0)$ we introduce the  new stopping times 
 $$\mathcal T_2^\eps (s)=\inf \{s\leq s'\leq \mathcal T_0^\eps(s_0), \,  \mFrep (s') \leq
  \upupsilon_2^{\omega} \mFrep(s)  \ \,   {\rm or \ }  s'=\mathcal T_0^\eps(s_0)\}, 
  $$
where $\upupsilon_2= \frac{\kappa_2^2}{10\kappa_3^2},$  so that $\upupsilon_2<1$. 
 Notice that, in view of \eqref{globe}, we have, if   $\mathcal T_2^\eps(s) < \mathcal T_0^\eps(s_0)$,
 \begin{equation}
 \label{majordplus}
 \dminep(\mathcal T_2^\eps (s))\geq \upupsilon_2^{-1} \frac{\kappa_2}{\kappa_3} \dminep (s)\geq 
 10\,  \dminep (s).   
  \end{equation}
% It follows in particular from  \eqref{majord} that there exists some constant $\upbeta_0>0$ such that, if 
 %$\upbeta_0\eps\leq \dmines (s) < +\infty$, then $\mathcal T_1^\eps(s) < \mathcal T_0^\eps(s). $
With  $\mathcal S_1$ introduced in Proposition \ref{getrude}, we set
 \begin{equation}
 \label{choixk1}
 \displaystyle{\mathcal  K_1=\mathcal S_1^{-\omega}\left (\frac{2 \kappa_3}{\kappa_2\upupsilon_2} \right)^{\omega+2}}.
 \end{equation}
 %where the constant $\mathcal S_1$ is introduced in Corollary \ref{gedinox}.

   \begin{proposition}
 \label{spillit}
 There exists $\upbeta_1>0$, depending only on $V$ and $M_0$,  
with the following properties.   If $J^+(s_0)\neq \emptyset$, $\hat s\in \mathcal I_0^\eps(s_0)$ 
and   
 \begin{equation}
 \label{margeauxbis}
\upbeta_1\, \eps  \leq  \dminep(\hat s) \leq \dmines(\hat s), 
 \end{equation}
then we have  
\begin{equation}
\label{loufdingo}
 \mathcal T_2^\eps (\hat s)-\hat s \leq   \mathcal  K_1 \,\left(  \dminep(\hat s)\right)^{\omega +2}, 
    \end{equation}
and if $\mathcal T_2^\eps(\hat s)<S$ then $\mathcal T_2^\eps (\hat s)< \mathcal T_0^\eps(s_0)$ and 
for any $s \in [\hat s, \mathcal T_2^\eps (\hat s)]$, we have
  \begin{equation}
\label{rakotoson}
\dmine(s) \geq  \frac{1}{2} \mathcal S_2 \dminep(\hat s), 
\end{equation}
and
\begin{equation}
 \label{loufdingo2}
 \mFat\left (s\right)^{-\frac{1}{\omega}}
 \leq
  \mFat( \hat s)^{-\frac{1}{\omega}}
  +\frac{1}{\kappa_3}  ( \dminep(\hat s)), 
 \end{equation}  
where $\mathcal S_2$ is defined in Proposition \ref{getrude} and $\kappa_3$ is defined in \eqref{globe}. 
 \end{proposition}

\begin{proof} The argument possesses strong similarities with   the proof of Proposition \ref{nihil}, we therefore just sketch its main points,  in particular relying implicitly on the notations introduced there, as far as this is possible. By translation in time we also assume that $\hat s=0$ and  argue by contradiction assuming  that for any   $n \in \N_*$  there exist some $0<\eps_n\leq 1$, a solution $v_n$ to $({\rm PGL})_{\eps_n}$  such that $ \mathcal E_{\eps_n} (v_n) \leq M_0, {\mathcal W P}_{\eps_n}^{\rL} (\upalpha_1 \eps_n, 0 )$ holds,  such that 
 $n\eps_n \leq   d_n^ +(0)$, 
 and such that  either, we have for any $s\in (0, S_1^n)$, where $
 S_1^n=\mathcal  K_1 d_n^+(0)^{\omega+2}, $ $d_n(s) \geq 8\mathfrak q(\upalpha_2) \eps_n  $ and 
\begin{equation}
\label{contrad2ter}
 \kappa_3^{\omega} (d_n^+ (s'))^{-\omega} \geq \Frep^n(s') \geq  \ 
     %\left( \frac{\kappa_2^2}{10\kappa_3^2}\right)^{\omega}\Frep^n(0) 
     \upupsilon_2^{\omega} \Frep^n(0) 
      \geq 
      \upupsilon_2^{\omega}
 %\left( \frac{\kappa_2^2}{10\kappa_3^2}\right)^{\omega} 
 \kappa_2^{\omega} (d_n^+ (0))^{-\omega}
 \end{equation}
 or, there is some $\tau_n \in (0, \mathcal T^n_2)$ such that
 \begin{equation}
 \label{contradde}
 \dminep(\tau_n) < \frac{1}{2} \mathcal S_2  \dminep(s) 
 \end{equation}
 or
  \begin{equation}
 \label{contralouf}
 \mFatn\left (\tau_n\right)^{-\frac{1}{\omega}}
< 
  \mFatn( 0)^{-\frac{1}{\omega}}
  +\frac{1}{32\kappa_3}  ( d_n^+(s)).  \end{equation}
 As in \eqref{harry}, but with a different scaling $r_n$ we set 
\begin{equation}
\label{harry2}
r_n= \dmin^{\eps_n, +}(0)\geq n \eps_n, \,  \tilde v_n(x, t)= v_n (r_n x, r_n^2 t), \ \,   
{\rm and  \ } 
 \tilde{ \mathfrak v}_n(x, s)= \tilde v_n(x, \tilde{ \eps}_n^{-\omega} s).
 \end{equation}
We verify that  $\tilde v_n$ is a solution to $({\rm PGL})_{\tilde \eps_n}$ with 
 $\tilde \eps_n=(r_n)^{-1} \eps_n \to 0$ as $n \to \infty$  
 and that the points 
  $\tilde a_k^n (\tau)=r_n^{-1} a_k^n (r_n^{-(2+\omega)}\tau)$ for ${k \in J},$  are  the front points of  
  $\tilde{ \mathfrak v}_n$.  
 We distinguish three cases, which are complementing going if necessary to subsequences. \\
  
  \noindent
  {\it Case 1: \eqref{contrad2ter}  holds, for any $n \in \N$}.  
  It follows     $\mathcal{W P}_{\tilde\eps_n}^{L_n} (\upalpha_1 \tilde \eps_n, \tau)$ holds for every  
  $\tau \in (0, \tilde {\mathcal S}_1^n)$, where $\tilde {\mathcal S}_1^n=r_n^{-(2+\omega)}\mathcal S_1^n=\mathcal  K_1$.
Let $j$ be an arbitrary index in $J^+$. 
Translating if necessary the origin, we may  assume  that 
$a_{j}^n (0)=0  \hbox{ so that } a_{j+1}^n (0)\geq d^+_n(0)\geq n \eps_n $
 and hence
$\tilde a_{j+1}^n (0)-\tilde a_{j}^n (0)\geq 1. $
   Since
 $\displaystyle{\tilde \eps_n \to 0 {\rm \ as  \ } n \to \infty}$,  we may implement part of the already established asymptotic analysis for $({\rm PGL})_\eps$ on the sequence $(\tilde v_n)_{n\in \N}$.
 First, passing possibly to a subsequence, we may assume that for some subset $\tilde J \subset J(0)$ the points  $\{\tilde a_k (0)\}_{k \in \tilde J} $ converge to some  finite   limits $\{\tilde a_k^0\}_{k \in \tilde J}$, whereas the points with indices in $J(0)\setminus \tilde J$ diverge either to $+\infty$ or to $-\infty$.  We choose $\tilde \rL\geq 1$ so that \eqref{confetti} holds.
% $$\underset {k \in \tilde J}  \cup \{\tilde a_k^0\} \subset [-\frac{\tilde \rL}{2},\frac{ \tilde  \rL}{2}].
% $$
It follows from  Proposition \ref{firststep1} that for $\tau \in (0,\mathcal  K_1) $, 
   we have  
   $$\vert \tilde a_{j+1}^n (\tau)- \tilde a_j^n (\tau) \vert  \to \vert \tilde a_{j+1} (\tau)- \tilde a_j(\tau) \vert 
   \geq 
      \left( \mathcal S_1 \tau+ \mathcal S_2  \textswab d_{\tilde a }^+(0)^{\omega+2} \right)^{\frac{1}{\omega+2}}
   =\left( \mathcal S_1 \tau+ \mathcal S_2 \right)^{\frac{1}{\omega+2}}
   $$
   as    $n \to \infty$, where the last inequality is a consequence of Proposition \ref{getrude}.  Taking the infimum over 
   $J^+$, we  obtain, for $n$ sufficiently large
  \begin{equation}
  \label{danssonjus}
   {\tilde d}_n^+(\tau)=\underset{ j \in J^+}\inf\vert \tilde a_{j+1}^n (\tau)- \tilde a_j^n (\tau) \vert   \geq 
%\textswab d_{\tilde a}^+(\tau) \geq 
   %\left( \mathcal S_1 t+ \mathcal S_2  \textswab d_{\tilde a }^+(0)^{\omega+2} \right)^{\frac{1}{\omega+2}}
\frac 12 \left( \mathcal S_1 \tau + \mathcal S_2 \right)^{\frac{1}{\omega+2}}
\geq  \frac{1}{2} \left( \mathcal S_1 \tau \right)^{\frac{1}{\omega+2}}, \forall \tau \in (0, \mathcal  K_1),
 \end{equation}
On the other hand,  going back to \eqref{contrad2ter},  with the same notation as in Proposition \ref{nihil}, we are led to the inequality
  \begin{equation}
  \label{juteux}
\tilde d_n^+(\tau) \leq 
%  \underset{ j \in J^+} \inf \vert \tilde a_{j+1}^n (\tau)- \tilde a_j^n (\tau)
 %\vert
\kappa_3 \kappa_2^{-1}\upupsilon_2^{-1}   {\rm \  for \ }  \tau \in (0,\mathcal  K_1). 
\end{equation}
In view of our choice \eqref{choixk1}  of $\mathcal  K_1$, relations \eqref{danssonjus} and \eqref{juteux} are contradictory for $\tau$ close to $\mathcal  K_1$ yielding hence a contradiction in Case 1. \\

\noindent 
  {\it Case 2: \eqref{contrad2ter}  does not hold, but \eqref{contradde} holds, for any $n \in \N$}. The argument is almost identical, we conclude again thanks to \eqref{danssonjus} but keeping $\mathcal S_2$ instead of $\mathcal S_1 \tau$ in its last inequality.

\noindent
  {\it  Case 3: \eqref{contrad2ter} does not hold  but \eqref{contralouf} holds, for any $n\in \N$.} 
  %, hence  the stopping time $\mathcal T_2^n$ is defined for any $n$, with 
 %$\mathcal T_2^n \leq S_2^n.$  
 As in the proof of Proposition \ref{nihil}, we  conclude that
 $0<\tilde {\mathcal T}_2 \equiv \liminf_{n\to +\infty} \tilde {\mathcal T}_2^n.$
   This situation is slightly more delicate than the ones analyzed so far, and we have \emph{to track also the fronts escaping possibly at infinity}.  
   Up to a  subsequence, we may assume that the set $J$ is decomposed as a disjoint union of clusters  
  $\displaystyle{ J=\underset {i=1}{\overset{q} \cup} J_p }$
where each of the sets $J_p$ is an ordered set of  $m_p+1$ consecutive points, that is $J_p=\{k_p, k_p+1, \cdots k_p+m_p\}$   and such that the two following properties holds:  
%\noindent
\begin{itemize}
   \item There exists a constant $C>0$ independent of $n$ such that   
\begin{equation}
\label{cisternino1}
\vert  \tilde a_{k_p}^n(0)- \tilde a_{k_p+r}(0)\vert \leq   C
 {\rm \ for \  any  \ } p \in \{1, \cdots, q\} \ {\rm  and \  any \ }  r \in \{k_p, \cdots, m_p\}
 \end{equation}
 \item For $1 \leq p_1 <p_2\leq q $, we have 
% \begin{equation*}
  $ \tilde a^n_{k_{p_2}}- \tilde a^n_{k_{p_1}}   \to + \infty.$ 
  % \end{equation*}
%\end{enumerate}
\end{itemize}
%\smallskip
\noindent
   For a given $p \in \{1, \cdots, q\}$,  translating if necessary the origin, we may assume that 
  $\tilde a^n_{k_p} (0)=0,$ and passing possibly to a further subsequence,   that the front points at time $0$ converge  as $n \to +\infty$ to some limits  denoted $\tilde a_{p, k}(0)$, for $k\in \{k_p, \cdots, k_p+m_p\}.$  Notice that, as an effect of the scaling, all other front points diverge to infinity, in the chosen frame. 
 We apply  now Proposition \ref{firststep1} to this cluster of points : it yields  uniform convergence, for  $k\in \{k_p, \cdots, k_p+m_p\}$ of the front points 
  $\tilde a^n_{k} (\cdot)$ to the solution $\tilde a_{p, k}(\cdot)$ of the differential equation \eqref{tyrannosaure} supplemented with the initial time conditions $\tilde a_{p, k}(0)$ defined above.  If $\Fat^p$ denotes the functional $\Fat$ defined in \eqref{repattract} restricted to the points of the cluster $J_p$, we have in view of \eqref{gediroff}
    $$
    \frac{d}{d\tau} \Fat^p(\tau) \geq 0,\qquad  \text{for any } p=1, \cdots, q, \quad \text{for any} \ \tau \in (0, \tilde {\mathcal T}_2).
    $$
 On the other hand, since the mutual distances between the distinct clusters diverge towards infinity, and hence their mutual interactions energies tend to zero, one obtains, in view of the uniform convergence for each separate cluster, that
 $$  \underset {n \to +\infty } \lim \tmFatn(\tau) =
  \underset {p=1} {\overset  {q} \sum}   \Fat^p (\tau)
  \geq 
  \underset {p=1} {\overset  {q} \sum}  \Fat^p (0)=
 \underset {n \to +\infty } \lim \tmFatn(0),
 \qquad \text{for }  \tau  \in   \  (0, \tilde{\mathcal T}_2).
 $$
 Therefore,  for $n$ sufficiently large we are led to
 $$  \tmFatn(\tilde{\mathcal T}_2)^{-\frac{1}{\omega}}\leq  \tmFatn(0)^{-\frac{1}{\omega} }+\frac{1}{2\kappa_3}.   $$
 Scaling back to the original variables, this contradicts ¬†\eqref{contralouf} and hence completes the proof. 
\end{proof}

From Proposition \ref{nihil} and Proposition \ref{spillit} we obtain 
    
 \begin{proposition}
 \label{vinsanto}
There exists $\mathcal K_2>0,$ depending only on $V$ and $M_0$, with the following properties. 
Assume that $J^-(s_0) \neq \emptyset$ and that $s \in \mathcal I_0^\eps(s_0)$ satisfies 
 \begin{equation}\label{eq:bondebutbis}
 \dmineL(s) \geq \max(\beta_0,\beta_1)\eps, \qquad \text{and} \qquad s + \mathcal  K_2   \dmines(s)^{\omega+2}<S.
\end{equation}
Then there exists some time  $\mathcal T_{\rm  col}^-(s) \in \mathcal I_0^\eps(s_0)$ such that 
\begin{equation}
\label{minorrtcol}
  \rtcol(s)- s  \leq  \mathcal K_2   \dmines(s)^{\omega+2},
\end{equation}
and 
  \begin{equation}
  \label{defrtcol}
  \dmineL(\rtcol (s))\leq  \max\left( \beta_0, 8 \mathfrak q_1(\upalpha_2)\right) \eps.
 \end{equation}
\end{proposition}
\begin{proof}
We distinguish two cases.\\

\smallskip
\noindent{\it Case I:}
\begin{equation}
 \label{demine}
 \dmineL(s)= \dmines(s) \leq \dminep(s).
\end{equation}
In that case we will make use of Proposition \ref{nihil} in an iterative argument. 
In view of \eqref{eq:bondebutbis}, we are in position to invoke Proposition \ref{nihil} at time $\hat s=s$ and set 
$s_1=\mathcal {T}_1^\eps (s),$ 
so that in particular  
\begin{equation}
\label{red0}
s_1-s \leq \mathcal  K_0 \dmin^{\eps, -}(s)^{\omega+2}
%, \ \  \dmin^{\eps, -}(s_1) \leq \frac{1}{10} 
\qquad \text{and} \qquad \dmin^{\eps, -}(s_1) \leq  \dmin^{\eps, +}(s_1). 
\end{equation}
Notice that by \eqref{eq:bondebutbis} and \eqref{red0} we have $s_1 < S.$\\
We distinguish two sub-cases:

\smallskip
{\it Case I.1: $s_1 = \mathcal T_0^\eps (s_0)$ or $\dmin^{\eps,-}(s_1)<\beta_0 \eps$.} In that case, we simply set
$\rtcol (s)= s_1$ and we are done if we require $\mathcal K_2 \geq \mathcal K_2,$ by \eqref{red0} and the definition of $\mathcal T_0^\eps(s_0).$

\smallskip
{\it Case I.2: $s_1 < \mathcal T_0^\eps (s_0)$ and $\dmin^{\eps,-}(s_1)\geq \beta_0 \eps.$ }
In that case, we may apply Proposition \ref{nihil} at time $\hat s=s_1$ and set 
$s_2=\mathcal {T}_1^\eps (s_1),$  
so that in particular  
\begin{equation}
\label{red0bis}
s_2-s_1 \leq \mathcal  K_0 \dmin^{\eps, -}(s_1)^{\omega+2}
%, \ \  \dmin^{\eps, -}(s_1) \leq \frac{1}{10} 
\qquad \text{and} \qquad \dmin^{\eps, -}(s_2) \leq  \dmin^{\eps, +}(s_2). 
\end{equation}
Moreover, since in that case $s_1= \mathcal T_1^\eps(s) < \mathcal T_0^\eps (s_0)$, it follows from \eqref{majordbis} that
\begin{equation}\label{eq:decroit0}
 \dmin^{\eps,-}(s_1) \leq \frac{1}{10} \dmin^{\eps,-}(s), 
\end{equation}
and therefore from \eqref{red0bis} we actually have
\begin{equation}
\label{red0ter0}
s_2-s_1 \leq \mathcal  K_0 10^{-(\omega+2)} \dmin^{\eps, -}(s)^{\omega+2}
%, \ \  \dmin^{\eps, -}(s_1) \leq \frac{1}{10} 
\qquad \text{and} \qquad \dmin^{\eps, -}(s_2) \leq  \dmin^{\eps, +}(s_2). 
\end{equation}

\medskip
 We then iterate the process until we fall into Case I.1. If we have not reached that stage up to step $m$, then thanks to Proposition \ref{nihil} applied at time  $\hat s = s_{m}$ we obtain, with $s_{m+1}:=\mathcal T_1^\eps(s_m)$, 
\begin{equation}
\label{red0bis*}
s_{m+1}-s_m \leq \mathcal  K_0 \dmin^{\eps, -}(s_m)^{\omega+2}
%, \ \  \dmin^{\eps, -}(s_1) \leq \frac{1}{10} 
\qquad \text{and} \qquad \dmin^{\eps, -}(s_{m+1}) \leq  \dmin^{\eps, +}(s_{m+1}). 
\end{equation}
Moreover, since Case I.1 was not reached before step $m$, we have $s_{p}= \mathcal T_1^\eps(s_{p-1}) < \mathcal T_0^\eps (s_0)$ for all $p\leq m$, so that repeated use of \eqref{majordbis} yields 
\begin{equation}\label{eq:decroit}
 \dmin^{\eps,-}(s_p) \leq \left(\frac{1}{10}\right)^{p} \dmin^{\eps,-}(s),\qquad \forall p\leq m.
\end{equation}
From \eqref{red0bis*} we thus also have
\begin{equation}
\label{red0ter}
s_{p+1}-s_p \leq \mathcal  K_0 10^{-p(\omega+2)} \dmin^{\eps, -}(s)^{\omega+2},\qquad \forall p\leq m,
\end{equation}
and therefore by summation 
\begin{equation}
\label{redishm}
s_{m+1}-s \leq \mathcal  K_0 (\sum_{p=0}^m 
10^{-p(\omega+2)})\dmin^{\eps, -}(s)^{\omega+2},
%{\rm  \  and  \ } 
%\dmin^{\eps, -}(s_m) \leq \frac{1}{10^{m-1}} 
%\dmin^{\eps, -}(s). 
\end{equation}
so that in particular from \eqref{eq:bondebutbis} it holds $s_{m+1}<S$ if we choose $\mathcal K_2 \geq 2 \mathcal K_0.$
It follows from \eqref{eq:decroit} that Case I.1 is necessarily reached in a finite number of steps, thus defining
$\rtcol(s)$, and from \eqref{redishm} we obtain the upper bound
\begin{equation}\label{eq:presque1}
\rtcol(s)-s \leq \mathcal  K_0 (\sum_{p=0}^\infty 
10^{-p(\omega+2)})\dmin^{\eps, -}(s)^{\omega+2} \leq 2 K_0 \dmin^{\eps, -}(s)^{\omega+2} ,
\end{equation}
from which \eqref{minorrtcol} follows. 

\medskip
\noindent{\it Case II:}
\begin{equation}
 \label{demineenvers}
  \dmineL(s)= \dminep(s) < \dmines(s).
\end{equation}
Note that this implies that $J^+(s_0) \neq \emptyset.$  
We will show that Case II can be reduced to Case I after some controlled interval of time necessary for the repulsive forces
to push $\dminep$ above $\dmines.$ More precisely, we define the stopping time 
 $$\rtcros (s)=\inf \{ \mathcal T_0^\eps(s)\geq  s'\geq s, \,
  \dmines(s')\leq \dminep(s')\}.
 $$
As in Case I, we implement an iterative argument, but based this time on Proposition \ref{spillit}. 
In view of \eqref{demineenvers} and \eqref{eq:bondebutbis}, we may apply Proposition \ref{spillit}
at time $\hat s = s$ and set 
$s_1=\mathcal {T}_2^\eps (s),$ 
so that in particular  
\begin{equation}
\label{red0lissy}
s_1-s \leq \mathcal  K_1 \dmin^{\eps, +}(s)^{\omega+2} \leq \mathcal  K_1 \dmin^{\eps, -}(s)^{\omega+2}. 
%, \ \  \dmin^{\eps, -}(s_1) \leq \frac{1}{10} 
%\qquad \text{and} \qquad \dmin^{\eps, -}(s_1) \leq  \dmin^{\eps, +}(s_1). 
\end{equation}
Notice that by \eqref{eq:bondebutbis} and \eqref{red0lissy} we have $s_1 < S$ and therefore
$\dmin^{\eps,+}(s_1)\geq 10 \dmin^{\eps,+}(s) \geq \beta_1$, and by \eqref{loufdingo2}
 \begin{equation}
  \label{eq:presque0}
  \begin{aligned}
  \mFat\left (s_1\right)^{-\frac{1}{\omega}}
 &\leq
  \mFat( s)^{-\frac{1}{\omega}}
  +\frac{1}{\kappa_3} \dminep(s)\\
 &\leq
  \mFat( s)^{-\frac{1}{\omega}}
  +\frac{1}{10\kappa_3}\dminep(s_{1}).
\end{aligned}
\end{equation}
We distinguish two sub-cases. 

\smallskip
{\it Case II.1: $s_1 \geq \rtcros(s).$} In that case we proceed to Case I which we will apply
starting at $s_1$ instead of $s$ and we set $\rtcol(s):=\rtcol(s_1).$ Since, combining the first inequality of \eqref{eq:presque0}   
with \eqref{globe}, we deduce that
\begin{equation}\label{eq:momo0}
 \dmines(s_{1})\leq \kappa_3\kappa_2^{-1} \dmines(s)+ \dminep(s) \leq \left( \kappa_3\kappa_2^{-1}+1\right) \dmines(s),
\end{equation}
the equivalent of \eqref{eq:presque1} becomes
\begin{equation}\label{eq:presque2}
\begin{aligned}
\rtcol(s_1)-s_1 &\leq \mathcal  K_0 (\sum_{p=0}^\infty 
10^{-p(\omega+2)})\dmin^{\eps, -}(s_1)^{\omega+2}\\
&\leq 2 \mathcal K_0 \dmin^{\eps, -}(s_1)^{\omega+2}\\
&\leq  2 \mathcal K_0 \left( \kappa_3\kappa_2^{-1}+1\right)^{\omega+2} \dmin^{\eps, -}(s)^{\omega+2},
\end{aligned}
\end{equation}
and therefore it follows from \eqref{red0lissy} that
\begin{equation}\label{eq:momo1}
\rtcol(s)-s \leq \rtcol(s_1)-s_1 + (s_1-s) \leq  \left( \mathcal K_1  + 2 \mathcal K_0 \left( \kappa_3\kappa_2^{-1}+1\right)^{\omega+2}\right) \dmin^{\eps, -}(s)^{\omega+2},
\end{equation}
and \eqref{minorrtcol} follows if $\mathcal K_2 \geq \mathcal K_1  + 2 \mathcal K_0 \left( \kappa_3\kappa_2^{-1}+1\right)^{\omega+2}.$

\smallskip
{\it Case II.2: $s_1 < \rtcros(s).$}  In that case we proceed to construct $s_2=\mathcal T_2^\eps(s_1).$ Notice that
combining the second inequality of \eqref{eq:presque0}   
with \eqref{globe}, we deduce that
\begin{equation}\label{eq:pepe0}
\dmines(s_{1})\leq \kappa_3\kappa_2^{-1} \dmines(s)+ \frac{1}{5}\dminep(s_1)
 \leq 
 {\kappa_3}{\kappa_2}^{-1} \dmines(s)+ \frac{1}{5}\dmines(s_1), 
\end{equation}
so that 
 \begin{equation}
 \label{forage0}
\dminep(s_1)\leq \dmines(s_1)\leq \frac54 \kappa_3\kappa_2^{-1} \dmines(s).  
\end{equation}

\medskip
\noindent
We explain now the iterative argument. Assume that for some $m\geq 1$ have already constructed $s_1,\cdots,s_m$, such that 
for  $2\leq p\leq m$ 
$$
s_p < S,\qquad \beta_1\eps \leq \dminep(s_p)\leq \dmines(s_p),\qquad s_p = \mathcal T_2^\eps(s_{p-1}). 
$$

%We distinguish two sub-cases:

%In that case since $s_1 < S$ 
%we have $\dmin^{\eps,+}(s_1)\geq 10 \dmin^{\eps,+}(s) \geq \beta_1$ 
%and therefore we may apply Proposition \ref{nihil} at time $\hat s=s_1$ and set 
%$s_2=\mathcal {T}_2^\eps (s_1).$  
%so that in particular  
%\begin{equation}
%\label{red0bislissy}
%s_2-s_1 \leq \mathcal  K_1 \dmin^{\eps, +}(s_1)^{\omega+2}.
%, \ \  \dmin^{\eps, -}(s_1) \leq \frac{1}{10} 
%\qquad \text{and} \qquad \dmin^{\eps, -}(s_2) \leq  \dmin^{\eps, +}(s_2). 
%\end{equation}
%We explain now how to iterate the previous process, until we will finally fall into Case II.1. Assume 
%that we have not reached that stage up to step $m$, where $s_{m+1}$ is constructed. 
\noindent
First, repeated use of \eqref{majordplus} yields 
\begin{equation}
\label{pourbien}
\dminep(s_p)\geq 10^p \dminep(s), \qquad \forall 1\leq p\leq m, 
\end{equation}
and actually
\begin{equation}
\dminep(s_p)\geq 10^{p-q} \dminep(s), \qquad \forall 1\leq q \leq  p\leq m.
\end{equation}
Hence, by repeated use of 
\eqref{loufdingo2}, we obtain
 \begin{equation*}
 \begin{aligned}
\label{ubs}
  \mFat\left (s_m\right)^{-\frac{1}{\omega}}
 &\leq
  \mFat( s)^{-\frac{1}{\omega}}
  +\frac{1}{\kappa_3} \left(\dminep(s)+ \underset {p=1} {\overset{m-1}\sum}  \dminep(s_p) \right)\\
 &\leq
  \mFat( s)^{-\frac{1}{\omega}}
  +\frac{1}{\kappa_3} \sum_{p=0}^{m-1} 10^{-p}  \dminep(s_{m-1})\\
  &\leq   \mFat( s)^{-\frac{1}{\omega}}
  +\frac{2}{\kappa_3}\dminep(s_{m-1})\\
  &\leq  \mFat( s)^{-\frac{1}{\omega}}
  +\frac{1}{5\kappa_3}\dminep(s_{m}).
\end{aligned}
\end{equation*}
Combining the latter with \eqref{globe}, we deduce that
$$
 \dmines(s_{m})\leq \kappa_3\kappa_2^{-1} \dmines(s)+ \frac{1}{5}\dminep(s_m)
 \leq 
 {\kappa_3}{\kappa_2}^{-1} \dmines(s)+ \frac{1}{5}\dmines(s_m), 
$$
so that 
 \begin{equation}
 \label{forage}
\dminep(s_m)\leq \dmines(s_m)\leq \frac54 \kappa_3\kappa_2^{-1} \dmines(s).  
\end{equation}
Let $s_{m+1} := \mathcal T_2^\eps(s_m).$ Then by \eqref{loufdingo} and \eqref{pourbien}
\begin{equation}
 \label{sommetoute}
\begin{aligned}
 s_{m+1}-s 
& \leq
 \mathcal  K_1 \left( \dminep(s)^{\omega+2} + \sum_{p=1}^{m} \dminep(s_p)^{\omega+2} \right)\\
&\leq  
\mathcal  K_1 \sum_{p=0}^{m-1} 10^{-(\omega+2)(m-p)} \dminep(s_{m})^{\omega+2}\\ 
& \leq
2\mathcal  K_1  \dminep(s_{m})^{\omega+2}.
\end{aligned}
\end{equation}
Combining \eqref{sommetoute} with \eqref{forage} we are led to 
$$s_{m+1}-s \leq 2 \left(\frac{5\kappa_3}{4\kappa_2}\right)^{\omega+2}\mathcal  K_1 \dmines(s)^{\omega+2}
$$
and therefore by \eqref{eq:bondebutbis} we have $s_{m+1}<S.$

Combining \eqref{forage} with  \eqref{pourbien}, we obtain
 $$
0 \leq \dmines(s_m)-\dminep(s_m) \leq  \frac{\kappa_3}{\kappa_2} \dmines(s)-10^m(\dminep(s)),
$$
and therefore necessarily
$$
 m \leq \log_{10} \left( \frac{\kappa_3 \dmines(s)}{\kappa_2 \dminep(s)}\right).
$$
It follows that the number  $m_0=\sup \{m\in \N_*,\dmines(s_m)\geq \dminep(s_m) \} $ is finite, and  
at that stage we proceed to Case I as in Case II.1 above, and the conclusion follows likewise, replacing  
\eqref{eq:momo0} by 
$$
\dminep(s_{m_0+1}) \leq \frac94 \kappa_3\kappa_2^{-1} \dmines(s)
$$
which is obtained combining 
$$
\dminep(s_{m_0+1}) \leq \kappa_3\kappa_2^{-1}\dmines(s) + \dminep(s_{m_0}),
$$ 
with
\begin{equation*}
\dminep(s_{m_0})\leq \dmines(s_{m_0})\leq \frac54 \kappa_3\kappa_2^{-1} \dmines(s).
\end{equation*}
\end{proof}

%%%%%%%%%%%%%%%%%%%%%%%%%%%%%%%%%%%%%%%%%%%%%%%%%%%%%%%%%%%%%%%%%%
\subsection{Proof of Proposition \ref{rtcolplus0}}
\label{proofouf}

We will fix the value of the constants $\upkappa_1$, $\upalpha_3$ and $\mathcal K_{\rm col}$ in the course of the proof.\\
Let $s_0$ be as in the statement. We first require that
$$
\upalpha_3 \geq \upalpha_2 \qquad \text{and that}\qquad \upalpha_3 \geq 16 \mathfrak q_1(\upalpha_2), 
$$
so that assumption $\wp(\upalpha_3,s_0)$ implies assumption \ref{eq:bondebut} of Subsection \ref{sect:contrad}. 

Next, we set $s=s_0+\eps^{\omega+2}$ and we wish to make sure that the assumptions of Proposition \ref{vinsanto}
are satisfied at time $s.$ In view of the upper bound \eqref{theodebase1} on the velocity of the front set, we deduce that
$$
\dmineL(s) \geq \dmineL(s_0) - C\uprho_0^{\frac{1}{\omega+2}}\eps \geq \upalpha_3 \eps -  C\uprho_0^{\frac{1}{\omega+2}}\eps \geq \max(\beta_0,\beta_1) \eps
$$
provided we choose $\upalpha_3$ sufficiently large. Also,
\begin{equation}\label{eq:pot0}
\frac12 \dmines(s_0) \leq \dmineL(s_0) - C\uprho_0^{\frac{1}{\omega+2}}\eps \leq \dmineL(s) \leq \dmineL(s_0) + C\uprho_0^{\frac{1}{\omega+2}}\eps \leq 2 \dmineL(s_0),
\end{equation}
and therefore provided we choose
$$
\mathcal K_{\rm col} \geq 2^{\omega+2} \mathcal K_2
$$
it follows from the assumption $s_0+\mathcal K_{\rm col}\dmineL(s_0)^{\omega+2} <S$ that $s+\mathcal K_2 \dmineL(s)^{\omega+2} <S.$   
Therefore we may apply Proposition \ref{vinsanto}. Let $\rtcol(s) \in \mathcal I_0^\eps(s_0)$ be given by its statement, so that
by \eqref{eq:pot0} 
$$
\rtcol(s)-s \leq 2^{\omega+2} \mathcal K_{2} \dmines(s_0)^{\omega+2}, 
$$
and
\begin{equation}\label{eq:cavient}
\dmineL(\rtcol(s)) \leq \max\left( \beta_0, 8 \mathfrak q_1(\upalpha_2)\right) \eps.
\end{equation}
By Proposition \ref{cornichon}, there exists some time 
$\rtcolp(s_0) \in [\rtcol(s), \rtcol(s)+ \mathfrak q_0(\upalpha_3)\eps^{\omega+2}]$   such that  $\mathrm{W P I}_\eps^{\scriptscriptstyle \rL} (\upalpha_3 \eps, \rtcolp(s_0))$ holds.  In  view of \eqref{eq:pot0} and since $\dmines(s_0)\geq \upalpha_3\eps$, it follows
that 
\begin{equation}\label{eq:dominetemps}
\begin{aligned}
0\leq \rtcolp(s_0)-s_0 &\leq \eps^{\omega+2} +2^{\omega+2} \mathcal K_{2} \dmines(s_0)^{\omega+2} + \mathfrak q_0(\upalpha_3) \eps^{\omega+2} \\
& \leq \left( 2^{\omega+2}\mathcal K_2+\frac{1+\mathfrak q_0(\upalpha_3)}{\upalpha_3^{\omega+2}} \right)\dmines(s_0)^{\omega+2}\\
&\leq \mathcal K_{\rm col} \dmines(s_0)^{\omega+2} 
\end{aligned}
\end{equation}
provided we finally {\it fix} the value of $\mathcal K_{\rm col}$ as 
$$\mathcal  K_{\rm col} = \left( 2^{\omega+2}\mathcal K_2+\frac{1+\mathfrak q_0(\upalpha_3)}{\upalpha_3^{\omega+2}} \right).$$
[Note that at this stage $\mathcal K_{\rm col}$ is fixed but its definition depend on $\upalpha_3$ which has not yet been fixed.  
Of course when we will fix $\upalpha_3$ below we shall do it without any reference to $\mathcal K_{\rm col}$, in order to avoid 
impossible loops]\\ 
Next, we first claim that 
$$
\matfEL(s_0) \geq \matfEL(s) \geq \matfEL(\rtcolp(s_0)).
$$
In view of Corollary \ref{avecdec}, it suffices to check that $\rL \geq \rL_0(s_0,\rtcolp(s_0))$, where we recall that
the function $\rL_0(\cdot)$ was defined in \eqref{defrL_0}. In view of \eqref{eq:dominetemps}, this reduces to
$$
100{\rm C}_e\rL^{-(\omega+2)}\mathcal K_{\rm col} \dmines(s_0)^{\omega+2}  \leq \frac{\upmu_1 }{4}.
$$
Since by \eqref{eq:lafrite} we have $\dmines(s_0) \leq 2\upkappa_0(s_0)\rL,$ it suffices therefore that
$$
\upkappa_0(s_0) \leq \frac12 \left( \frac{\upmu_1}{400{\rm C_{\rm e} \mathcal K_{\rm col}}}\right)^\frac{1}{\omega+2}\equiv \upkappa_1,
$$
and we have now {\it fixed} the value of $\upkappa_1.$\\
Next, we claim that actually
$$
\matfEL(\rtcolp(s_0)) \leq \matfEL(s_0)- \upmu_1.
$$
Indeed, otherwise by Corollary \ref{avecdec} we would have $\matfEL (\rtcolp(s_0))= \matfEL(s_0)$, and therefore condition 
$\mathcal C(\upalpha_3 \eps, \rL, s_0, \rtcolp(s_0))$ of Subsection \ref{propafront} would hold. Invoking 
Proposition \ref{pneurose}, this would imply that condition ${\mathcal WP}^{\rL}(\Lambda_{\rm log}(\upalpha_3\eps),\tau)$ holds for $\tau \in (s_0+\eps^{\omega+2},\rtcolp(s_0))$, so that in particular
$$
\dmine(\rtcol(s_0))\geq \Lambda_{\rm log}(\upalpha_3\eps).
$$
It suffices thus to choose $\upalpha_3$ sufficiently big so that 
$$
\Lambda_{\rm log}(\upalpha_3\eps) > \max(\beta_0, 8\mathfrak q_1(\upalpha_2)))\eps,
$$
and the contradiction then follows from \eqref{eq:cavient}. \qed

%%%%%%%%%%%%%%%%%%%%%%%%%%%%%%%%%%%%%%%%%%%%%%%%%%%%%%%%%%%%%%%%%%%%%%%%%%%%%%%%%%%%%%%%%%%%%%%%%%%%%%%%%%%%%%%%%%%%%%%%%%%%
%%%%%%%%%%%%%%%%%%%%%%%%%%%%%%%%%%%%%%%%%%%%%%%%%%%%%%%%%%%%%%%%%%
\subsection{Proof of Proposition \ref{nettoyage}}
\label{prooof}
%%%%%%%%%%%%%%%%%%%%%%%%%%%%%%%%%%%%%%%%%%%%%%%%%%%%%%%%%%%%%%%%%%%

We will fix the values of $\upkappa_*$ and $\uprho_*$ in the course of the proofs, as the smallest
numbers which satisfy a finite number lower bound inequalities.

First, recall that it follows from \eqref{confitdoie} and \eqref{theodebase1} that if $0\leq s\leq \uprho_0(R-r)^{\omega+2}$ then 
\begin{equation}\label{eq:grimb}
\mathfrak D_\eps(s) \cap \ILLLL \subset  \underset{ k\in J_0} \cup  (b^\eps_k-R,
b^\eps_k+R) \subset \IdkzL, \ \   {\rm where \ the \ union \ is \ disjoint},
\end{equation}
and in particular 
$\mathcal{C}_{\rL,S}$ holds where
$$
S := \uprho_0(R-r)^{\omega+2 } \geq \uprho_0 \left( \frac{R}{2}\right)^{\omega+2}.
$$
Having \eqref{defrL_0} in mind, and in view of \eqref{eq:grimb} and \eqref{poolish}, we estimate 
$$
100{\rm C_{\rm e}} \rL^{-(\omega+2)}S \leq  
100{\rm C_{\rm e}} \left(\frac{R}{2\rL}\right)^{\omega+2}S \leq
100{\rm C_{\rm e}} \upalpha_*^{-(\omega+2)} \leq \frac{\upmu_1}{4},
$$
where the last inequality follows provided we choose $\upalpha_*$ sufficiently large. As a consequence,
the function $\matfEL$ is non-increasing on the set of times $s$ in the interval
$[\eps^\omega\rL^2,S]$ where $\wp(\upalpha_1\eps,s)$ holds.

For such times $s$, the front points $\{a^\eps_k(s)\}_{k
\in J(s)}$ are well-defined, and for $q \in J_0$,  we have defined in the introduction $J_q(s)=\{k \in J(s),
a^\eps_k (s) \in [b^\eps_q-R, b^\eps_q+R]\}$, and we have set  $\ell_q=\sharp  J_q $ and 
$J_q(s)=\{k_{q} , k_{q+1},\cdots,  k_{q +\ell_q-1}\}, $ 
 where $k_1=1$, and  $k_q=\ell_1+\cdots +\ell_{q-1}+1$, for $q \geq 2$. 

\medskip
\noindent{\it Step 1. Annihilations of all the pairs of fronts-antifronts.} We claim that 
there exists some time $\tilde s \in (\eps^\omega\rL^2,\frac 12 S)$ such that $\wp(\udl,\tilde s)$ holds and such that  for any $q \in J_0$,   
 $\epsilon_{k+\frac 12}(\tilde s)=+1$, {\rm \ for   \  } $k \in J_q(\tilde s)\setminus 
\{k_{q}(\tilde s)+\ell_q(\tilde s)-1\} \ {\rm \ or  \  }   \sharp J_q (\tilde s) \leq 1, $ or equivalently that  $\dagger _k(\tilde s)=\dagger_{k'}(\tilde s)$ for $k$ and $k'$ in  the same $J_q(\tilde s)$. In particular, $\dmines(\tilde s) \geq 2R.$

\begin{proof}[Proof of the claim.]
If we require $\upalpha_*$ to be sufficiently large, then by \eqref{poolish} we have that $\eps^\omega\rL^2+\eps^{\omega+1}\rL \leq S/2,$
and therefore by Proposition \ref{parareglo} we may choose a first time 
$$
s_0 \in [\eps^\omega\rL^2,\eps^\omega\rL^2+\eps^{\omega+1}\rL] \quad\text{such that}\quad \wp(\udl,s_0) \ \text{holds}.
$$
Actually, we have
\begin{equation}\label{eq:bornes0}
s_0\leq \eps^\omega\rL^2+\eps^{\omega+1}\rL\leq 2 \eps^\omega\rL^2 \leq 2 \alpha_*^{-(\omega+2)}r^{\omega+2} \leq 2\alpha_*^{-2(\omega+2)} R^{\omega+2}  \leq \frac{1}{\uprho_0}2^{\omega+3} \alpha_*^{-2(\omega+2)}S.
\end{equation}
Note that by \eqref{theodebase1} we have the inclusion
$$
\mathfrak D_\eps(s_0) \cap \IL \subset \mathcal N(b,r_0),
$$
where
$$
r_0 = r + \left(\frac{s_0}{\uprho_0}\right)^\frac{1}{\omega+2} \leq 2r,
$$
provided once more that $\upalpha_*$ is sufficiently large, and where for $\rho>0$ we have set 
$
\mathcal N(b,\rho) = \cup_{q\in J_0} [b_j^\eps-\rho,b_j^\eps+\rho].
$
In view of   the confinement  condition \eqref{confitdoie} only two cases can occur:
\begin{equation*}
{\rm i) \ }  \quad \dmines(s_0)\geq  3R-2r_0     \qquad\text{or}\qquad
 {\rm   ii )\ }  \quad \dmines(s_0) \leq 2r_0. 
\end{equation*}

If case i) occurs, then, for any $q \in J_0$, we have
 $\epsilon_{k+\frac 12}=+1$, {\rm \ for  \ any \  } $k  \in  J_q ( \tau_1 )\setminus  \{k_{q}(s_0)+\ell_q(s_0)-1\} $. Choosing $\tilde s=s_0$,  Step 1  is completed in the case considered.

 If  instead case ii) occurs, then we will make use of Proposition \ref{rtcolplus0} to remove the small pairs of fronts-antifronts  present at small scales. More precisely, assume that for some $j\geq 0$ we have constructed $0\leq s_j\leq S$ and $r_j>0$ such that $\wp(\udl,s_j)$ holds, such that we have  
 \begin{equation}\label{eq:capousse}
\mathfrak D_\eps(s_j) \cap \IL \subset \mathcal N(b,r_j), \qquad  \matfEL(s_j)\leq \matfEL(s_0)-j\upmu_1,
 \end{equation}
 as well as the estimates,
 \begin{equation}\label{eq:bornesr}
  r_0\leq r_j \leq \gamma^j r_0\leq \frac{R}{2}, \qquad s_j \leq s_0 +
  (2^{\omega+2}
  \mathcal K_{\rm col}+1)\frac{\gamma^{j(\omega+2)}-1}{\gamma^{\omega+2}-1}r_0^{\omega+2}\leq \frac{S}{2},
 \end{equation}
 where $\gamma:=\left( 2 + 2\left(\frac{\mathcal K_{\rm col}}{\uprho_0}\right)^\frac{1}{\omega+2}\right)$, 
and moreover that case ii) holds at step $j$, that is
\begin{equation}\label{eq:crepe}
 \dmines(s_j) \leq 2r_j\leq R.
\end{equation}
Let $\tilde s_{j}:= \rtcolp(s_j)$ be given by Proposition \ref{rtcolplus0} (the confinement condition holds in view of \eqref{eq:grimb} and we have $\udl \geq \upalpha_3 \eps$ provided $\upalpha_*$ is sufficiently large), and let then  $s_{j+1}\in [\tilde s_j,\tilde s_j + \eps^{\omega+1}\rL]$ satisfying $\wp(\udl,s_{j+1})$ be given by Proposition \ref{parareglo}. In particular, we have
\begin{equation}\label{eq:sale}
 \matfEL(s_{j+1}) \leq \matfEL(\tilde s_j) \leq \matfEL(s_j)- \upmu_1 \leq \matfEL(s_0)-(j+1)\upmu_1.
\end{equation}
Since
$$
s_{j+1} - s_j \leq \mathcal K_{\rm col} (2r_j)^{\omega+2} + \eps^{\omega+1}\rL \leq \left(2^{\omega+2}\mathcal K_{\rm col}+1\right)r_j^{\omega+2},
$$
we have, in view of \eqref{eq:bornesr} 
\begin{equation}\label{eq:moinsdeux}
\begin{split}
s_{j+1} &\leq s_0 +  \left(2^{\omega+2}\mathcal K_{\rm col}+1\right) \left[ \frac{\gamma^{j(\omega+2)}-1}{\gamma^{\omega+2}-1} + 
\gamma^{j(\omega+2)} \right] r_0^{\omega+2}\\ 
&\leq s_0 +   \left(2^{\omega+2}\mathcal K_{\rm col}+1\right)
\frac{\gamma^{(j+1)(\omega+2)}-1}{\gamma^{\omega+2}-1} r_0^{\omega+2},
\end{split}
\end{equation}
and by \eqref{theodebase1} $\mathfrak D_\eps(s_{j+1}) \cap \IL \subset \mathcal N(b,r_{j+1})$, where 
\begin{equation}\label{eq:moinsune}
r_{j+1} = r_j + 2\left(\frac{\mathcal K_{\rm col}}{\uprho_0}\right)^\frac{1}{\omega+2} r_j + \frac{1}{\uprho_0}\eps \left( \frac{\rL}{\eps}\right)^\frac{1}{\omega+2} \leq \left( 2 + 2\left(\frac{\mathcal K_{\rm col}}{\uprho_0}\right)^\frac{1}{\omega+2}\right) r_j=\gamma r_j.
\end{equation}
In view of \eqref{eq:bornes0} and \eqref{poolish}, we also have
\begin{equation}\label{eq:nottoolong}
\gamma^{j+1} r_0 \leq 2 \gamma^{j+1} \upalpha_*^{-1} R 
\end{equation}
and
\begin{equation}\label{eq:nottoolongbis}\begin{split}
&s_{0} + \left(2^{\omega+2}\mathcal K_{\rm col}+1\right)
\frac{\gamma^{(j+1)(\omega+2)}-1}{\gamma^{\omega+2}-1} r_0^{\omega+2}\\  \leq  & \left[
\frac{2^{\omega+3}}{\uprho_0} \alpha_*^{-(\omega+2)} +  \frac{2^{2\omega+4}}{\uprho_0} \left(2^{\omega+2}\mathcal K_{\rm col}+1\right)
\frac{\gamma^{(j+1)(\omega+2)}-1}{\gamma^{\omega+2}-1} \right] \alpha_*^{-(\omega+2)}S.
\end{split}\end{equation}
It follows from \eqref{eq:moinsdeux}, \eqref{eq:moinsune}, \eqref{eq:nottoolong} and \eqref{eq:nottoolongbis} that if 
$\upalpha_*$ is sufficiently large (depending only on $M_0$, $V$ \underline{and} $j$), then \eqref{eq:bornesr} holds 
also for $s_{j+1}.$  As above we distinguish two cases :
\begin{equation*}
{\rm i) \ }  \quad \dmines(s_{j+1})\geq  3R-2r_{j+1}     \qquad\text{or}\qquad
{\rm   ii )\ }  \quad \dmines(s_{j+1}) \leq 2r_{j+1}. 
\end{equation*}

If case i) holds then by \eqref{eq:bornesr} we have $\dmines(s_{j+1}) \geq 2R$, we set $\tilde s =s_{j+1}$ which therefore satisfies the requirements of the claim, and we proceed to Step 2.

If case ii) occur then we proceed to construct $s_{j+2}$ as above. The key fact in this recurrence construction  is the second
inequality in \eqref{eq:capousse}, which, since $\matfEL(s_{j}) \geq 0$ independently of $j$, implies that {\it the process
as to reach case i) in a number of steps less than or equal to} $M_0/\upmu_1.$ In particular, choosing the constant 
$\upalpha_*$ sufficiently big so that the right-hand side of \eqref{eq:nottoolong} is smaller than $R/2$ for all $0\leq j \leq M_0/\upmu_1$ and so that the right hand side of \eqref{eq:nottoolongbis} is smaller than $S/2$ for all $0\leq j \leq M_0/\upmu_1$ ensures that the construction was licit and that the process necessarily reaches case i) before it could reach $j=M_0/\upmu_1+1,$ so defining $\tilde s$ as above.  
\end{proof}

\noindent
{\it Step 2: Divergence of the remaining repulsing fronts at small scale.} 
At this stage we have constructing $\tilde s \in [\eps^\omega\rL^2,\frac12S]$ which satisfies the requirements of
the claim in Step 1.  Moreover, note that in view of \eqref{eq:bornes0} and \eqref{eq:bornesr} we have the upper bound 
\begin{equation}\label{eq:hozbad}
\tilde s  \leq  \left(2 \alpha_*^{-(\omega+2)} +  2^{\omega+2}(2^{\omega+2}
  \mathcal K_{\rm col}+1)\frac{\gamma^{\frac{M_0}{\upmu_1}(\omega+2)}-1}{\gamma^{\omega+2}-1}\right)r^{\omega+2}.
\end{equation}

In order to complete the proof, we next distinguish two cases:
$$
 i)\quad \sharp J_q (\tilde s)\leq 1, \ \text{ for any }q \in J_0.\qquad 
 ii)\quad \sharp J_{q_0} (\tilde s)>1, \ \text{ form some }q_0 \in J_0.
$$

If case i) holds, then we actually have 
\begin{equation}\label{eq:vegbad}
\dmineL(\tilde s) \geq 2R.
\end{equation}
Since $2R \geq 16 \mathfrak q_1(\udl)\eps$ when $\upalpha_*$ is sufficiently large, it follows from Corollary \ref{parareglo3} 
that $\wp(\udll,s)$ holds for any    
  $\tilde s + \eps^{2+\omega}  \leq  s \leq \mathcal T_0^\eps(\udl, \tilde s)$, where
$$
\mathcal T_0^\eps(\udl, \tilde s) = \max \left\{ \tilde s+\eps^{2+\omega} \leq s \leq S \quad \text{s.t.}\quad \dmineL(s')\geq 8\q_1(\udl)\eps \quad \forall s' \in [\tilde s+\eps^{\omega+2},s]\right\}.
$$
In particular, $\wp(\udll,s)$ holds for any $s$ in
$\tilde s + \eps^{2+\omega}  \leq  s \leq \mathcal T_3^\eps(\udl, \tilde s)$, where
$$
\mathcal T_3^\eps(\udl, \tilde s) = \max \left\{ \tilde s+\eps^{2+\omega} \leq s \leq S \quad \text{s.t.}\quad \dmineL(s')\geq R \quad \forall s' \in [\tilde s+\eps^{\omega+2},s]\right\}.
$$
In view of \eqref{eq:vegbad} and estimate \eqref{theodebase1}, we obtain the lower bound
\begin{equation}\label{eq:polygbad}
\mathcal T_3^\eps(\udl,\tilde s) \geq \tilde s + \uprho_0 R^{\omega+2}.
\end{equation}
Note that \eqref{eq:polygbad} and \eqref{poolish} also yield
\begin{equation}\label{eq:polygbadbis}
\mathcal T_3^\eps(\udl,\tilde s) \geq \tilde s + \uprho_0 \upalpha_*^{-1}r^{\omega+2}\geq \uprho_0 \upalpha_*^{-1}r^{\omega+2}.
\end{equation}
Combining \eqref{eq:hozbad} and \eqref{eq:polygbadbis} we deduce in particular that 
$$
\wp(\upalpha_1\eps,s_r)\text{ holds} \qquad \text{and} \qquad \dmineL(s_r)\geq R \geq r,
$$
which is the claim of Proposition \ref{nettoyage}, provided 
\begin{equation}
\rho_* \geq \Big(3 +  2^{\omega+2}(2^{\omega+2}
  \mathcal K_{\rm col}+1)\frac{\gamma^{\frac{M_0}{\upmu_1}(\omega+2)}-1}{\gamma^{\omega+2}-1}\Big) \qquad \text{and} \qquad \rho_* \leq   \uprho_0 \upalpha_*^{-1}.
\end{equation}

It remains to consider the situation where case ii) holds. In that case, we have
$$
\dmines(\tilde s) \geq 2R \qquad \text{and} \qquad \dminep(\tilde s) \leq 2 \gamma^{M_0/\upmu_1}r \leq R,
$$
so that we are in a situation suited for Proposition \ref{spillit}. We may actually apply Proposition \ref{spillit} recursively
with $s_0:=\tilde s$ and $\hat s\equiv \hat s_k = (\mathcal T_2^\eps)^k(\hat s_0)$ where $\hat s_0 = \tilde s + \eps^{\omega+2}$, 
as long as $\dminep(\hat s_k)$ remains sufficiently small with respect to $R$ (say e.g. $\dminep(\hat s_k) \leq \alpha_*^{-\frac12}R$ provided 
$\upalpha_*$ is chosen sufficiently large), the details are completely  similar to the ones in Case II of Proposition \ref{vinsanto} and are 
therefore not repeated here. If we denote by ${k_0}$ the first index for which $\dminep(\hat s_{k_0})$ becomes larger than $\frac{2}{\mathcal 
S_2} r$ (in view of \eqref{rakotoson}) and $k_1$ the last index before $\dminep$ reaches $\alpha_*^{-\frac12}R,$ then we have 
$$\hat s_{k_0} \leq C r^{\omega+2}\qquad  \text{and}  \qquad \hat s_{k_1} \geq \frac{1}{C}\alpha_*^{-(\omega+2)/2}R^{\omega+2} \geq \frac{1}{C} \alpha_*^{(\omega+2)/2}r^{\omega+2},
$$
for some constant $C>0$ depending only on $M_0$ and $V$, 
and the conclusion that $\wp(\alpha_1\eps,s_r)$ holds follows as in case i) above, choosing first $\uprho_*$ sufficiently large (independently of $\upalpha_*$) and then $\alpha_*$ sufficiently large (given $\uprho_*$).\qed

 %%%%%%%%%%%%%%%%%%%%%%%%%%%%%%%%%%%%%%%%%%%%%%%%%%%%%%%%
 %%%%%%%%%%%%%%%%%%%%%
 \section{Proofs of Theorem \ref{maintheo1}, \ref{alignedsolution} and \ref{colissimo}}
 \label{extended}
 %%%%%%%%%%%%%%%%%%%%%%%%%%%%%%%%%%%%%%%%%%%%%%%%%%%%%%%%%%%%%%%%%%

%%%%%%%%%%%%%%%%%%%%%%%%%%%%%%%%%%%%%%%%%%%%%%%%%%%%%%%%%%%%%%%%%%
\subsection {Proof of Theorem \ref{maintheo1}}
Theorem \ref{maintheo1} being essentially a special case of Theorem \ref{alignedsolution},   we go directly to  the proof of Theorem \ref{alignedsolution}. Notice however that in Theorem \ref{maintheo1} the solution to the limiting system is unique, so that the result is not constrained by the need to pass to a subsequence. 
 
 %%%%%%%%%%%%%%%%%%%%%%%%%%%%%%%%%%%%%%%%%%%%%%%%%%%%%%%%%%%
 \subsection {Proof of Theorem \ref{alignedsolution}}
We fix $S<S_{\rm max}$ and let $\rL\geq \upkappa_*^{-1}\rL_0$, where $\rL_0$ is defined in the statement of Proposition \ref{firststep1} and 
$\upkappa_*$ in the statement of Proposition \ref{nettoyage}. We set $R=\frac 12 \min \{ a_{k+1}^0-a_{k}^0 , k = 1,\cdots,\ell_0-1\}$ and consider an arbitrary  $0<r<R/\upalpha_*.$
Since $({\rm H}_{1})$ holds, there exists some constant 
$\eps_{r}>0$ such that, if $0<\eps\leq \eps_{r}$, then \eqref{confitdoie} holds with $b_k\equiv a_k^0$ for any  $k \in \{1,\cdots,\ell_0\}.$ We are  therefore in  position to make use of Proposition \ref{nettoyage} and assert  
that for all such $\eps$ condition $\wp(\upalpha_1\eps,s_{r})$ holds as well as \eqref{clearance} and \eqref{clarence}.
It follows in particular from \eqref{clearance} and \eqref{clarence} that for every $k\in 1,\cdots,\ell_0$ we have 
$\sharp J_k(s_{r}) = |m_k^0|$, where $m_k^0$ is defined in \eqref{multiplicity}, and therefore $\sharp J(s_r)=\sum_{k=1}^{\ell_0} |m_k^0| \equiv \ell_1,$ in other words  the number of fronts as well as their properties do not depend on $\eps$ nor on $r.$  

We construct next the limiting splitting solution to the ordinary differential equation and  the corresponding subsequence proceeding  backwards in time  and using a diagonal argument. For that purpose, we introduce an arbitrary decreasing  sequence  $\{r_m\}_{m \in \N_*}$ such that 
$0<r_1 \leq R/\upalpha_*$, and such that 
$r_m \to 0$ as $m \to + \infty$.  For instance, we may take $r_m=\frac{1}{m} R/\upalpha_*$, and we set $s_m=s_{r_m}.$
Taking first  $m=1$, we find a subsequence $\{\eps_{n, 1}\}_{n \in \N_*}$ such that $\eps_{n, 1} \to 0$ as $n \to \infty$, and  such that all points $\{a_k^{\eps_{n, 1}} (s_1)\}_{k \in J}$  converge to  some limits $\{a_k^1 (s_1)\}_{k \in J}$ as $n \to + \infty$.  It follows from  \eqref{clarence}, passing to the limit $n \to +\infty,$ that 
\begin{equation}\label{eq:phoenix}
 \dminstar(s_1) \geq r_1.
\end{equation}
We are therefore in position to apply the convergence result of Proposition \ref{firststep1}, which yields in particular that
\begin{equation}
\label{trajectoryset2}
\mathfrak D_{\eps_{n,1}}(s) \cap \ILLLL 
\longrightarrow \cup_{k=1}^{\ell_1} \{a_k^1(s)\} \qquad \forall s_1<s<S_{\rm max}^1,
\end{equation} 
as $n\to +\infty,$ where  $\{a_k^1(\cdot)\}_{k \in J}$ is the unique solution of \eqref{tyrannosaure} with initial data $\{a_k^1 (s_1)\}_{k \in J}$ on its maximal time of existence $(s_1, S_{\rm max}^1).$  

 Taking  next $m=2$, we may extract a subsequence $\{\eps_{n, 2}\}_{n \in \N_*}$ from the sequence $\{\eps_{n, 1}\}_{n \in \N_*}$  
 such that all the points 
 $\{a_k^{\eps_{n, 2}} (s_{2})\}_{k \in J}$  converge to  some limits $\{a_k^2 (s_{2})\}_{k \in J}$ as $n \to + \infty$. Arguing as above, we may assert that 
 \begin{equation}
\label{trajectoryset3}
\mathfrak D_{\eps_{n,2}}(s) \cap \ILLLL 
\longrightarrow \cup_{k=1}^{\ell_1} \{a_k^2(s)\} \qquad \forall s_2<s<S_{\rm max}^2,
\end{equation} 
as $n\to +\infty,$ where  $\{a_k^2(\cdot)\}_{k \in J}$ is the unique solution of \eqref{tyrannosaure} with initial data $\{a_k^2 (s_2)\}_{k \in J}$ on its maximal time of existence $(s_2, S_{\rm max}^2).$  It follows from \eqref{clearance}, namely that only
repulsive forces are present at scale smaller than $R$, that $S_{\rm max}^2 \geq s_1.$ Therefore, since we have extracted a subsequence, it follows from \eqref{trajectoryset2} and \eqref{trajectoryset3} that $a_k^2(s_1)=a_k^1(s_1)$ for all $k\in J$, and therefore also that $S_{\rm max}^2= S_{\rm max}^1 \equiv S_{\rm max}$ and $a_k^2(\cdot)=a_k^1(\cdot)=a_k(\cdot)$ on $(s_2,S_{\rm max}).$   

We proceed similarly at each step $m\in \N_*$, extracting a subsequence  $\{\eps_{n, m}\}_{n \in \N_*}$ from the sequence $\{\eps_{n, m-1}\}_{n \in \N^*}$ such that all the points  $\{a_k^{\eps_{n, m}} (s_{m})\}_{k \in J}. $  Finally setting, for $n \in N_*$,  $\eps_n=\eps_{n, n}$, we obtain that
 \begin{equation*}
\mathfrak D_{\eps_{n}}(s) \cap \ILLLL 
\longrightarrow \cup_{k=1}^{\ell_1} \{a_k(s)\} \qquad \forall 0<s<S_{\rm max}^m,
\end{equation*} 
where  $\{a_k(\cdot)\}_{k \in J}$ is a splitting solution of \eqref{tyrannosaure} with initial data $\{a_k^0\}_{k \in J_0}$, on its maximal time of existence $(0, S_{\rm max}).$  Since $\rL \geq \rL_0$ was arbitrary, it follows that \eqref{trajectoryset} holds.  

It remains to prove that \eqref{unifbis}. This is actually a direct consequence of \eqref{trajectoryset}, the continuity of the trajectories $a_k(\cdot)$ and regularizing effect off the front set stated in Proposition \ref{estimpar} (e.g. \eqref{crottin2} for the $L^\infty$ norm). As a matter of fact, it is standard to deduce from this the fact that  
the convergence towards the equilibria $\upsigma_q$, locally outside the trajectory set, holds in any $\mathcal{C}^m$ norm, since the potential $V$ was assumed to be smooth.    
\qed

%%%%%%%%%%%%%%%%%%%%%%%%%%%%%%%%%%%%%%%%%%%%%%%%%%%%%%%%%%%%%%%%%%
\subsection {Proof of Theorem \ref{colissimo}}
%%%%%%%%%%%%%%%%%%%%%%%%%%%%%%%%%%%%%%%%%%%%%%%%%%%%%%%%%%%%%%%%%%   
   
As underlined in the introduction, Theorem \ref{colissimo} follows rather directly from Theorem \ref{alignedsolution}
and more importantly its consequence Corollary \ref{alarrivee} (whose proof is elementary and explained after Proposition \ref{getrude}), combined with Theorem \ref{mainbs2} and Proposition \ref{estimpar}.

Let thus $\rL> \rL_0$ and $\delta>0$ be arbitrarily given, we shall prove that, at least for $\eps\equiv \eps_n$ sufficiently small, 
\begin{equation}\label{okfrontset}
\mathcal{D}_{\eps}(S_{\rm max}) \cap  \IL \subset \cup_{j \in \{1,\cdots,\ell_2\}} [b_j-\delta,b_j+\delta],
\end{equation}
and 
\begin{equation}\label{okL1}
| \vm(x,S_{\rm max}) - \upsigma_{\hat\imath(j+\frac12)} | \leq C(\delta,\rL) \eps^\frac{1}{\theta-1},
\end{equation}
for all $j \in \{0,\cdots,\ell_2\}$ and for all $x \in (b_j+\delta,b_{j+1}-\delta),$ where we have used the convention that $b_0=-\rL$ and $b_{\ell_2+1}=\rL.$ 
Since $\rL$ can be arbitrarily big and $\delta$ arbitrarily small, this will imply that assumption $(\rm H_1)$ is verified
at times $S_{\rm max}$, which is the claim of Theorem \ref{colissimo}. 

Concerning \eqref{okfrontset}, by Corollary \ref{alarrivee} there exists  
\begin{equation}\label{eq:erable}
s^- \in \big[S_{\rm max} - \uprho_0\left(\frac{\delta}{4}\right)^{\omega+2} ,S_{\rm max}\Big).
\end{equation}
such that 
$$
\cup_{k \in \{1,\cdots,\ell_1\}} \{a_k(s^-)\} \subset \cup_{j \in \{1,\cdots,\ell_2\}} [b_j-\frac14\delta,b_j+\frac14 \delta]. 
$$   
The latter and Theorem \ref{alignedsolution} imply that, for $\eps$ sufficiently small,
\begin{equation}\label{eq:platane}
\mathcal D_\eps(s^-) \cap \ILL \cap \Big( \cup_{j \in \{1,\cdots,\ell_2\}} [b_j - \frac12 \delta,b_j+\frac12 \delta]\Big)^c = \emptyset. 
\end{equation}
In turn, Theorem \ref{mainbs2} (inclusion \eqref{theodebase1}) and \eqref{eq:platane}, combined with the upper bound \eqref{eq:erable} on $S_{\rm max}-s^-$,  imply that
\begin{equation}\label{eq:orme}
\mathcal D_\eps(s) \cap \ItdL \cap \Big( \cup_{j \in \{1,\cdots,\ell_2\}} [b_j - \frac34 \delta,b_j+\frac34 \delta]\Big)^c = \emptyset, \qquad \forall s \in [s^-,S_{\rm max}].
\end{equation}
For $s=S_{\rm max}$ this is stronger than \eqref{okfrontset}.

We proceed to \eqref{okL1}. In view of \eqref{eq:orme}, for any $x_0 \in \IL \setminus \big( \cup_{j\in \{1,\cdots,\ell_2\}}[b_j-\delta,b_j+\delta]\big)$ we have, for $\eps$ sufficiently small,
$$
\vm(y,s) \in B(\upsigma_i,\mu_0) \qquad \forall (y,s) \in [x_0-\frac18 \delta, x_0+\frac18 \delta] \times [s^-,S_{\rm max}].
$$
The latter is nothing but \eqref{offside} for $r=\frac18 \delta$, $s_0=s^-$ and $S=S_{\rm max}$, and therefore the conclusion
\eqref{okL1} follows from \eqref{crottin2} of Proposition \ref{estimpar}, with $C(\delta,\rL)= \frac15 {\rm C_e}(8/\delta)^\frac{1}{\theta-1}$ as soon as $\eps^\omega/(S_{\rm max}-s^-) \leq \delta^2/64.$ 
\qed
   
%%%%%%%%%%%%%%%%%%%%%%%%%%%%%%%%%%%%%%%%%%%%%%%%%%%%%%%%%%%%%%%
%%%%%%%%%%%%%%%%%%%%%%%%%%%%%%%%%%%%%%%%%%%%%%%%%%%%%%%%%%%%%%%%%%%%%%%%%%%%%%%%%
\setcounter{equation}{0} \setcounter{subsection}{0}
\setcounter{lemma}{0}\setcounter{proposition}{0}\setcounter{theorem}{0}
\setcounter{corollary}{0} \setcounter{remark}{0}
\renewcommand{\theequation}{A.\arabic{equation}}
\renewcommand{\thesubsection}{A.\arabic{subsection}}
\renewcommand{\thelemma}{A.\arabic{lem}}
\section*{Appendix A}
\renewcommand{\thesection}{A}
\numberwithin{equation}{section} \numberwithin{theorem}{section}
\numberwithin{lemma}{section} \numberwithin{proposition}{section}
\numberwithin{remark}{section} \numberwithin{corollary}{section}
%%%%%%%%%%%%%%%%%%%%%%%%%%%%%%%%%%%%%%%%%%%%%%%%%%%%%%%%%%%%%%%%%%%%%%%%%%%%%%%%%
In this Appendix we establish properties concerning the stationary solutions $\Ump$, $\Umps$,  $\Umpep$, etc, which we have used in the course of the previous discussion, mainly in Section \ref{raffinage}. 

\subsection{The operator $\mathcal  L_\mu$}

Consider  for $\mu>0$ the nonlinear operator $\mathcal L_\upmu$, defined for a smooth functions $U$ on $\R$ by 
$$\mathcal  L_\upmu (U) = -\frac{d^2}{dx^2} U + 2\upmu \theta  \, U^{2\theta-1}, $$
and set for simplicity $\mathcal L\equiv \mathcal L_1$. Most results  in this section are  deduced from the comparison principle:  if $u_1$ and $u_2$ are two functions defined on some non empty interval $I$, such that 
\begin{equation}
\label{compare}
  \mathcal L_{\upmu} (u_1) \geq 0,  \  \mathcal L_{\upmu} (u_2) \leq 0,\ {\rm and  } \  u_1\geq u_2 \ {\rm on}  \  \partial I, 
 \end{equation}
then
$\displaystyle{ u_1(x) \geq u_2(x)  \  {\rm for  \ } x \in I.   }$
 Scaling arguments are also used extensively. Given $r>0$ and $\eta >0$ we provide a rescaling of a given  smooth function $U$  as follows
 \begin{equation}
\label{scalinglaw}
\left\{
\begin{aligned}
U_{\eta, R}&= \eta \, U\left( \frac{U}{r}\right), {\rm \ and \  we  \ verify  \ that  \ } \\
 \mathcal L_\upmu (U_{\eta, r})(x)&=\frac{\eta}{r^2} \mathcal L_{\gamma} (U) (\frac{x}{r}) 
 \  {\rm where \ } 
 \gamma=\upmu\eta^{2(\theta-1)} r^2. 
 \end{aligned}
 \right. 
 \end{equation}
  In particular, if $\mathcal L_\upmu (U)=0$, then
%, for any $r>0$ and any $\lambda >0$,
 we have
$$\mathcal L_\upmu (r^{-\frac{1}{(\theta-1)}}U(\frac{\cdot}{r} ))=0 \ {\rm and } \ 
\mathcal L_{\lambda \upmu} (\lambda^{-\frac{1}{2(\theta-1)}}\,  U)=0, {\rm \ for \ any \ } r>0 {\rm \ and \ any \ }
\lambda >0.
$$
Notice also that  $U^*$ defined on $(0, +\infty)$ by $U^*(x)=[\sqrt{2}(\theta-1)x]^{-\frac{1}{\theta-1}}$ solves $\mathcal L(U^*)=0.$
%We first have

\begin{lemma}  
\label{veryordinary}
There exists a unique  smooth map $\Ump_r$  on  $(-r, r)$ such that  $\mathcal L(\Ump_r)=0$ and $\Ump_r(\pm r)=+\infty$, and  a unique solution   $\Umps_r$ such that  $\mathcal L(\, \Umps_r)=0$ and\,  $\Umps_r(\pm r)=\pm \infty$. Moreover, $\Ump_r$ is even, 
$\Umps_r$ is odd, and,  setting   $\Ump\equiv \Ump_1$ and $\Umps\equiv \Umps_1$,  we have 
\begin{equation}
\label{variance}
 \Ump_r (x)=r^{ -\frac{1}{\theta-1} }\Ump(\frac{x}{r})  \  {\rm and} \ 
 \Umps_r (x)=r^{ -\frac{1}{\theta-1} }\Umps(\frac{x}{r}).
\end{equation}
\end{lemma} 

\begin{proof}  For $n \in \N^*$,  we  construct on $(-r, r)$ a unique solution $\Ump_{r, n}$ that solves $\mathcal L(\Ump_{r, n})=0$ and $\Ump_{r, n}(\pm r)=n$,  minimizing the corresponding convex energy. 
By the comparison principle, $\Ump_{r, n}$ is non negative, increasing with $n$ and uniformly bounded on compact subsets of $(-r, r)$ in view of \eqref{univbound} below.   Hence a unique limit $\Ump_r$ exists, solution to $\mathcal L(\Ump_r)=0$. We  observe that 
$\Ump_{r, n}(\cdot) \geq U^*(r_n-\cdot)$, where $r_n=r+[\sqrt{2}(\theta-1)]^{-1}n^{-(\theta-1)}$, so that we obtain the required boundary conditions for $\Ump_r$. Uniqueness may again be established thanks to the comparison principle.
We construct similarly a unique solution $\Umps_{r, n}$ that solves $\mathcal L(\Umps_{r, n})=0$ and $\Ump_{r, n}(\pm r)=\pm n$. We notice  that $\Umps_{r, n}$ is odd, its restriction on $(0,r)$ non negative and increasing with $n$. Moreover, on some interval $(a, r)$, where $0<a<r$  does not depend on $n$,  we have
$\Ump_{r, n}(\cdot) \geq U^*(\tilde r_n-\cdot)$ where $\tilde r_n=r+[(\theta-1)]^{-1}n^{-(\theta-1)},$
 and we conclude as for the first assertion.  
 \end{proof} 

\begin{remark}
\label{molo}
 {\rm  Given $r>0$ and $\lambda >0$ we 
notice that the function $\Urp$ and $\Urps$ defined by
\begin{equation}
\label{romanurl}
%\left\{
%\begin{aligned}
\Urp (x)=\lambda^{-\frac{1}{2 (\theta-1)}} \Ump_r (x)
%=\lambda^{-\frac{1}{2( \theta-1)}}r^{-\frac{1}{\theta-1}} \Ump (\frac{x}{r}) \  
\       {\rm and \ } 
 \Urps(x)=\lambda^{-\frac{1}{2(\theta-1)}} \Umps_r (x) 
 %= \lambda^{-\frac{1}{2 (\theta-1)}}r^{-\frac{1}{\theta-1}} \Umps (\frac{x}{r}).
 %\end{aligned}
% \right.
 \end{equation}
solve $\mathcal L_\lambda (\Urp)=0$ and $\mathcal L_\lambda (\Urps)=0$  with the same boundary conditions as
$\Ump_r$ and $\Umps_r$.
}
\end{remark}

\medskip
\begin{lemma} 
\label{append}
% Let $r>0$ and $u$ be  a smooth function on  $(-r, r)$. 
 i) Assume that $\mathcal L(u)\leq 0$ on $(-r, r)$. Then, we have,  for $x \in (-r, r)$
\begin{equation}
\label{univbound}
  u(x)  \leq \left (\sqrt{2}  (\theta-1)\right) ^{-\frac{1}{\theta-1}} \left[ (x+r)^{-\frac{1}{\theta-1}}  +(x-r)^{-\frac{1}{\theta-1}} \right].
  \end{equation}
ii) Assume that $\mathcal L(u) \geq 0$  on $(-r, r)$ and that $u(-r)=u(r)=+\infty$. Then we have 
 \begin{equation}
 \label{pasuniv}
  u(x) \geq \left (\sqrt{2}  (\theta-1)\right) ^{-\frac{1}{\theta-1}}\max \{ (x+r)^{-\frac{1}{\theta-1}}, (r-x)^{-\frac{1}{\theta-1}}\}.  
  \end{equation}
% iii)  Assume that $\mathcal L(u)= 0$  on $(-r, r)$ and that $u(-r)=u(r)=+\infty$. Then we have
% \begin{equation}
% \label{univbound2}
% \vert \frac{d}{dx} u(x)\vert \leq  C\left[  (x+r)^{-\frac{\theta}{\theta-1}}+  (r-x)^{-\frac{\theta}{\theta-1}}\right]. 
% \end{equation}
\end{lemma}
\begin{proof} Set $\tilde U=U^* (\cdot +r) +  U^*(r-\cdot)$.   By subaddivity and translation invariance, we have
 $\mathcal L (\tilde U) \geq 0$ on $(-r, r)$ with $\tilde U( \pm r)=+\infty,$
 so that \eqref{univbound} follows from the comparison principle  \eqref{compare} with $u_1=\tilde U$ and $u_2=u$. Similarly, \eqref{pasuniv} follows from   \eqref{compare} with $u_1=u$ 	 and $u_2=U^*(\cdot +r)$ or $u_2=U^*(r-\cdot)$. 
 %If 
 %the assumptions of case iii) are met, then 
 %there exists some $a_0\in (-r, r)$ such that $\dot u(a_0)=0$. We may  therefore write
%$$ \vert \dot u (x) \vert=\vert  \int_{a_0}^x \ddot u(y)dy \vert = \vert  \int_{a_0}^x  2\theta u(y)^{2\theta-1} dy \vert, $$
%and the conclusion follows from estimate i). 
\end{proof}

Combining estimate  ii)  of Lemma \ref{append} with the scaling law of Lemma \ref{veryordinary} we are led to
%\begin{equation*}
%\label{deriver}
%\vert \frac{d}{dr} \Ump_r(x) \vert 
%\leq 
%C \left[ (r-x)^{-\frac{1}{\theta-1}}(r^{-1}+(r-x)^{-1})+ 
% (r+x)^{-\frac{1}{\theta-1}}(r^{-1}+(r+x)^{-1})\right],  {\rm \ so \ that \ } 
% \leq C\left[ (x-r)^{-\frac{\theta}{\theta-1}}r^{-\frac{1}{\theta-1}}+ 
%(x-r)^{-\frac{1}{\theta-1}}r^{-\frac{\theta}{\theta-1}}
%\right].
%\end{equation*}
% so that, for $\displaystyle {x\in (-\frac{7}{8}r, \frac{7}{8}r)}$, we have
 \begin{equation}
 \label{deriveur}
 \vert \frac{d}{dr} \Ump_r(x) \vert + \vert \frac{d}{dr} \Umps_r(x) \vert  \leq Cr^{-\frac{\theta}{\theta-1}}, {\rm \ for \ } x\in (-\frac{7}{8}r, \frac{7}{8}r). 
\end{equation}
%%%%%%%%%%%%%%%%%%%%%%%%%%%%%%%%%%%%%%%%%%%%%%%%%%%%%%%%%%%%%%%%%%%%%
%%%%%%%%%%%%%%%%%%%%%%%%%%
%%%%%%%%%%%%%%%%%%%%%%%%%%%%%%%%%%%%%%%%%%
\subsection {The discrepancy for $\mathcal L_\upmu$}
\label{discrete}
%%%%%%%%%%%%%%%%%%%%%%%%%%%%%%%%%%%%%%%%%%%%%%%%%%%%%%%%%%%%%%%%%%%%%
The discrepancy  $\Xi_{\upmu}$   for   $\mathcal L_\upmu$ relates to a given  function $u$ the function $\Xi_\upmu (u)$ defined by
\begin{equation}
\label{ximu}
\Xi_\upmu (u)=\frac{\dot u^2}{2}-\upmu {u}^{2 \theta}.
\end{equation}
This function is \emph{constant} if $u$ solves $\L_\upmu(u)=0$. We set  
$\Xi=\Xi_1$,
 \begin{equation}
 \label{defab}
  A_\theta \equiv \Xi (\, \Ump)=-(\Ump (0))^{2\theta}<0 \ {\rm and } \ \, 
  B_\theta\equiv \Xi (\, \Umps)=\frac { (\Umps (0))_x^2}{2}>0. 
\end{equation}
In view of  the  scaling relations \eqref{variance} and Remark \ref{molo}, we are hence  led to the identities
\begin{equation}
\label{variance2}
\left\{
\begin{aligned}
&\Xi_\lambda (\, \Urp)= \lambda^{-\frac{1}{ \theta-1)}}r^{-\frac{2 \theta}{\theta-1}} A_\theta= \lambda^{-\frac{1}{ \theta-1)}}r^{-(\omega+1)} A_\theta,\\ 
&\Xi_\lambda (\, \Urps)=\lambda^{-\frac{1}{ \theta-1}}r^{-\frac{2 \theta}{\theta-1}} B_\theta=\lambda^{-\frac{1}{ \theta-1}}r^{-(\omega+1)} B_\theta.
\end{aligned}
\right.
\end{equation}
%%%%%%%%%%%%%%%%%%%%%%%%%%%%%%%%%%%%%%%%%%%%%%%%%%%%%%%%%%%%%%%%%%%
%%%%%%%%%%%%%%%%%%%%%%%%%%%%%%%%%%%%%%%%%%%%%%%%%%%%%%%%%%%%%%%%%%%%%%%%%%%%%%%%%%
\subsection {The operator $\mathcal L^\eps$}
%%%%%%%%%%%%%%%%%%%%%%%%%%%%%%%%%%%%%%%%%%%%%%%%%%%%%%%%%%%%%%%%%%%%%%%%%%%%%%%%
%%%%%%%%%%%%%%
In this subsection, we consider more generally, for 
 given $\lambda>0$ the  operator $\Le$  given   by
$$\Le (U)= -\frac{d^2}{dx^2} U + 2\lambda f_\eps(U), $$
with $f_\eps$  defined in \eqref{feps}, and the solutions $\Umpep$,  $\Umpes$, and $\Umpeps$ of $\Le(U)=0$ on $(-r, r)$
with corresponding infinite boundary conditions, whose existence and uniqueness is proved as for Lemma \ref{veryordinary}.

\begin{lemma}
\label{ref}
 We have the estimates
$$ \vert\, \Umpep(x)\vert + \vert \, \Umpeps (x)\vert \leq   C\left(\lambda^2 (\theta-1)\right) ^{-\frac{1}{\theta-1}} \left[ (x+r)^{-\frac{1}{\theta-1}}  +(x-r)^{-\frac{1}{\theta-1}} \right].$$

\end{lemma}

 \begin{proof}  It follows from \eqref{labourage} that $\L_{\frac{3}{4} \lambda}(\Umpep)\leq 0$, so that, invoking  the comparison principle as well as the scaling law \eqref{scalinglaw} we deduce that 
 $\Umpep\leq (3\lambda/4)^{-2(\theta-1)} \Ump_r.$
 A similar estimate holds for $\Umpeps$ and the conclusion follows from Lemma \ref{append}.
  \end{proof}
We complete this appendix by comparing the solutions $\Ump_r$ and $\Umpep$,
as well as $\Umps_r$ and $\Umpeps.$

\begin{proposition}
\label{catamaran}  In the interval $(-\frac{7}{8}r, \frac{7}{8}r)$ we have the estimate
\begin{equation*}
\vert  \Umpep -\Ump_r \vert \leq C\eps^\frac{1}{\theta} r^{-\frac{2\theta-1}{\theta(\theta-1)}}.
\end{equation*}
\end{proposition}
\begin{proof}
Let $\eps< \delta < r/16$ to be fixed. It follows from Lemma \ref{ref} that, for $x \in (-r+\delta, r-\delta)$, we have 
\begin{equation*}
\label{avoine0}
%\begin{aligned}
0\leq \eps^{\frac{1}{\theta-1}}\Umpep (x) \leq C \left(\frac{\eps}{\delta}\right)^{\frac{1}{\theta-1}},
%&\leq , 
%\end{aligned}
\end{equation*}
 and therefore also 
 \begin{equation}
 \label{avoine}
\left \vert  \eps^{\frac{1}{\theta-1}}\Umpep (x) g \left ( \eps^{\frac{1}{\theta-1}}\Umpep (x)\right) \right \vert 
 \leq C \left(\frac{\eps}{\delta}\right)^{\frac{1}{\theta-1}}.
 \end{equation}
  %We set $\lambda_\eps^\pm=\lambda.(1\pm \eps^{\frac{1}{\theta}})$. 
It follows from \eqref{avoine} and the fact that $\Le(\Umpep)=0$, that $\mathcal L_{\lambda_\eps^-} (\Umpep) \leq 0$,
%and $\mathcal L_{\lambda_\eps^-} (\Umpep) \leq 0$, 
where
$\lambda_\eps^\pm=\lambda (1 \pm C(\frac{\eps}{\delta})^{\frac{1}{\theta-1}}).
$
On the other hand, by the 
scaling law \eqref{scalinglaw}, we have   
$$\mathcal L_{\lambda_\eps^-}\left ((\frac{\lambda_\eps^-}{\lambda})^{-\frac{1}{2(\theta-1)}}\, \Ump_{r-\delta}\right)=0.% \  {\rm and } \   
 %\mathcal L_{\lambda_\eps^-} \left((\lambda_\eps^-)^{-\frac{1}{2(\theta-1)}}\, \Ump_{\reintg}\right)=0.
$$
It follows from the comparison principle, since the second function is infinite on the boundary of the interval $[-r+\delta,r-\delta],$  that
$$
\Umpep \leq (\frac{\lambda_\eps^-}{\lambda})^{-\frac{1}{2(\theta-1)}}\, \Ump_{r-\delta} \quad\text{on } [-r+\delta,r-\delta].
$$
Integrating the inequality  \eqref{deriveur}  between 
$r-\delta$ and   $r$, we deduce that for  $x\in (-\frac{7}{8}r, \frac{7}{8}r)$, we have the inequality
\begin{equation}
\label{uneautre}
|\Ump_{r-\delta}(x)-\Ump_{r} (x)| \leq C\delta r^{-\frac{\theta}{\theta-1}}.
\end{equation}
On the other hand, it follows from estimate \eqref{univbound} of Lemma \ref{append} that for $x\in (-\frac{7}{8}r, \frac{7}{8}r)$,
\begin{equation}\label{eq:tong}
 |(\frac{\lambda_\eps^-}{\lambda})^{-\frac{1}{2(\theta-1)}}\, \Ump_{r-\delta}(x) - \Ump_{r-\delta}(x)| \leq C\left(\frac{\eps}{\delta}\right)^\frac{1}{\theta-1}r^{-\frac{1}{\theta-1}}.
\end{equation}
We optimize the outcome of \eqref{uneautre} and \eqref{eq:tong} choosing $\delta:=\eps^\frac{1}{\theta} r^\frac{\theta-1}{\theta}$
and we therefore obtain
$$
\Umpep \leq \Ump_r + C\eps^\frac{1}{\theta} r^{-\frac{2\theta-1}{\theta(\theta-1)}} \quad\text{on } (-\frac{7}{8}r, \frac{7}{8}r).
$$
The lower bound for $\Umpep$ is obtained in a similar way but reversing the role of super and subsolutions:  the function $ (\frac{\lambda_\eps^+}{\lambda})^{-\frac{1}{2(\theta-1)}}\, \Ump_{r+\delta}$ yields a subsolution for $\mathcal L^\eps$ on $[-r,r]$
whereas $\Umpep$ is a solution. The conclusion then follows.
\end{proof}

Similarly, we have:

\begin{proposition} In the interval $(-\frac{7}{8}r, \frac{7}{8}r)$ we have the estimate
\begin{equation*}
\vert  \Umpeps -\Umps_r \vert \leq C\eps^\frac{1}{\theta} r^{-\frac{2\theta-1}{\theta(\theta-1)}}.
\end{equation*} 
\end{proposition}
\begin{proof}
We only sketch the necessary adaptations since the argument is closely parallel to the proof of Proposition \ref{catamaran}. 
First, by the maximum principle $\Umpeps$ can only have negative maximae and positive minimae, so that actually $\Umpeps$ has no critical point and a single zero, which we call $a_\eps.$ Arguing as in Proposition \ref{catamaran}, one first  obtains 
$$\left(\lambda_\eps^-\right/\lambda)^{-\frac{1}{2(\theta-1)} }\Umps_{r-\delta-a_\eps}(\cdot-a_\eps) \geq \Umpeps \geq \left(\lambda_\eps^+\right/\lambda)^{-\frac{1}{2(\theta-1)} }\Umps_{r+\delta-a_\eps}(\cdot-a_\eps) \quad \text{on }[a_\eps,r-\delta],  $$
and
$$-\left(\lambda_\eps^-\right/\lambda)^{-\frac{1}{2(\theta-1)} }\Umps_{r+a_\eps-\delta}(\cdot-a_\eps) \geq - \Umpeps \geq -\left(\lambda_\eps^+\right/\lambda)^{-\frac{1}{2(\theta-1)} }\Umps_{r+a_\eps+\delta}(\cdot-a_\eps) \quad \text{on }[-r+\delta,a_\eps].$$
Since $\Umpeps$ is continuously differentiable at the point $a_\eps$ (indeed it solves $\mathcal L^\eps(\Umpeps) =0$), and since the derivative at zero of the function $\Umps_r$ is a deacreasing function of $r$, it first follows from the last two sets of inequalities that $|a_\eps| \leq \delta,$ and the conclusion then follows as in Proposition \ref{catamaran}. 
\end{proof}

%%%%%%%%%%%%%%%%%%%%%%%%%%%%%%%%%%%%%%%%%%%%%%%%%%%%%%%%%%%%%%%%%%
%%%%%%%%%%%%%%%%%%%%%%%%%%%%%%%%%%%%%%%%%%%%%%%%%%%%%%%%%%%%%%%%%%%%%%%%%%%%%%%%%
\setcounter{equation}{0} \setcounter{subsection}{0}
\setcounter{lemma}{0}\setcounter{proposition}{0}\setcounter{theorem}{0}
\setcounter{corollary}{0} \setcounter{remark}{0}
\renewcommand{\theequation}{B.\arabic{equation}}
\renewcommand{\thesubsection}{B.\arabic{subsection}}
\renewcommand{\thelemma}{B.\arabic{lem}}
\section*{Appendix B}
\renewcommand{\thesection}{B}
\numberwithin{equation}{section} \numberwithin{theorem}{section}
\numberwithin{lemma}{section} \numberwithin{proposition}{section}
\numberwithin{remark}{section} \numberwithin{corollary}{section}
%%%%%%%%%%%%%%%%%%%%%%%%%%%%%%%%%%%%%%%%%%%%%%%%%%%%%%%%%%%%%%%%%%%%%%%%%%%%%%%%%
\subsection { Some properties of the ordinary differential equation \eqref{tyrannosaure}}

This Appendix is devoted to properties of  the system of ordinary differential equations \eqref{tyrannosaure},  the result being somewhat parallel to the results in Section 2 of \cite{BS}.  
We assume that we are given  $\ell \in \N^*$, 
%a mapping  $\dagger$  from $J$  to $\{+, -\}$, where  $J=\{1, \cdots, \ell\}$,
  and a solution,  for $k \in J=\{1, \cdots, \ell\}$
   $t\mapsto a(t)=(a_1 (t), \cdots,a_\ell(t)) $  to the system 
      \begin{equation}  
     \label{tyrannosaures}
\E_{k}\frac{d}{ds} {a}_{k}(s)=  - \mathcal B_{(k-\frac1 2)} [(a_k(s)-a_{k-1} (s)]^{-(\omega+1)}+\mathcal B_{(k+\frac 12 )} 
[(a_{k+1}(s)-a_{k} (s)]^{-(\omega+1)},
\end{equation}
 where the numbers $\E_k$ are supposed to be positive, and are actually taken in \eqref{tyrannosaure} equal to 
 $\S_{i(k)}$, whereas the numbers $\mathcal B_{k+\frac 12}$, which may have positive or negative signs,  are taken in \eqref{tyrannosaure} to be equal to $\Gamma_{i(k-\frac 12)}$.   We also define $\E_{\rm min}=\min \{\E_i\}$ and  $\E_{\rm max}= \max\{\E_i\}$. 
   We consider a solution on its maximal interval of existence, which we call  $[0,T_{\max}]$. 
   An important feature of the equation \eqref{tyrannosaures} is its gradient flow structure.
   The behavior of this  system is indeed related   to the  function $F$  defined on $\R^{\ell}$ by 
   \begin{equation*}
   \label{gederowitch}
   \left\{
   \begin{aligned}
   F (U)&=\overset {\ell-1}{\underset {k=0}\sum} F_{k+\frac12}(U), 
   {\rm \  where, \ for \ } k=0, \cdots, \ell-1, {\rm \  and  \ }U=(u_1, \cdots, u_{\ell}), {\rm \ we  \  set  }  \\
     \
F_{k+\frac12}(U)&=-\omega^{-1}\mathcal B_{k+\frac 12} \left(u_{k+1}-{u_k}  \right)^{-\omega}
     \hbox{ with the convention that }u_0=-\infty.
     \end{aligned}
     \right.
\end{equation*} 
 If $u_1<u_2<\cdots <u_\ell$,  for $k=1, \dots, \ell$,
  then we have
\begin{equation}
\label{gedenocci}
%\begin{aligned}
\frac{\partial F}{\partial u_k}(U)= 
\mathcal B_{k-\frac 12} \left(u_{k}-{u_{k-1}}  \right)^{-(\omega+1)} -\mathcal B_{k+\frac 12} \left(u_{k+1}-{u_k}  \right)^{-(\omega+1)}, 
\end{equation}
 so that \eqref{tyrannosaure} writes
 $\displaystyle{\frac{d}{ds} {a}_{k}(s)=-\E_{k}^{-1}\frac{\partial F}{\partial u_k}(a(s)).}$
Hence, we have
 \begin{equation}
 \label{liapounov1}
 %\begin{aligned}
 \frac{d}{dt} F(a(t)) = \overset{\ell} {\underset {k=1} \sum}\frac{\partial F}{\partial u_k}(a(t))
\frac{d a_k}{dt} (t)=- \overset{\ell} {\underset {k=1} \sum} \E_k^{-1}\left(\frac{\partial F}{\partial u_k}(a(t) )\right)^2 
 \leq -\E_{\rm max }^{-1} \vert \nabla F(a(t)) \vert^2, 
%\end{aligned} 
\end{equation} 
 hence  $F$  decreases along the flow. 
%2}=\mp 1 \}$, and  with the convention that the quantity is equal to $+\infty$ in case the defining set is empty. 
   We also consider the  positive functionals defined by
     \begin{equation*}
  %   \label{repattractA}
     \Frep (U)=\underset{k\in J^+}\sum F_{k+\frac 12} (U), \  \Fat=- \underset{k\in J^-}\sum F_{k+\frac 12} (U), \,   {\rm \ for \ } 
     U=(u_1, \cdots, u_\ell), 
     \end{equation*}
   % which are both positive.
   where $J^\pm=\{ k\in \{0,\ell-1\}  {\rm  \ such \ that \ }  \epsilon_{k+\frac12}\equiv  {\rm sgn } (\mathcal B_{k+\frac 12})=\pm 1\}$.
 
 \begin{proposition}
\label{geducci}
  Let $a=(a_1, \cdots, a_\ell)$ be a solution to \eqref{tyrannosaures} on its maximal interval of existence $[0,T_{\rm max}]$. Then, we have, for any  time 
  $t \in [0, T_{\rm max}]$
  \begin{equation}
  \label{gedenimo}
  \left\{
  \begin{aligned}
  \left (\Frep(a(t))\right)^{-\frac{\omega+2}{\omega}} &\geq \left (\Frep(a(0))\right)^{-\frac{\omega+2}{\omega}} + 
\mathcal S_0 t, &\textswab d_{a}^+(t)  \geq
  \left( \mathcal S_1 t+ \mathcal S_2  \textswab d_{a}^+(0)^{\omega+2} \right)^{\frac{1}{\omega+2}}\\
\left (\Fat(a(t))\right)^{-\frac{\omega+2}{\omega}}&\leq \left (\Fat(a(0))\right)^{-\frac{\omega+2}{\omega}}-
\mathcal S_0 t, &\textswab d_{a}^-(t)  \leq  \left(\mathcal S_3\textswab d_{a}^-(0)^{\omega+2}-\mathcal S_4 t \right)^{\frac{1}{\omega+2}}, 
\end{aligned}
 \right.
\end{equation}
 where $\mathcal S_0>0, $ $\mathcal S_1>0$, $\mathcal S_2>0$, $\mathcal S_3>0$ and $\mathcal S_4>0$  depend only the coefficients of  \eqref{tyrannosaures}.
  \end{proposition}
  
   Since $\tda^-(s) \geq 0$, an immediate consequence of \eqref{gedenimo} is that 
   \begin{equation}
   \label{majortmax}
   T_{\rm max} \leq \frac{\mathcal S_3}{\mathcal S_4} \tda^-(0). 
\end{equation}                  
%Our next results provided on upperbound.
%\begin{lemma}
%\label{bounsong}
%We have
%$\displaystyle{T_{\rm max} \geq  S_5\tda^-(0)}.$
%\end{lemma

 %, the  completion of the proof of Proposition \ref{geducci} being presented in a separate subsection. 
Since the system \eqref{tyrannosaures} involves both attractive and repulsive forces, for the proof of Proposition \ref{geducci} it is convenient to  divide   the collection $\{a_1(t), a_2(t), \cdots,a_\ell(t)  \}$  into repulsive and attractive chains.   
    %  Consider more generally a positive integer $\ell \in \N^*$, and for $k \in \{0,\cdots, \ell-1\} $,  a number $\epsilon_{k+\frac 12}$ which belongs to $\{+1, -1\}$. Set $J=\{1, \cdots, \ell\}$.
     We say that a subset  $A$ of $J$ is a \emph {chain} if $A=\{k, k+1, k+2, \cdots , k+m\}$ is an ordered subset of  $m$ consecutive elements  in $J$, with $m\geq 1$.

 \begin{definition} 
 \label{commeundef}
 %Let $A=\{k, k+1, k+2, \cdots , k+m\}$ be an ordered subset of $m$ consecutive elements in $J$, with $m\geq 1$.\\
 %\indent
   A chain $A$ is said to be 	 \emph{repulsive}  (resp. \emph{attractive}) if and  only if $\epsilon_{j+\frac 12}=-1$ (resp. $+1$)  for $j=k, \cdots, k+m$. 
 %given  two elements  $i_1$ and $i_2$ in $J$, we have if $\dagger_{i_1}=\dagger_{i_2}.$
   It is said to be  a\emph{ maximal repulsive chain} (resp. \emph{maximal attractive chain}), if  there does exists any repulsive (resp.  attractive) chain which contains $A$ strictly.  
 %   The chain $A$ is said to be  \emph{attractive} if and  only  $\epsilon_{j+\frac 12}=+1$ for $j=k, \cdots, k+m$. 
 %It is said to be  a\emph{ maximal attractive   chain}, if  there does exists an attractive  chain which contains $A$ strictly.
 \end{definition}  
 %  This definition  strongly depends on the map $\epsilon$: in the contexte of equation \eqref{tyrannosaure},   we choose $\epsilon$ given by \eqref{thalys}.
 It follows from  our definition that  repulsive or attractive chain contain at least two elements.  
 %For a given map $\epsilon$,
   We  may decompose $J$, in increasing order, as
 \begin{equation}
 \label{gedede}
 J=B_0\cup A_1\cup  B_1\cup A_2\cup B_2\cup \cdots \cup B_{p-1} \cup A_p\cup B_{p}, 
 \end{equation}
 where  the   chains $A_i$ are maximal repulsive chains, the chains $ B_i$ are maximal attractive for $i=1,\dots, p-1$ , and the sets  $B_0$ and $B_p$ being possibly empty or maximal attractive chains.   Moreover  for $i=1\ \cdots, p$ the sets
    $A_i\cap B_i$, and  $B_i \cap A_{i+1}$
  contain one element.

%%%%%%%%%%%%%%%%%%%%%%%%%%%%%%%%%%%%%%%%%%%%%%%%%%%%%%%%%%  
   \subsection{ Maximal repulsive chains}
   \label{repulse}
    %%%%%%%%%%%%%%%%%%%%%%%%%%%%%%%%%%%%%%%%%%%%%%%%%%%%%%%%%%  
   In this subsection,  we restrict ourselves to the study  of the behavior of a \emph{maximal repulsive chain} $A=\{j, j+1, ...j+m \}$ of $m+1$ consecutive  points, $m\leq \ell-2$ \emph{within the general system \eqref{tyrannosaures}}. Setting, for $k =0, \cdots, m$,  $\tu_k(\cdot)=a_{k+j}(\cdot)$, we are led to study  $\tU(\cdot)=(\tu_0(\cdot),\tu_1(\cdot), \cdots, \tu_{m}(\cdot))$. 
 %     We restrict the functional $F$ to the maximal chain, setting,   for 
%$U=(u_0, u_1, \cdots, u_{m})\in \R^{m+1}$
%$\displaystyle{\tilde F_{\rm rep}(U)=\overset {m-1}{\underset {k=0}\sum} F_{k+\frac12}(U), }$
Since $B_{k+\frac 12} <0$ in the repulsive case
%, and since
%, we obtain  $\Frep >0$.
% $a$ satisfies \eqref{tyrannosaures}, 
the chain 
 $\tU$  is moved through   a system of  $m-1$ ode's,
 \begin{equation}
 \label{gedeous0}
 \E_k  \frac{d}{ds}  \tu_k(s)=-\vert \mathcal B_{(k-\frac1 2)}  \vert [(\tu_k(s)-\tu_{k-1} (s)]^{-(\omega+1)}+\vert \mathcal B_{(k+\frac 12 )}\vert [(\tu_{k+1}(s)-\tu_{k} (s)]^{-(\omega+1)}  
    \end{equation}
 for $k=1, \cdots, m-1$,  whereas  the end points satisfy two differential inequalities
  \begin{equation}
  \label{gedeous}
  %\left \{
%  \begin {aligned}
  \ \frac{d}{ds}  \tu_{m}(s)\geq -\E_{m}^{-1}\frac{\partial  \tFrep}{\partial u_m}(\tu_{m}(s)),   \ \ 
    \frac{d}{ds}  \tu_{0}(s) \leq  -\E_{0}^{-1}\frac{\partial \tFrep}{\partial u_0}(a(s)),
 % \end{aligned}
 % \right.
  \end{equation}
  where we have set
  $\displaystyle{\tilde F_{\rm rep}(U)=\overset {m-1}{\underset {k=0}\sum} F_{k+\frac12}(U).}$
   We  assume that at initial time we have  
\begin{equation}
\label{gedeonseck}
\tu_0(0)<\tu_1(0)< \cdots < \tu_{m}(0).
\end{equation}
Set
 $\displaystyle{ \textswab d_ {\mathfrak u}  (t)=\min \{ \tu_{k+1}(t)- \tu_k(t), \  \ k=0,\cdots, m-1 \} .}$
 We prove in this subsection: 
\begin{proposition}
\label{grasdouble} Assume that the function $\tU$  satisfies the system \eqref{gedeous0} and   \eqref{gedeous} on $[0, T_{\rm max}]$, and that \eqref{gedeonseck} hold. Then, we have, for any $t\in [0, T_{\rm max}]$
\begin{equation}
\label{vendee}
\left (\tFrep(\tU(t))\right)^{-\frac{\omega+2}{\omega}}-\left (\tFrep(\tU(0))\right)^{-\frac{\omega+2}{\omega}} \geq 
 \frac {\omega+1}{ 4 \omega} \E_{\rm max }^{-1}  \left(\omega \mathcal B_{\rm max}\right)^{-\frac{2(\omega+1)}{\omega} }t, \ {\rm \ so \ that \ }
\end{equation}
\begin{equation}
\label{gedissis0}
 \textswab d_{\mathfrak   u}(t)  \geq
  \left( \mathcal S_1 t+ \mathcal S_2  \textswab d_{\mathfrak  u}(0)^{\omega+2} \right)^{\frac{1}{\omega+2}},
 \end{equation}
 where $\mathcal S_1>0$ and $\mathcal S_2>0$  depend only on the coefficients of the equation \eqref{tyrannosaures}. 
 \end{proposition} 
The proof relies on several elementary observations.
 \begin{lemma}
    \label{gedescu} Let $\tU$ be a solution to \eqref{gedeous0},  \eqref{gedeous} and \eqref{gedeonseck}. Then, we have, 
    \begin{equation}
    \label{gedescu2}
   \frac{d}{dt} \tFrep (\tU(t))\leq -\E_{\rm max }^{-1} \left \vert \nabla \tFrep (\tU(t))\right \vert^2, {\ for \ every \ t \ }  \in [0, T_{\rm max}].
\end{equation}
%n particular  $\displaystyle{
%F(\tu(t))\leq\eps  \left [ \mathcal  S_4 \eps^{-2} t+ \frac{\eps }{F(\tu(0)) } \right]^{-1} \leq F(\tu(0))}$.
    \end{lemma}
The proof is similar to \eqref{liapounov1} and we omit it. For $U=(u_0, u_1, \cdots, u_{m})\in \R^{m+1}$ with $u_0<\cdots <u_{m}$ set
$\displaystyle{\rho_{\rm min} (U)=\inf\{\vert u_{k+1}-u_k\vert, \ k=0, \cdots, m-1\} }$
and 
$\displaystyle{\B_{\rm min}= \inf\{ \vert \B_{k+\frac 12} \vert \} }$,
$\displaystyle{  \B_{\rm max}= \sup\{ \vert \B_{k+\frac 12} \vert \}.}$
 \begin{lemma}
   \label{gedinacci} 
   Let  $U=(u_0, \cdots, u_{m})$ be  such that  $u_0<u_1< \cdots <u_m$.  We have, 
     \begin{equation}
\label{gedeson}
  \B_{\rm min }\omega^{-1} \rho_{\rm min}(U)^{-\omega}\leq
\tFrep(U)
\leq \B_{\rm max} (m+1)(\omega)^{-1} \rho_{\rm min}(U)^{-\omega}.
\end{equation}
 \begin{equation}
 \label{gedeson2}
  \vert \nabla  \tFrep(U) \vert \leq 
   (m+2)     \mathcal B_{\rm max}
((\omega-1)  \mathcal B_{\rm min})^{-\frac{\omega+1}{\omega} } (\tFrep(U))^{\frac{\omega+1}{\omega}}, 
 \end{equation}
\begin{equation}
 \label{regedere}
    \vert \frac{\partial  \tFrep(U)}{\partial u_k} \vert \geq  \frac 12  \left(\omega  \mathcal B_{\rm max}\right)^{-\frac{\omega+1}{\omega} }\left(\tFrep(U)\right)^{\frac{\omega+1}{\omega}}, \hbox{ for every  }  k=0, \cdots, m.
    \end{equation}
   \end{lemma} 
    \begin{proof}  Inequalities \eqref{gedeson} and \eqref{gedeson2} are direct consequences of the definition of $\tFrep$.  We turn therefore to estimate \eqref{regedere}.  In view of formula \eqref{gedenocci}, the cases $k=0$ and $k=m+1$ are straightforward.   Next, let $k=1,\cdots m$ and  set
    $\displaystyle{ T_{k+\frac 12}=\mathcal B_{k+\frac 12} \left(u_{k+1}-{u_k}\right)^{-(\omega+1)} }.$
   We distinguish two cases:\\ 
   \noindent
    Case 1: $T_{k-\frac 12}  \leq  \frac12    T_{k +\frac 12} $. Then, we have, in view of \eqref{gedenocci}
    $\displaystyle{ T_{k+\frac 12}  \leq 2\vert  \frac{\partial F}{\partial u_k}(U)\vert
    \leq 2 \vert \nabla F(U) \vert,  }$
 and we are done.    \\
     Case 2 :  $T_{k-\frac 12}  \geq  \frac12    T_{k +\frac 12} $. In that case, we repeat the argument with  $k$ replaced by $k-1.$ Then either $T_{k-\frac  32}  \leq  \frac12    T_{k -\frac 12}$, which yields as above
     $\displaystyle{ T_{k-\frac 1 2}    \leq 2 \vert \nabla F(U) \vert,  }$
      so that 
 we are done, or $T_{k-\frac 3 2}  \geq  \frac  1 2    T_{k +\frac 12}$, and we go on.  Since we have to stop at $k=0$, this leads to the desired inequality \eqref{regedere}. 
     \end{proof}
  
\begin{proof}[Proof of Proposition \ref  {grasdouble}] 
Combining \eqref{gedescu2} with \eqref{regedere} we are led to 
%the differential inequality
    \begin{equation*}
    \label{gedescu3}
   \frac{d}{dt} \tFrep (\tU(t))\leq -\frac 1 4\E_{\rm max }^{-1}  \left(\omega  \mathcal B_{\rm max}\right)^{-\frac{2(\omega+1)}{\omega} }\left(\tFrep(\tu(t))\right)^{\frac{2(\omega+1)}{\omega}}. 
\end{equation*}
Integrating this differential equation, we obtain \eqref{vendee}.
 %which is the first inequality in \ref{gedenimo}.
Combining the last inequality of  Lemma \ref{gedescu} with  inequality \eqref{gedeson}, inequality \eqref{gedissis0} follows.
   \end{proof}

     %%%%%%%%%%%%%%%%%%%%%%%%%%%%%%%%%%%%%%%%%%%%%%%%%%%%%%%%%%  
   \subsection{ Maximal attractive  chains}
   \label{attract}
    %%%%%%%%%%%%%%%%%%%%%%%%%%%%%%%%%%%%%%%%%%%%%%%%%%%%%%%%%%  
  Maximal attractive chains $B=\{j, j+1, ...j+m \}$, with  $m\leq \ell-1$ within the general system \eqref{tyrannosaures}, are handled similarly.  Defining $\tU$ as above, the function $\tU$ still  satisfies \eqref{gedeous0}, but the inequalities \eqref{gedeous} are now replaced by
     \begin{equation}
  \label{gedeous2}
%  \left \{
  %\begin aligned}
  \frac{d}{ds}  \tu_{m}(s)\leq -\E_{m}^{-1}\frac{\partial \tFat}{\partial u_m}(\tu_{m}(s)),   \ \ 
  \frac{d}{ds}  \tu_{0}(s) \geq  -\E_{0}^{-1}\frac{\partial \tFat}{\partial u_0}(a(s)).
 % \frac{d}{ds}  \tu_{m}(s)\leq 0, \
   %  \eps \frac{d}{ds}  \tu_{0}(s) \geq 0.
 % \end{aligned}
  %\right.
  \end{equation}
%wherehe behavior of the chain $B$ is now still related  to the functional
$\tFat (U)$ is  defined by $\tFat=-\tFrep$, so that we have in the attractive case  $\tFat \geq 0$. Up to a change of sign, the function $\tFat$ verifies the properties \eqref{gedeson}, \eqref{gedeson2} and \eqref{regedere} stated in Proposition \ref{gedinacci}.
However the differential inequality \eqref{gedescu2} is now turned into
   \begin{equation}
    \label{gediroff}
  \frac{d}{dt} \tFat (\tU(t))\geq \E_{\rm max}^{-1}\left \vert \nabla \tFat (\tU(t))\right \vert^2
    \geq \mathcal C  \tFat(\tU(t))^{\frac{2(\omega+1)}{\omega}},
    \end{equation}
   where the last inequality follows from \eqref{regedere} and  where $\mathcal C$ is some constant depending only on the coefficients in \eqref{tyrannosaures}.
Integrating \eqref{gediroff}  and invoking once more \eqref{gedeson}, we obtain 
\begin{lemma}
\label{intheboite2}
 Assume  that  $\tU$  satisfies the system \eqref{gedeous0} and   \eqref{gedeous2} on $[0, T_{\rm max}]$   with \eqref{gedeonseck}.  Then for constants $\mathcal S_3>0$ and $\mathcal S_4>0$  depending only on the coefficients of  \eqref{tyrannosaures}, we have 
$$ \textswab d_{\mathfrak u} (t) \leq \left(\mathcal S_3  \textswab d_{\mathfrak u}(0)^{\omega+2}-\mathcal S_4 t \right)^{\frac{1}{\omega+2}}.  $$
 %where $\mathcal S_3>0$ and $\mathcal S_4>0$  depend only on the coefficients of  \eqref{tyrannosaures}. 
\end{lemma}

%%%%%%%%%%%%%%%%%%%%%%%%%%%%%%%%%%%%%%%%%%%%%%%%%%%%%%%%
\subsection{Proof of Proposition \ref{geducci} completed}
Inequalities \eqref{gedenimo} of Proposition  \ref{geducci} follow immediately from  Proposition \ref{grasdouble} and Lemma \ref{intheboite2} applied to each separate maximal chain provided by the decomposition \eqref{gedede}.
 \qed

\end{document}